\documentclass[reqno,twoside,12pt,a4paper]{amsart}
\usepackage{layout}
\usepackage{latexsym}
\usepackage{amsmath}
\usepackage{amssymb}
\usepackage{ascmac}
\usepackage{mathrsfs}
\usepackage{tikz}
\usetikzlibrary{intersections, calc, arrows}
\usepackage{amsmath}
\usepackage{cases} 
\usepackage{caption}
\usepackage{cite}
\usepackage{color}
\definecolor{viola}{rgb}{0.3,0,0.7}
\definecolor{ciclamino}{rgb}{0.5,0,0.5}
\definecolor{blu}{rgb}{0,0,0.7}
\definecolor{rosso}{rgb}{0.85,0,0}
\def\pier #1{#1}
\def\takeshi #1{#1}

\renewcommand{\d}{\, {\mathrm d}}
\newcommand{\dx}{\, {\mathrm d} x}
\newcommand{\dg}{\, {\mathrm d} \Gamma}
\newcommand{\ds}{\, {\mathrm d} s}
\newcommand{\dt}{\, {\mathrm d} t}

\def\ke{{\varepsilon,\kappa}}
\def\Gke{{\Gamma,\varepsilon,\kappa}}
\def\kek{{\varepsilon_k,\kappa_k}}
\def\Gkek{{\Gamma,\varepsilon_k,\kappa_k}}
\numberwithin{equation}{section}

\newtheorem{theorem}{Theorem}[section]
\newtheorem{corollary}[theorem]{Corollary}
\newtheorem{lemma}[theorem]{Lemma}
\newtheorem{proposition}[theorem]{Proposition}

\renewenvironment{proof}{\noindent {\bf Proof.}}{\hfill $\Box$}

\allowdisplaybreaks[4]

\title[Asymptotic analysis of problems with dynamic boundary conditions]
{Asymptotic analysis of the Allen--Cahn equation with\\[1mm] dynamic boundary conditions of Cahn--Hilliard type}

\author[P.\ Colli]{Pierluigi Colli}
\address{Pierluigi Colli: Dipartimento di Matematica ``F. Casorati'', Universit\`a di Pavia
 and Research Associate at the IMATI -- C.N.R. Pavia, via Ferrata 5, 27100 Pavia, Italy}
\email{pierluigi.colli@unipv.it}

\author[T.\ Fukao]{Takeshi Fukao}
\address{Takeshi Fukao: Faculty of Advanced Science and Technology, Ryukoku University, 
1-5 Yokotani, Seta Oe-cho, Otsu-shi, Shiga 520-2194, Japan}
\email{fukao@math.ryukoku.ac.jp}

\dedicatory{}
\begin{document}
%\layout

\thispagestyle{empty}

\begin{abstract} 
\pier{Problems for partial differential equations coupled with dynamic boundary conditions can be viewed as a type of transmission problem between the bulk and its boundary. For the heat equation and the Allen--Cahn equation, various forms of such problems with dynamic boundary conditions are studied in this paper.
In the case of the Cahn--Hilliard equation in the bulk, several models have been proposed in which the boundary equations and conditions differ. Recently, the vanishing surface diffusion limit has been investigated in more than one of these models. In such settings, the resulting dynamic boundary equation typically takes the form of a forward-backward parabolic equation.
In this paper, we focus on a different model, in which the Allen--Cahn equation governs the bulk dynamics, while the boundary condition is of Cahn--Hilliard type. We analyze the asymptotic behavior of the system, including the well-posedness of the limiting problems and corresponding error estimates for the differences between solutions. These aspects are discussed for three types of limiting systems.}
\smallskip

\noindent {\sc Key words:} \pier{Allen--Cahn equation, dynamic boundary conditions of Cahn--Hilliard type, asymptotic analyses, well-posedness, rates of convergence.}
\smallskip

\noindent {\sc Mathematics Subject Classification 2020:} \pier{35K61, 35K35, 35D30, 58J35, 74N20.}
\end{abstract}

\maketitle

%%%%% Section 1. %%%%%
\section{Introduction}
\label{intro}
\setcounter{equation}{0}

In the study of time-dependent partial differential equations (PDEs), we distinguish between two types of processes: the \emph{forward process}, which progresses in the positive time direction starting from \( t=0 $, and the \emph{backward process}, which corresponds to evolution in the negative time direction. To illustrate this distinction, consider the classical heat equation. Let \( T > 0 $ and \( \Omega \subset \mathbb{R}^d $ be a bounded domain with smooth boundary \( \Gamma := \partial \Omega $, where \( d \in \mathbb{N} $, \( d \geq 2 $. The heat equation takes the form
\begin{equation*}
	\partial_t u - \Delta u = f \quad \text{in } Q := \Omega \times (0, T),
\end{equation*}
supplemented with suitable boundary and initial conditions:
\begin{alignat*}{2}
	B u &= f_\Gamma && \quad \text{on } \Sigma := \Gamma \times (0,T), \\
	u(0) &= u_0 && \quad \text{in } \Omega,
\end{alignat*}
where \( f: Q \to \mathbb{R} $, \( f_\Gamma: \Sigma \to \mathbb{R} $, and \( u_0: \Omega \to \mathbb{R} $ are given, and \( B $ denotes a boundary operator.
It is well known that the sign of the Laplacian is crucial: reversing the sign renders the problem ill-posed in general. To clarify this, define the transformation \takeshi{$ U(x, t) := u(x,T - t) $}. Then the backward heat equation
\begin{equation*}
	\partial_t u + \Delta u = f \quad \text{in } Q ,
\end{equation*}
with boundary and initial conditions as above, can be reformulated  in terms of $U$ as
\begin{alignat*}{2}
	\partial_t U - \Delta U &= -f && \quad \text{in } Q, \\
	B U &= f_\Gamma && \quad \text{on } \Sigma, \\
	U(T) &= u_0 && \quad \text{in } \Omega.
\end{alignat*}
Here, the former initial condition is replaced with a final condition at time \( T $. This formulation demonstrates that the backward heat equation requires high regularity to data in order to obtain solutions, due to the inherent smoothing effect of the heat operator.

\smallskip

One of the goals of this paper is to investigate the well-posedness of PDE systems that exhibit a backward-like structure in the boundary condition. Specifically, we consider dynamic boundary conditions with a positive surface diffusion term. We focus on boundary conditions of the form
\begin{equation*}
	\partial_t u + \sum_{k=0}^{4} B_k u = f_\Gamma \quad \text{on } \Sigma,
\end{equation*}
where \( B_k $ represents a differential operator of order \( k $. These boundary conditions, which include time derivatives, are referred to as \emph{dynamic boundary conditions}.
In particular, we study a dynamic boundary condition of Cahn--Hilliard type:
\begin{equation*}
	\partial_t u - \Delta_\Gamma \left( \partial_{\boldsymbol{\nu}} u 
	- \kappa \Delta_\Gamma u + \mathcal{W}_\Gamma'(u) - f_\Gamma \right) = 0 
	\quad \text{on } \Sigma,
\end{equation*}
where \( \Delta_\Gamma $ denotes the Laplace--Beltrami operator on the boundary \( \Gamma $ (cf. \cite{Gri09, Heb96}), and \( \partial_{\boldsymbol{\nu}} $ is the outward normal derivative. The function \( \mathcal{W}_\Gamma' $ is the derivative of a double-well potential \( \mathcal{W}_\Gamma $, with typical examples including \( \mathcal{W}_\Gamma'(r) = r^3 - r $ or \( \mathcal{W}_\Gamma'(r) = -r $. In our analysis, we study the asymptotic behavior as the surface diffusion parameter \( \kappa \to 0 $, which corresponds to setting \( B_4 = 0 $ and leads to a boundary condition of forward-backward type. Our main result demonstrates that, despite the apparent ill-posedness of such a formulation, the problem remains well-posed in a weak sense due to the leading third-order term \( B_3 = -\Delta_\Gamma \partial_{\boldsymbol{\nu}} $.

\smallskip

This vanishing surface diffusion limit has been investigated in prior works, including \cite{CF20b, CFS22, CFW20, LW19, Sca19}. In \cite{CFS22}, asymptotic analysis was carried out starting from a Cahn--Hilliard equation with a dynamic boundary condition of the same type, as introduced in \cite{Gal06, GMS11}, leading to third-order boundary dynamics. Extensions of this idea have been pursued in \cite{CFS22b} and \cite{LW24}, based on models from \cite{CFW20, LW19, KLLM21}. All these works involve fourth-order PDEs in the bulk (i.e., the Cahn--Hilliard equation).
Here, we study a model involving the second-order Allen--Cahn equation in the bulk, paired with a Cahn--Hilliard-type dynamic boundary condition. While Allen--Cahn equations with dynamic or Wentzell-type boundary conditions have been investigated before (see, e.g., \cite{CC13, CF15a, GG08, FGGR02}), the combination considered here is, to our knowledge, novel.

\smallskip

The central goal of this paper is to clarify the relationships among four types of problems, beginning with the following system (cf. \cite{CF20}). Let \( \varepsilon, \kappa > 0 $ be asymptotic parameters and consider:
\begin{align*}
{(\text{\rm Laplace equation})} & \
 \begin{cases}
  - \varepsilon \Delta \mu =0 
  &  {\rm  in ~ } Q, \\
     \mu_{|_\Gamma} =\mu_\Gamma
  & {\rm on~} \Sigma, 
 \end{cases}
  \\ 
  {(\text{\rm Allen--Cahn equation})} & \
 \begin{cases}
   \partial _t u - \Delta u + \mathcal{W}'(u) = f
  &  {\rm in~} Q,\\
   u _{|_\Gamma}= u_\Gamma 
  & {\rm on~} \Sigma,\\ 
    u(0) = u_0 &{\rm in~} \Omega,\\
  \end{cases}
  \\
 {(\text{\rm Cahn--Hilliard boundary dynamics})} & \
 \begin{cases}
  \partial _t u_\Gamma + \varepsilon \partial_{\boldsymbol{\nu}} \mu - \Delta_\Gamma \mu_\Gamma = 0
  & {\rm on~} \Sigma, \\ 
  \mu_\Gamma = \partial_{\boldsymbol{\nu}} u - \kappa \Delta _\Gamma u_\Gamma 
  + {\mathcal W}'_\Gamma(u_\Gamma) - f_\Gamma
    & {\rm on~} \Sigma, \\
  u_\Gamma (0) = u_{0\Gamma} & {\rm on~} \Gamma,
  \end{cases}
\end{align*}
Here, $ \mu|_\Gamma $ and $ u|_\Gamma $ denote the traces of $ \mu $ and $ u $ on $ \Gamma $, respectively; $ u_{0\Gamma} : \Gamma \to \mathbb{R} $ is the boundary initial data. Different double-well potentials may be used in the bulk and on the boundary. For clarity, we distinguish between bulk variables $ u, \mu $ and boundary variables $ u_\Gamma, \mu_\Gamma $.

\smallskip

We study three asymptotic regimes: $ \kappa \to 0 $, $ \varepsilon \to 0 $, and the simultaneous limit $ \varepsilon, \kappa \to 0 $. In particular, we note that, when $ \varepsilon = 0 $, the system reduces to the Allen--Cahn equation with a dynamic boundary condition of Cahn--Hilliard type.
The analysis carried out in this paper relies on uniform \emph{a priori} estimates and rigorous 
limiting procedures. The structure of the work is outlined as follows.
Section~\ref{main} introduces the basic functional framework and provides a detailed discussion of 
the target problems. In Section~\ref{asy-ana}, we begin with Subsection~3.1, where we derive uniform 
estimates for the general problem, focusing initially on the limit $\kappa \to 0$ while keeping $
\varepsilon > 0$ fixed. These estimates are inspired by techniques developed in earlier studies of 
the Cahn--Hilliard equation with dynamic boundary conditions, such as~\cite{CF15}.
Subsection~3.2 is devoted to the convergence analysis as $\kappa \to 0$, employing weak formulations and demiclosedness arguments as introduced in~\cite{CF20b, Sca19} and further elaborated 
in~\cite{CFS22, CFS22b}. Subsequently, Subsections~3.3 and~3.4 address the limits $\varepsilon \to 0$ and the simultaneous limit $\varepsilon, \kappa \to 0$, respectively, using similar analytical techniques.
Section~\ref{cont-dep} is concerned with continuous dependence results for the limiting problems, which in turn yield uniqueness of the corresponding solutions. Finally, Section~5 provides error 
estimates for all three limiting regimes, based on higher-order regularity results. In each case, we 
establish convergence rates of order $1/2$ with respect to appropriate norms measuring the differences between solutions.

\section{Functional setting and problem statement}
\label{main}

In this section, \pier{we begin by introducing the functional spaces that will be used throughout the analysis. We also recall several useful tools, including a number of classical inequalities. Subsequently, we review a relevant existence result and provide a discussion of the three limiting problems that will be investigated in the later sections.}

\subsection{Notation and useful tools}

Let $T>0$ and $\Omega \subset \mathbb{R}^d$ be a bounded domain with smooth boundary $\Gamma:=\partial \Omega$, $d \in \mathbb{N}$ with $d \ge 2$. 
Hereafter we use the following notation for function spaces: 
$H:=L^2(\Omega)$, 
$V:=H^1(\Omega)$, 
$W:=H^2(\Omega)$, 
$H_\Gamma:=L^2(\Gamma)$, 
$Z_\Gamma:=H^{1/2}(\Gamma)$, 
$V_\Gamma:=H^1(\Gamma)$, and 
$W_\Gamma:=H^2(\Gamma)$. We denote the norm of a Hilbert space $X$ by 
$\|\cdot\|_X$. Moreover, $X'$ stands for the dual space of $X$ with their duality pair 
$\langle \cdot, \cdot \rangle_{X',X}$. \pier{By identification of $H$ with its dual space,} we have the Gelfand triple 
$V \hookrightarrow \hookrightarrow H  \hookrightarrow \hookrightarrow V'$, where the notation 
``$\hookrightarrow \hookrightarrow$'' stands for \pier{a dense and compact embedding}.
\smallskip 

Next, \pier{for $s >1/2$ we recall the standard trace operator} 
$\gamma_0 : H^s(\Omega) \to H^{s-1/2}(\Gamma)$ 
(see e.g., \cite{BG87, GR86, LM70, Nec67}), that is, $\gamma_0 v =v_{|_\Gamma}$ 
 for all $v \in C^{\infty}(\overline{\Omega}) \cap H^s(\Omega)$. Moreover, there exists a positive constant $C_{\rm tr}$ such that 
\begin{equation}
	\|\gamma _0 v \|_{H^{s-1/2}(\Gamma)} \le C_{\rm tr} \|v\|_{H^s(\Omega)}
		\quad 
	{\rm for~all~} v \in H^s(\Omega).
	\label{statre1}
\end{equation} 
Hereafter, 
the two notations of the trace 
$\gamma_0 v$ and $v_{|_\Gamma}$ are used interchangeably and 
 if there is no confusion. 
Analogously, for $s>3/2$ \pier{the operator
$\gamma_1 : H^s(\Omega) \to H^{s-3/2}(\Gamma)$, defined by
$\gamma_1 v =(\partial _{\boldsymbol{\nu}} v)_{|_\Gamma} $
 for all $v \in C^{\infty}(\overline{\Omega}) \cap H^s(\Omega)$, fulfills}  
\begin{equation} 
	\|\gamma _1 v \|_{H^{s-3/2}(\Gamma)} \le C_{\rm tr} \|v\|_{H^s(\Omega)}
	 \quad 
	{\rm for~all~} v \in H^s(\Omega),
	\label{statre2}
\end{equation}
where we use \pier{the same notation $C_{\rm tr}$ for the positive constant in \eqref{statre2}, for simplicity.} 
Following the convention, we \pier{adopt the notation $\partial_{\boldsymbol {\nu}} v$
for} the trace $\gamma_1 v $.  
\smallskip

\pier{In the sequel, we follow the convention that the symbol $C$ denotes a generic positive constant that may depend only on $\Omega$, $T$, and the data 
of the problems under consideration. The value of this constant may vary from one occurrence to another, and even within a single formula. Furthermore, we 
use the notation $C_\delta$ to indicate a positive constant that may also depend on the parameter $\delta$.}

\smallskip

Let $s \in (0,1)$ and recall the compact embedding 
$V \hookrightarrow \hookrightarrow H^{s}(\Omega)$. Applying the 
Ehrling--Lions lemma~\pier{(see, e.g., \cite[p.~58]{Lio69}) yields that 
for each $\delta> 0$ there exists a constant $C_\delta>0$} such that 
\begin{gather*}
	\|v\|_{H^s(\Omega)}^2 \le \delta \|v\|_V^2 + C_\delta \|v\|_H^2 
	\quad {\rm for~all~} v \in V.
\end{gather*}
Therefore, if $s \in (1/2,1)$, \pier{from \eqref{statre1}  it follows that} 
\begin{align}
	 \| \gamma_0 v \|_{H_\Gamma}^2 & \le \pier{C} \|\gamma _0 v\|_{H^{s-1/2}(\Gamma)} ^2
    \le \pier{C} \| v \|_{H^{s}(\Omega)}^2 
	\notag \\
	& \le \delta \|v\|_V^2 + 
	C_\delta \|v\|_H^2 
	\quad {\rm for~all~} v \in V , 
	\label{comp2}
\end{align}
for all $\delta >0$. On the other hand, \pier{the elliptic regularity theory (see~\cite[Theorem 3.2, p.~1.79]{BG87} or~\cite[Section 7.3, pp.~187--190]{LM70}) allows us to deduce that
\begin{align}
&\|v\|_{H^{3/2}(\Omega)} \le 
	C_{\rm e} \bigl( \|\Delta v\|_{H} + \|\gamma _0 v\|_{V_\Gamma} \bigr)
	\quad\hbox{ if $\gamma _0 v \in V_\Gamma$,} \label{LM1} \\
&\|v\|_{H^{3/2}(\Omega)} \le C_{\rm e} \bigl( \| v\|_{H} + \|\Delta v\|_{H} + \|\partial _{\boldsymbol{\nu}} v\|_{H_\Gamma} \bigr)
	\quad\hbox{ if $\partial _{\boldsymbol{\nu}} v \in H_\Gamma$,} \label{LM2}
\end{align}
for all  $v \in V$ with $\Delta v \in H$, where} $C_{\rm e}$ is a suitable positive constant. 
We also \pier{recall that the normal derivative can be interpreted in the following weak sense: for elements  $v \in V$ with $\Delta v \in H$ it holds that $\partial _{\boldsymbol{\nu}} v \in Z_\Gamma'$ and 
\begin{equation}
	\langle \partial _{\boldsymbol{\nu}} v, z_\Gamma \rangle_{Z_\Gamma',Z_\Gamma} 
	= (\Delta v, {\mathcal R}z_\Gamma )_H + (\nabla v, \nabla {\mathcal R}z_\Gamma)_H
	\label{normal}
\end{equation}
for all $z_\Gamma \in Z_\Gamma$ (see, e.g., \cite[Corollary 2.6]{GR86}),} where 
\begin{gather}
\pier{\hbox{${\mathcal R}$ is a recovering operator ${\mathcal R}:Z_\Gamma \to V$ 
such that}}
\notag\\
\pier{\hbox{$({\mathcal R} z_\Gamma)_{|_\Gamma}=
\gamma_0 {\mathcal R} z_\Gamma =z_\Gamma$ for all $z_\Gamma \in Z_\Gamma$.}}
\label{recov}
\end{gather} 
\pier{We fix the linear and bounded operator ${\mathcal R}$ once and for all throughout the paper.}
Notice that the relation~\eqref{normal} implies that 
\begin{equation}
	\|\partial _{\boldsymbol{\nu}} v\|_{Z_\Gamma'} \le \pier{C} \bigl( \|\Delta v\| _{H} + \|\nabla v\| _{H} \bigr). \label{normal1}
\end{equation}
\pier{Next, we recall another useful result~\cite[Theorem~\pier{2.27}, p.~1.64]{BG87} for the trace $\partial_{\boldsymbol{\nu}} v$: in fact,  
if $v \in H^{3/2}(\Omega)$ and additionally $\Delta v \in H$, 
then $\partial_{\boldsymbol{\nu}} v \in H_\Gamma$ and it turns out that}
\begin{equation}
	\pier{\|\partial _{\boldsymbol{\nu}} v\|_{H_\Gamma} \le C \bigl( \| v\| _{H^{3/2}(\Omega)} +\|\Delta v\| _{H}  \bigr).} \label{normal2}
\end{equation}
These facts are useful to complete 
the proof of main theorems. 
We also \pier{point out} the following \pier{inequalities of Poincar\'e and 
Poincar\'e--Wirtinger type (see, e.g.,~\cite{Heb96, Nec67})} 
\begin{align}
	&\|v\|_H^2 \le C_{\rm P} \biggl\{ \int_\Omega |\nabla v| ^2 \dx + 
	\biggl| \pier{\int_\Gamma v_{|_\Gamma} \dg }\biggr|^2
	\biggr\} 
	\quad \mbox{for all } v \in V, 
	\label{Poin0}
	\\
	&\|v\|_{H}^2 \le C_{\rm P} \left\{ 
	\int_\Omega |\nabla v| ^2 \dx + \int_\Gamma |\nabla_\Gamma v_\Gamma|^2 \dg
	\right\} \nonumber\\
	&\qquad\qquad \mbox{for all } (v,v_\Gamma) \in \boldsymbol{V} \ \mbox{with } \int_\Gamma v_\Gamma \dg=0,
	\label{Poin1}
	\\
	&\|z_\Gamma\|_{H_\Gamma}^2 \le C_{\rm P} \biggl\{ 
	\int_\Gamma |\nabla_\Gamma z_\Gamma|^2 \dg
	+
	\biggl| \int_\Gamma z_\Gamma \dg  \biggr|^2
	\biggr\} 
	\quad \mbox{for all }z_\Gamma \in V_\Gamma,
	\label{Poin2}
\end{align}
where $C_{\rm P}>0$ is a constant and  
\begin{equation}
\boldsymbol{V}:=\{ (z,z_\Gamma) \in V \times V_\Gamma : \pier{z_{|_\Gamma}} =z_\Gamma ~ {\rm a.e.~on~}\Gamma \}. \label{pier1}
\end{equation}

\subsection{Starting problem}

\pier{We begin our discussion with a known result concerning a 
quasistatic  Cahn--Hilliard equation on the boundary $\Gamma$, coupled with a bulk condition of Allen--Cahn type~\cite{CF20}. Let 
$\varepsilon, \kappa >0$ be two key parameters that play a crucial role in the asymptotic analysis presented in this paper. Referring 
to the well-posedness results from \cite[Theorems 2.3, 2.4]{CF20},
we are going to recall the existence of a weak solution and partial uniqueness -- specifically, the uniqueness of $u$ and $u_\Gamma$ -- for the 
following system}
\begin{align}
  - \varepsilon \Delta \mu =0 
  & \quad {\rm a.e.\ in ~ } Q, \label{ek1}\\
     \mu_{|_\Gamma} =\mu_\Gamma
  & \quad {\rm a.e.\ on~} \Sigma, \label{ek2}\\
   \partial _t u - \Delta u + \xi +\pi(u) = f, \quad \xi \in \beta (u) 
  & \quad {\rm a.e.\ in~} Q, \label{ek3}\\
   u_{|_\Gamma} = u_\Gamma 
  & \quad {\rm a.e.\ on~} \Sigma, \label{ek4}\\ 
    u(0) = u_0 & \quad {\rm a.e.\ in~} \Omega,
    \label{ek5}\\
  \partial _t u_\Gamma + \varepsilon \partial_{\boldsymbol{\nu}} \mu - \Delta_\Gamma \mu_\Gamma = 0
  & \quad {\rm a.e.\ on~} \Sigma, \label{ek6}\\ 
  \mu_\Gamma = \partial_{\boldsymbol{\nu}} u - \kappa \Delta _\Gamma u_\Gamma 
  + \xi_\Gamma + \pi_\Gamma(u_\Gamma) - f_\Gamma, \quad 
  \xi_\Gamma \in \beta_\Gamma(u_\Gamma)
  & \quad {\rm a.e.\ on~} \Sigma, \label{ek7}\\
  u_\Gamma (0) = u_{0\Gamma} & 
    \quad {\rm a.e.\ on~} \Gamma.
   \label{ek8}  
\end{align}
\pier{The terms $\beta+\pi$ and $\beta_\Gamma+\pi_\Gamma$ result from the derivatives or subdifferentials of the double-well potentials ${\mathcal W}$ and ${\mathcal W}_\Gamma$, respectively.
In particular,}
$\beta$, $\beta_\Gamma: \mathbb{R} \to 2^{\mathbb{R}}$ are maximal monotone graphs on $\mathbb{R} \times \mathbb{R}$, while
$\pi$, $\pi_\Gamma: \mathbb{R} \to \mathbb{R}$ are {L}ipschitz continuous functions. For example, \pier{as for  
$\beta_\Gamma$ and $\pi_\Gamma$ we may consider
\smallskip
\renewcommand{\labelitemi}{$\triangleright$}
\begin{itemize}
 \item $\beta_\Gamma(r)=r^3$, $\pi_\Gamma(r)=-r$ for $r \in \mathbb{R}$ (corresponding to the smooth double well potential);\smallskip
 \item $\beta_\Gamma(r)=\ln((1+r)/(1-r))$, $\pi_\Gamma(r)=-2cr$ for $r \in (-1,1)$ (derived from the singular potential of logarithmic type, where $c>0$ is a \pier{sufficiently} large constant which breaks monotonicity);\smallskip
\item $\beta_\Gamma(r)=\partial I_{[-1,1]}(r)$, $\pi_\Gamma(r)=-r$ for $r \in [-1,1]$ (for the non-smooth potential, where the symbol $\partial$ stands for the subdifferential in $\mathbb{R}$);
\smallskip  
 \item $\beta_\Gamma(r)=0$, $\pi_\Gamma (r)=-r$ for $r \in \mathbb{R}$ (for the backward dynamic boundary condition of the type of heat equation, in the case when $\kappa \to 0$). \smallskip
\end{itemize}
Concerning the last case, we note that a structure of second-order partial differential equation of forward-backward type can be found on the boundary equation: indeed, if we 
combine two equations \eqref{ek6} and \eqref{ek7} and let $\kappa \to 0$, then we find it. The choices for $\beta, \, \pi$ are similar 
to the ones for $\beta_\Gamma, \, \pi_\Gamma$, although we can take different selections for the bulk nonlinearities. In particular,
besides the Allen--Cahn equation we also mention the case of the standard heat equation, where $\beta =\pi \equiv 0$.}
\smallskip

The assumptions for 
$\beta$, $\beta_\Gamma$, $\pi$, $\pi_\Gamma$, and given data are set up as follows:
\begin{enumerate}
\item[(A1)] $\beta$, $\beta_\Gamma: \mathbb{R} \to 2^{\mathbb{R}}$ are maximal monotone graphs in $\mathbb{R} \times \mathbb{R}$, which
coincide with the subdifferentials $\beta =\partial \widehat{\beta}$, $\beta_\Gamma =\partial \widehat{\beta}_\Gamma$ of 
some proper, lower semicontinuous, and convex functions 
$\widehat{\beta}$, $\widehat{\beta}_\Gamma: \mathbb{R} \to [0,+\infty]$ such that 
$\widehat{\beta}(0)=\widehat{\beta}_\Gamma(0)=0$, with the 
corresponding effective domains denoted by $D(\beta)$ and $D(\beta_\Gamma)$, respectively; 
\item[(A2)] $D (\beta_\Gamma) \subseteq D(\beta)$ and there exist two constants 
$\varrho \geq 1$ and $c_0>0$ such that 
\begin{equation}
	\bigl| \beta^\circ(r) \bigr| \le \varrho\bigl| \beta ^\circ_\Gamma (r) \bigr|+c_0 \quad 
	\text{for all } r \in D(\beta_\Gamma);
	\label{ccond}
\end{equation}
\item[(A3)] $\pi$, $\pi_\Gamma: \mathbb{R} \to \mathbb{R}$ are Lipschitz continuous functions with their Lipschitz constants $L$ and $L_\Gamma$, respectively;  
\item[(A4)] $u_0 \in V$, $u_{0\Gamma} \in V_\Gamma$ satisfy 
$\widehat{\beta}(u_0) \in L^1(\Omega)$,
$\widehat{\beta}_\Gamma(u_{0\Gamma}) \in L^1(\Gamma)$,
 and $(u_0)_{|_\Gamma}=u_{0\Gamma}$ a.e.\ on $\Gamma$.  
Moreover, let
\begin{equation*}
	m_\Gamma:=\frac{1}{|\Gamma|} \int_{\Gamma} u_{0\Gamma} \dg \in {\rm int} \, D(\beta_\Gamma); 
\end{equation*}
\item[(A5)] $f \in L^2(0,T;H)$ and $f_\Gamma \in W^{1,1}(0,T;H_\Gamma)$.
\end{enumerate}
\smallskip

As a remark, \pier{we point out that the assumption (A1) allows a 
wide class of suitable monotone terms $\beta$ and 
$\beta_\Gamma$, including singular and nonsmooth graphs.
The assumption (A2) means that $\beta_\Gamma$ is dominant
over $\beta$. Of course, it 
automatically holds if we choose $\beta $ with the same growth behavior of $ \beta_\Gamma$. 
In \eqref{ccond} $\beta^\circ $ and $ \beta^\circ_\Gamma$ denote the minimal sections of $\beta $ and $\beta_\Gamma$, specified by (e.g.\ for $\beta$)}
$\beta^\circ(r):=\{ r^* \in \beta(r) : |r^*|=\min_{s \in \beta(r)} |s|\}$ for $r\in D(\beta)$. 
\smallskip

Under these setting, we now recall the result \pier{shown 
in~\cite[see Theorems~2.3, 2.4]{CF20} and stating the existence of a weak solution to}~\eqref{ek1}--\eqref{ek8}. 

\begin{proposition}
\label{ANA}
Under \pier{the} assumptions {\rm (A1)}--{\rm (A5)}, there exist 
\begin{gather*}
	u \in H^1(0,T;H) \cap C\bigl( [0,T]; V \bigr) \cap L^2(0,T;W), \\
	\mu \in L^2(0,T;V), \quad \xi \in L^2(0,T;H), \\
	u_{\Gamma} \in H^1(0,T;V_\Gamma') \cap L^\infty (0,T;V_\Gamma) \cap L^2(0,T;W_\Gamma), \\ 
	\mu_{\Gamma} \in L^2(0,T;V_\Gamma), \quad \xi_{\Gamma} \in L^2(0,T;H_\Gamma),
\end{gather*}
such that they satisfy \eqref{ek2}--\eqref{ek5}, \eqref{ek7}, \eqref{ek8}, and
\begin{equation}
	\bigl\langle \partial _t u_{\Gamma}(t),z_\Gamma \bigr\rangle_{V_\Gamma',V_\Gamma}
	+ \varepsilon \int_\Omega \nabla \mu (t) \cdot \nabla z \dx 
	+ \int_\Gamma \nabla _\Gamma \mu_{\Gamma}(t) \cdot \nabla_\Gamma z_\Gamma \dg
	=0 \label{ek1and6}
\end{equation}
for all test functions \pier{$(z,z_\Gamma) \in \boldsymbol{V}$
and for a.a.~ $t \in (0,T)$.}
\end{proposition}
\smallskip

\pier{We recall that $\boldsymbol{V}$ is defined in \eqref{pier1}
and emphasize that \eqref{ek1and6} represents a weak formulation of \eqref{ek1} and \eqref{ek6}.
Hereafter, we use as well the space}\pier{%}
\begin{equation}
	\boldsymbol{Z}:=\bigl\{ (z,z_\Gamma) \in V \times Z_\Gamma : z_{|_\Gamma}=z_\Gamma~ {\rm a.e.~ on~} \Gamma \bigr\}
    %\quad \bigl( \simeq V=H^1(\Omega) \bigr).
    \label{pier2}
\end{equation}
and remark that $\boldsymbol{Z}$ is exactly the set of pairs 
$(z,z_\Gamma)$, for all $z\in V$ along with their trace $z_{|_\Gamma}$: then, $\boldsymbol{Z}$ is actually isomorphic to $V$.}
\smallskip

We \pier{term $({\rm P})_{\varepsilon\kappa}$ the above problem, which is 
formally described by equations and conditions~\eqref{ek1}--
\eqref{ek8}. 
We deal with this problem, in the aim of performing 
three asymptotics: $\kappa \to 0$, $\varepsilon \to 0$, and both of them tending to $0$.   
Therefore, in several points it will be important to make clear the dependence of the components of the solution in terms of  
$\varepsilon$ and $\kappa$, so we will use $u_{\varepsilon, \kappa}$ in place of $u$, $\mu_{\varepsilon,\kappa}$ in place $\mu$, and so on. Both notations will be employed according to the context.}
\smallskip

\pier{In order to discuss higher regularities and other properties of the solutions, we need additional requirements for 
$\beta$ and $\beta _\Gamma$, related to the growth conditions. 
A similar framework has been considered in the contributions~\cite{CFS22, CFS22b} and reads 
\begin{enumerate}
\item[(A6)] $D(\beta)=D(\beta_\Gamma)$ and there exists a constant $C_\beta \ge 1$ such that 
\begin{equation*}
	\frac{1}{C_\beta}\bigl| \beta_{\Gamma} ^\circ (r) \bigr| - C_\beta \le \bigl| \beta^\circ (r) \bigr|
	\le C_\beta \bigl( \bigl| \beta_\Gamma ^\circ (r) \bigr| +1 \bigr) \quad {\rm for~all~}r \in D(\beta).
\end{equation*}
\end{enumerate} 
Of course, this is realized by choosing $\beta$ with the same domain and growth of $\beta_\Gamma$. 
As a remark, we anticipate that the error estimates can be obtained under this assumption.}
\smallskip

\subsection{Three target problems}

We set up three target problems which \pier{are obtained as follows: $\kappa \to 0$ with a fixed $\varepsilon>0$, 
$\varepsilon \to 0$ with a fixed $\kappa>0$, and 
both $\varepsilon, \kappa \to 0$. We name each problems by $({\rm P})_{\varepsilon}$, $({\rm P})_{\kappa}$, and $({\rm P})$, respectively (see~{\sc Figure 1}).}
\smallskip

\begin{figure}[h]
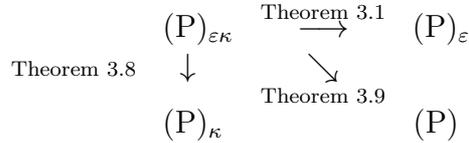

\begin{equation*}
\begin{matrix}
& ({\rm P})_{\varepsilon\kappa} 
& \overset{\text{\tiny{\rm Theorem~\ref{thm1}}}}{\longrightarrow} 
& ({\rm P})_{\varepsilon} \\
\text{\tiny{\rm Theorem~\ref{thm2}}} 
& \takeshi{~ \downarrow \quad} 
& \takeshi{
\underset{\text{\tiny{\rm Theorem~\ref{thm3}}}}{\searrow} } 
& \\ 
& \takeshi{({\rm P})_{\kappa}~~} 
& \qquad \qquad 
& \takeshi{({\rm P})~} \\
\end{matrix}
\end{equation*}
\caption{Asymptotics between $({\rm P})_{\varepsilon\kappa}$, $({\rm P})_{\varepsilon}$, $({\rm P})_{\kappa}$, and $({\rm P})$}
\end{figure}

The first problem $({\rm P})_\varepsilon$ \pier{contains} a sort of forward-backward dynamic boundary condition. More precisely, \pier{the resulting system couples an Allen--Cahn equation for $u$ with a possible forward-backward dynamic boundary condition for the trace $u_\Gamma$. The problem consists in finding a \pier{sextuple} $(u,\mu,\xi, u_\Gamma, \mu_\Gamma, \xi_\Gamma)$ of functions that satisfy
%\eqref{ek1}--\eqref{ek6}, \eqref{e7}, \eqref{ek8}, that is  
\begin{align}
  - \varepsilon \Delta \mu =0 
  & \quad {\rm a.e.\ in ~ } Q, 
  \label{e1} 
 \\
         \mu_{|_\Gamma} =\mu_\Gamma
  & \quad {\rm a.e.\ on~} \Sigma, 
  \label{e2}
  \\
   \partial _t u - \Delta u + \xi +\pi(u) = f, \quad \xi \in \beta (u) 
  & \quad {\rm a.e.\ in~} Q, 
  \label{e3}
  \\
   u_{|_\Gamma} = u_\Gamma 
  & \quad {\rm a.e.\ on~} \Sigma, 
  \label{e4}
  \\ 
  u(0) = u_0 & \quad {\rm a.e.\ in~} \Omega,
  \label{e5} \\
  \partial _t u_\Gamma + \varepsilon \partial_{\boldsymbol{\nu}} \mu - \Delta_\Gamma \mu_\Gamma = 0
  & \quad {\rm a.e.\ on~} \Sigma, \label{e6}
  \\ 
  \mu_\Gamma = \partial_{\boldsymbol{\nu}} u 
  + \xi_\Gamma + \pi_\Gamma(u_\Gamma) - f_\Gamma, \quad 
  \xi_\Gamma \in \beta_\Gamma(u_\Gamma)
  & \quad {\rm a.e.\ on~} \Sigma, \label{e7} \\
  u_\Gamma (0) = u_{0\Gamma} & 
    \quad {\rm a.e.\ on~} \Gamma, \label{e8}
\end{align}
that is, the system~\eqref{ek1}--\eqref{ek8} with $\kappa=0$.
%where \eqref{ek7} in $({\rm P})_{\varepsilon\kappa}$ is replaced by~\eqref{e7} here.
Now, we emphasize that \eqref{e6} and \eqref{e7} provide a 
nonlinear diffusion equation in terms of $u_\Gamma$, which somehow works as a dynamic boundary condition for the equations in the bulk, where we have the Laplace equation \eqref{e1} for $\mu$ 
with non-homogeneous Dirichlet boundary condition \eqref{e2} and 
the Allen--Cahn equation~\eqref{e3} for $u$ 
with non-homogeneous Dirichlet boundary condition \eqref{ek4}. 
As a remark, bulk equations have two kinds of boundary conditions, respectively, but in terms of $\mu_\Gamma$ and $u_\Gamma$ that are 
the unknowns on the boundary. Thus, the full system \eqref{e1}--
\eqref{e8} actually yields a transmission problem in the bulk and on the boundary.}
\smallskip

The second problem $({\rm P})_\kappa$ \pier{is provided by the 
Allen--Cahn equation~\eqref{e3} with a dynamic boundary condition of Cahn--Hilliard type: indeed, one has to
find a \pier{quintuple} $(u,\xi, u_\Gamma, \mu_\Gamma, \xi_\Gamma)$ of functions satisfying 
\eqref{e3}, \eqref{e4}, 
\begin{align}
  \partial _t u_\Gamma  - \Delta_\Gamma \mu_\Gamma = 0
  & \quad {\rm a.e.\ on~} \Sigma, \label{k6}\\ 
  \mu_\Gamma = \partial_{\boldsymbol{\nu}} u - \kappa \Delta _\Gamma u_\Gamma 
  + \xi_\Gamma + \pi_\Gamma(u_\Gamma) - f_\Gamma, \quad 
  \xi_\Gamma \in \beta_\Gamma(u_\Gamma)
  & \quad {\rm a.e.\ on~} \Sigma, \label{k7}
\end{align}
and the initial conditions~\eqref{e5} and \eqref{e8}.} 
We point out that in the problem $({\rm P})_\kappa$ the chemical potential $\mu$ in the bulk completely disappears from the formulation. 
\pier{Problem~$({\rm P})_\kappa$ is also a sort of transmission problem via the Dirichlet boundary condition \eqref{e4}, where $u_\Gamma$ has to solve the Cahn--Hilliard equation specified by \eqref{k6} and \eqref{k7} and including  the normal derivative $\partial _{\boldsymbol{\nu}} u$ of $u$}. 
\smallskip

The last problem $({\rm P})$ \pier{reduces to the previous one, but with $\kappa=0 $ in \eqref{k7}, or it may be seen as the system \eqref{e3}--\eqref{e8} with $\varepsilon = 0 $ in \eqref{e6}. Thus, the 
solution we search is a \pier{quintuple} $(u,\xi, u_\Gamma, \mu_\Gamma, \xi_\Gamma)$ of functions fulfilling   
\eqref{e3}--\eqref{e5}, \eqref{k6}, \eqref{e7}, \eqref{e8}, that~is  
\begin{align}
   \partial _t u - \Delta u + \xi +\pi(u) = f, \quad \xi \in \beta (u) 
  & \quad {\rm a.e.\ in~} Q, \notag \\
   u_{|_\Gamma} = u_\Gamma 
  & \quad {\rm a.e.\ on~} \Sigma, \notag \\ 
%    u(0) = u_0 & \quad {\rm a.e.\ in~} \Omega,
%    \label{p3}\\
  \partial _t u_\Gamma  - \Delta_\Gamma \mu_\Gamma = 0
  & \quad {\rm a.e.\ on~} \Sigma, \notag \\ 
  \mu_\Gamma = \partial_{\boldsymbol{\nu}} u
  + \xi_\Gamma + \pi_\Gamma(u_\Gamma) - f_\Gamma, \quad 
  \xi_\Gamma \in \beta_\Gamma(u_\Gamma)
  & \quad {\rm a.e.\ on~} \Sigma, \notag   
\end{align}
without rewriting initial conditions.} 
We point out that in the strong formulation of the problem ({\rm P}), 
the \pier{last two boundary equations can be merged as only one equation. 
A  striking example of $({\rm P})$ is represented by a heat equation in the bulk, coupled with a backward dynamic boundary condition on the boundary: 
\begin{align*}
   \partial _t u - \Delta u = f
  & \quad {\rm a.e.\ in~} Q, \\\
   u_{|_\Gamma} = u_\Gamma 
  & \quad {\rm a.e.\ on~} \Sigma, \\ 
%    u(0) = u_0 & \quad {\rm a.e.\ in~} \Omega,
%    \label{k3}\\
  \partial _t u_\Gamma  
  +\Delta_\Gamma u_\Gamma = -\Delta_\Gamma (f_\Gamma - \partial_{\boldsymbol{\nu}} u) 
  & \quad {\rm a.e.\ on~} \Sigma,
\end{align*}
where the choices $\beta (r)=\pi(r)=\beta_\Gamma(r) =0$, $\pi_\Gamma (r)=-r$ for $r \in \mathbb{R}$, have been taken. 
As a remark, note that the sign in front of the Laplace--Beltrami operator in the left-hand side of the last equation is positive.}

\smallskip 

From the next section, we will discuss the relationship between $({\rm P})_{\varepsilon\kappa}$, $({\rm P})_{\kappa}$, $({\rm P})_{\varepsilon}$, and 
$({\rm P})$ by the limiting procedure. 
Under the assumptions (A1)--(A5), \pier{the well-posedness 
of $({\rm P})_{\varepsilon\kappa}$ is ensured by 
Proposition~\ref{ANA}. Additionally, the same kind of estimates 
obtained in the proof holds at the level of Yosida approximations 
of $\beta$ and $\beta_\Gamma$, see \cite[Lemma~A.1]{CF20b}. 
Based on the known result, we are now dealing with the uniform estimates.} 

\smallskip

Moreover, \pier{let us comment on the assumption (A6), which has been already used to derive the higher regularity of the solution in the previous works~\cite{CFS22, CFS22b}. In general, the regularity $L^2(0,T;H_\Gamma)$ for $\xi_\Gamma$, the element of $\beta_\Gamma(u_\Gamma)$, is related to the one of the normal derivative $\partial _{\boldsymbol{\nu}} u$. The assumption (A6) also helps to obtain the regularity in $ u\in L^2(0,T;H^{3/2}(\Omega))$ from the elliptic estimate~\eqref{LM2}, and from the standard trace theory \eqref{statre1} this ensures that~$u_\Gamma \in L^2(0,T;V_\Gamma)$.}

\section{Asymptotic analyses}
\label{asy-ana}

\pier{In this section, we analyze three asymptotic regimes: \takeshi{$\kappa \to 0$, 
$\varepsilon \to 0$,} and the simultaneous limit where both parameters tend to zero. Specifically, we aim to illustrate the following convergence framework:
\begin{equation*}
({\rm P})_{\varepsilon\kappa} \stackrel{\rm Theorem~\ref{thm1}}{\longrightarrow} ({\rm P})_{\varepsilon}, \quad ({\rm P})_{\varepsilon\kappa} \stackrel{\rm Theorem~\ref{thm2}}{\longrightarrow} ({\rm P})_{\kappa}, \quad ({\rm P})_{\varepsilon\kappa} \stackrel{\rm Theorem~\ref{thm3}}{\longrightarrow} ({\rm P}).
\end{equation*}
We begin with the asymptotic analysis
$({\rm P})_{\varepsilon\kappa} \to ({\rm P})_{\varepsilon}$ as $\kappa \to 0$ which is addressed in Subsections 3.1 and 3.2. In Subsection 3.1, we establish uniform estimates, and in Subsection 3.2, we complete the proof of Theorem~\ref{thm1}.
Next, in Subsection 3.3, we study the limit $({\rm P})_{\varepsilon\kappa} \to ({\rm P})_{\kappa}$ as $\varepsilon \to 0$. 
Finally, Subsection 3.4 is devoted to the joint asymptotic behavior $({\rm P})_{\varepsilon\kappa} \to ({\rm P})$ as both parameters tend to zero.}

\subsection{First asymptotic result and uniform estimates}
\pier{In this subsection, we consider the limit as $\kappa \to 0$ while keeping $\varepsilon>0$ fixed. 
This corresponds to the vanishing diffusion term $\kappa \Delta _\Gamma$ on the boundary. 
We now state our first theorem concerning the asymptotic behavior of $({\rm P})_{\varepsilon\kappa}$ as it converges to $({\rm P})_\varepsilon$.}

\begin{theorem}\label{thm1}
Assume {\rm (A1)}--{\rm (A5)}. Then there exists \pier{a sextuple} $(u,\mu, \xi, u_\Gamma, \mu_\Gamma, \xi_\Gamma)$ satisfying the following regularity properties:
\begin{gather*}
	u \in H^1(0,T;H) \cap L^\infty(0,T;V), \quad \Delta u \in L^2(0,T;H),\\
	\mu \in L^2(0,T;V), \quad \Delta \mu \in L^2(0,T;H), \quad \xi \in L^2(0,T;H), \\
	u_\Gamma \in H^1(0,T;V_\Gamma') \cap C\bigl([0,T];H_\Gamma \bigr) \cap L^\infty (0,T;Z_\Gamma), \\
	\mu_\Gamma \in L^2(0,T;V_\Gamma), \quad \xi_\Gamma \in L^2(0,T;Z_\Gamma')
\end{gather*}
\pier{and fulfilling \eqref{e1}--\eqref{e5}, \eqref{e8}, and 
the conditions \eqref{e6} and \eqref{e7} in the following weak sense: 
\begin{align} 
	\langle \partial_t u_{\Gamma},z_\Gamma \rangle_{V_\Gamma',V_\Gamma}
	+ \varepsilon \langle \partial_{\boldsymbol{\nu}} \mu, z_\Gamma \rangle_{Z_\Gamma',Z_\Gamma} 
	+ \int_\Gamma \nabla _\Gamma \mu_{\Gamma} \cdot \nabla_\Gamma z_\Gamma \dg
	=0 \notag \\
	\text{ for all } z_\Gamma \in V_\Gamma, 
	\text{ a.e.\ in } (0,T),
	\label{e6weak}\\
	\noalign{\smallskip}
	\int_\Gamma \mu_\Gamma z_\Gamma \dg =
	\langle \partial_{\boldsymbol{\nu}} u + \xi _{\Gamma},z_\Gamma \rangle_{Z_\Gamma',Z_\Gamma}
	+ \int_\Gamma \bigl( \pi_\Gamma(u_\Gamma) -f_\Gamma \bigr) z_\Gamma \dg \quad \hbox{and} 
	\notag %\label{main1-2}
	\\
	\langle \xi_\Gamma ,z_\Gamma-u_\Gamma \rangle_{Z_\Gamma',Z_\Gamma } 
	+ \int_\Gamma \widehat{\beta}_\Gamma (u_\Gamma) \dg 
	\le \int_\Gamma \widehat{\beta}_\Gamma(z_\Gamma)\, \dg \notag \\
	\mbox{for all } z_\Gamma \in Z_\Gamma, 
	\mbox{ a.e.\ in } (0,T).
	\label{e7weak}
\end{align}
Moreover,} the \pier{sextuple}
$(u,\mu, \xi, u_\Gamma, \mu_\Gamma, \xi_\Gamma)$ is obtained as limit of the family 
$\{(u_\kappa,\mu_\kappa, \xi_\kappa$, $u_{\Gamma,\kappa}, \mu_{\Gamma,\kappa}, \xi_{\Gamma,\kappa})\}_{\kappa \in (0,1]}$ \pier{of solutions to $({\rm P})_{\varepsilon\kappa}$ as $\kappa \searrow 0$ in the following sense:
there is a vanishing} subsequence $\{\kappa_k \}_{k \in \mathbb{N}}$ such that, as $k \to +\infty$, 
\begin{align}
	&u_{\kappa_k} \to u \quad \text{weakly~star~in~} H^1(0,T;H) \cap L^\infty (0,T;V),
	\label{cv1k} \\
	&u_{\kappa_k} \to u 
	\quad \text{strongly~in~} C\bigl([0,T];H \bigr),
	\label{cv2k} \\
	&\mu_{\kappa_k} \to \mu  \quad \text{weakly~in~} L^2 (0,T;V),
	\label{cv4k} \\
	&\xi_{\kappa_k} \to \xi \quad  \text{weakly~in~} L^2 (0,T;H),
	\label{cv5k} \\
	&u_{\Gamma,\kappa_k} \to u_\Gamma \quad \text{weakly~star~in~} H^1(0,T;V_\Gamma') \cap L^\infty (0,T;Z_\Gamma),
	\label{cv6k} \\
	&u_{\Gamma,\kappa_k} \to u_\Gamma \quad \text{strongly~in~} C\bigl([0,T];H_\Gamma \bigr),
	\label{cv7k} \\
	&{\kappa_k} u_{\Gamma,\kappa_k} \to 0 \quad \text{strongly~in~} L^\infty (0,T;V_\Gamma),
	\label{cv8k}\\
	&\mu_{\Gamma,\kappa_k} \to \mu_\Gamma \quad \text{weakly~in~} L^2 (0,T;V_\Gamma),
	\label{cv9k} \\
	&\xi_{\Gamma,\kappa_k} \to \xi_\Gamma \quad \text{weakly~in~} L^2 (0,T;V_\Gamma'),
	\label{cv10k} \\
	&(-{\kappa_k} \Delta_\Gamma u_{\Gamma,\kappa_k} +\xi_{\Gamma,\kappa_k}) \to \xi_\Gamma \quad \text{weakly~in~} 
	L^2 (0,T;Z_\Gamma').
	\label{cv11k}
\end{align}
\end{theorem}
\smallskip

\pier{The proof of this theorem is presented in the following subsection, after establishing the basic estimates in the current one.}
\smallskip

\pier{Arguing as in previous works~\cite{CF15, CF20, CF20b, CFS22, CFS22b, CFW20}, we employ the Yosida approximations
$\beta_\lambda$ of 
$\beta$ and 
$\beta_{\Gamma,\lambda}$ of 
$\beta_\Gamma$, with parameter $\lambda>0$: 
$\beta_\lambda $ and $\beta_{\Gamma, \lambda}$ are defined by 
\begin{gather*}
	\beta_\lambda (r):=\frac{1}{\lambda} \bigl( r-J_\lambda (r) \bigr) :=\frac{1}{\lambda} 
	\bigl( r-(I+\lambda \beta) ^{-1}(r)\bigr), \\
	\beta_{\Gamma,\lambda} (r):=\frac{1}{\lambda} \bigl( r-J_{\Gamma,\lambda} (r) \bigr) :=\frac{1}{\lambda} 
	\bigl( r-(I+\lambda \beta_\Gamma) ^{-1}(r)\bigr) \quad {\rm for~} r \in \mathbb{R}.
\end{gather*}
%As a remark, using {\rm (A2)}
%the assumption of a constant $\varrho \ge 1$, 
%we do not need to take different parameter of Yosida approximation, 
%see, \cite[Appendix]{CF20}. 
From the theory of maximal monotone operators (see, e.g., \cite{Bar10, Bre73}), 
it follows that $\beta_\lambda$ and $\beta_{\Gamma,\lambda}$ are Lipschitz continuous 
functions with Lipschitz constant  $1/\lambda$. Moreover, it holds that
\begin{gather*}
	\bigl| \beta_\lambda (r) \bigr| \le \bigl|\beta^{\circ}(r) \bigr|,
	\quad 
	0 \le \widehat{\beta}_\lambda(r) \le \widehat{\beta}(r), 
		\quad {\rm for~all~} r \in D(\beta),\\
	\bigl| \beta_{\Gamma,\lambda} (r) \bigr| \le \bigl|\beta_\Gamma^{\circ}(r) \bigr|, 
	\quad 
	0 \le \widehat{\beta}_{\Gamma,\lambda} (r) \le \widehat{\beta}_\Gamma(r)
	\quad {\rm for~all~} r \in D(\beta_\Gamma). 
\end{gather*}
Moreover, in order to make rigorous the first estimate we are showing, let us 
consider an additional approximation based on viscous Cahn--Hilliard equations in the bulk and on the boundary. 
The reason is that for the proof of the first estimate
we need the regularity $\partial _t u_{\Gamma, \kappa}\in L^2(0,T;H_\Gamma)$, which is not ensured by Proposition~\ref{ANA}. Then, we use the same approximation employed in \cite{CF20} and,
applying \cite[Theorem 2.2]{CF15} and \cite[Proposition~3.1]{CF20}, we see that 
for each $\tau, \lambda, \varepsilon \in (0,1]$, and $\kappa \in (0,1]$,
there exists a \pier{sextuple} $(u_\kappa,\mu_\kappa,\xi_\kappa, u_{\Gamma,\kappa}, 
\mu_{\Gamma,\kappa}, \xi_{\Gamma,\kappa})$
fulfilling at least that
\begin{gather*}
	u_\kappa \in H^1(0,T;H) \cap C\bigl( [0,T]; V \bigr) \cap L^2(0,T;W), \\
	\mu_\kappa \in L^2(0,T;W), \quad \xi_\kappa =\beta_\lambda (u_\kappa) \in L^2(0,T;V), \\
	u_{\Gamma,\kappa} \in H^1(0,T;H_\Gamma) \cap C\bigl( [0,T];V_\Gamma\bigr) \cap L^2(0,T;W_\Gamma), \\ 
	\mu_{\Gamma, \kappa} \in L^2(0,T;W_\Gamma), \quad \xi_{\Gamma,\kappa}=\beta_{\Gamma,\lambda}(u_{\Gamma,\kappa}) \in L^2(0,T;V_\Gamma)
\end{gather*}
and solving 
\begin{align}
 \tau \partial_t u_\kappa - \varepsilon \Delta \mu_\kappa =0 
  & \quad {\rm a.e.\ in ~ } Q, \label{ek1lam}\\
   \tau \mu_\kappa = \partial_t u_\kappa - \Delta u_\kappa + \beta_\lambda (u_\kappa) +\pi(u_\kappa) - f 
  & \quad {\rm a.e.\ in~} Q, \label{ek2lam}\\
   (\mu_\kappa)_{|_\Gamma} =\mu_{\Gamma,\kappa}
  & \quad {\rm a.e.\ on~} \Sigma, \label{ek3lam}\\
   (u_\kappa)_{|_\Gamma} = u_{\Gamma,\kappa} 
  & \quad {\rm a.e.\ on~} \Sigma, \label{ek4lam}\\ 
  \partial _t u_{\Gamma,\kappa} + \varepsilon \partial_{\boldsymbol{\nu}} \mu_\kappa 
  - \Delta_\Gamma \mu_{\Gamma,\kappa} = 0
  & \quad {\rm a.e.\ on~} \Sigma, \label{ek5lam}\\ 
  \mu_{\Gamma,\kappa} = \tau \partial_t u_{\Gamma,\kappa} 
  +  \partial_{\boldsymbol{\nu}} u_\kappa - \kappa \Delta _\Gamma u_{\Gamma,\kappa} 
  + \beta_{\Gamma, \lambda} (u_{\Gamma,\kappa}) + \pi_\Gamma(u_{\Gamma,\kappa}) - f_\Gamma
  & \quad {\rm a.e.\ on~} \Sigma, \label{ek6lam}\\
     u_\kappa(0) = u_0 & \quad {\rm a.e.\ in~} \Omega,\label{ek7lam} \\
  u_{\Gamma,\kappa} (0) = u_{0\Gamma} & 
    \quad {\rm a.e.\ on~} \Gamma.\label{ek8lam}
\end{align}
Here, the new terms with the coefficient $\tau$ in \eqref{ek1lam}, \eqref{ek2lam}, and \eqref{ek6lam} 
actually play a role of regularizing terms.  
%of viscosity for the equation and dynamic boundary conditions of Cahn--Hilliard type. 
%Thanks to this we can get the enough regularity 
%$\partial _t u_{\Gamma,\kappa} \in L^2(0,T;H_\Gamma)$, 
%indeed we will use it to prove the first uniform estimate. 
Moreover, recalling the discussion in~\cite[Section 4.3]{CF20}, we can 
consider the limiting procedure $\tau \to 0$ keeping $\lambda >0$. 
In order to make clear the structure, we can also write this approximate system by 
\begin{align*}
&
\begin{pmatrix}
\tau & 0 \\
0 & 1 
\end{pmatrix}
\frac{\partial}{\partial t}
\begin{pmatrix}
u_\kappa \\
u_{\Gamma ,\kappa} 
\end{pmatrix}
+  
\begin{pmatrix}
-\varepsilon \Delta & 0 \\
\varepsilon \partial _{\boldsymbol{\nu}} & -\Delta_\Gamma  
\end{pmatrix}
\begin{pmatrix}
\mu_\kappa \\ 
\mu_{\Gamma, \kappa}
\end{pmatrix}
 = 
\begin{pmatrix}
0 \\ 0
\end{pmatrix}, \\
&
\begin{pmatrix}
\tau & 0 \\
0 & 1 
\end{pmatrix}
\begin{pmatrix}
\mu_\kappa \\
\mu_{\Gamma ,\kappa} 
\end{pmatrix}
=
\begin{pmatrix}
1 & 0 \\
0 & \tau
\end{pmatrix}  
\frac{\partial}{\partial t}
\begin{pmatrix}
u_\kappa \\
u_{\Gamma ,\kappa} 
\end{pmatrix}
+
\begin{pmatrix}
- \Delta & 0 \\
\partial _{\boldsymbol{\nu}} & -\kappa \Delta_\Gamma  
\end{pmatrix}
\begin{pmatrix}
u_\kappa \\ 
u_{\Gamma, \kappa}
\end{pmatrix}
\\
&\quad {}
+ 
\begin{pmatrix}
\beta _\lambda (u_\kappa) + \pi (u_\kappa) - f \\ 
\beta_{\Gamma, \lambda} (u_{\Gamma,\kappa}) + \pi_\Gamma (u_{\Gamma,\kappa}) - f _\Gamma 
\end{pmatrix}, \\
&\text{ with the initial condition }
\begin{pmatrix}
u_{\kappa} (0) \\
u_{\Gamma,\kappa} (0) 
\end{pmatrix}
=\begin{pmatrix}
u_{0} \\
u_{0\Gamma} 
\end{pmatrix}.
\end{align*}
From this, we see that the system is nothing but a viscous Cahn--Hilliard equation for the pair $(u_\kappa, u_{\Gamma, \kappa})^{\sf T}$ that in the sequel will be written as $(u_\kappa, u_{\Gamma, \kappa})$. The same applies to other pairs.}

\begin{lemma} \label{FE}
There exists a positive constant $M_1$, independent of $\tau, \lambda, \varepsilon$, and $\kappa$, such that
\begin{align*}
	&\|\partial_t u_\kappa \|_{L^2(0,T;H)} 
	+
	\sqrt{\tau}\|\partial_t u_{\Gamma,\kappa} \|_{L^2(0,T;H_\Gamma)}
	+ \| u _\kappa \|_{L^\infty (0,T;V)}  \pier{{}+ \| u_{\Gamma,\kappa}\|_{L^\infty(0,T;Z_\Gamma)}}
	\notag \\
	&{}
	+ \sqrt{\kappa} \pier{\| u_{\Gamma,\kappa}\|_{L^\infty(0,T;V_\Gamma)}}
	+ \bigl\| \widehat{\beta}_\lambda (u_\kappa) \bigr\|_{L^\infty(0,T;L^1(\Omega))}
	+  \bigl\| \widehat{\beta}_{\Gamma,\lambda} (u_{\Gamma,\kappa}) \bigr\|_{L^\infty(0,T;L^1(\Gamma))}
	\notag \\
	&{}
	+ \sqrt{\varepsilon} \|\nabla \mu_{\kappa}\|_{L^2(0,T;H)} 
	+ \|\nabla_\Gamma \mu_{\Gamma, \kappa}\|_{L^2(0,T;H_\Gamma)} \le M_1. 
\end{align*}
\end{lemma}

\begin{proof} 
\pier{We test \eqref{ek2lam} by $\partial_t u_\kappa$ and integrate the resultant over $(0,t)$ with respect to the time variable $s$, obtaining}
\begin{align}
	& \int_0^t \! \! \int_\Omega |\partial_t u_\kappa |^2 
	\dx \ds
	+ \frac{1}{2} \int_\Omega \bigl| \nabla u_\kappa (t) \bigr|^2 
	\dx
	+ \int_\Omega \widehat{\beta}_\lambda \bigl( u_\kappa (t)\bigr) 
	\dx
	\notag \\
	& \quad {}
	+ \int_\Omega \widehat{\pi} \bigl( u_\kappa (t)\bigr) 
	\dx
	- \int_0^t \! \! \int_\Gamma \partial_{\boldsymbol{\nu}} u_\kappa \partial_t u_{\Gamma,\kappa} 
	\dg\ds
	-
	\tau \int_0^t \! \! \int_\Omega \mu_\kappa \partial _t u_\kappa 
	\dx \ds
	\notag \\ 
	 &{} =
	 \frac{1}{2} \int_\Omega | \nabla u_0 |^2 
	\dx
	+ \int_\Omega \widehat{\beta}_\lambda ( u_0) 
	\dx 
	+ \int_\Omega \widehat{\pi}( u_0) 
	\dx
	+ \int_0^t \! \! \int_\Omega f \partial _t u_\kappa 
	\dx \ds
	\label{1st3-1}
\end{align}
for all $t \in [0,T]$. \pier{Applying the same procedure to equation~\eqref{ek6lam}, 
tested with $\partial_t u_{\Gamma,\kappa} \in L^2(0,T;H_\Gamma)$, yields the following:
\begin{align}
	& \tau  \int_0^t \! \! \int_\Gamma |\partial_t u_{\Gamma,\kappa}|^2 \dg \ds 
	+ \frac{\kappa}{2} \int_\Gamma \bigl| \nabla_\Gamma u_{\Gamma,\kappa} (t) \bigr|^2 
	\dg + \int_\Gamma \widehat{\beta}_{\Gamma,\lambda} \bigl( u_{\Gamma,\kappa} (t)\bigr) 
	\dg 
	\notag \\
	& \quad {}+ \int_\Gamma \widehat{\pi}_\Gamma \bigl( u_{\Gamma,\kappa} (t)\bigr) 
	\dg 
	+ \int_0^t \! \! \int_\Gamma \partial_{\boldsymbol{\nu}} u_\kappa \partial_t u_{\Gamma,\kappa} 
	\dg \ds 
	- \int_0^t \! \! \int_\Gamma \mu_{\Gamma,\kappa} \partial_t u_{\Gamma,\kappa} 
	\dg \ds 
	\notag \\
	&{}= \frac{\kappa}{2} \int_\Gamma | \nabla_\Gamma u_{0\Gamma}|^2 
	\dg
	+ \int_\Gamma \widehat{\beta}_{\Gamma,\lambda} ( u_{0\Gamma}) 
	\dg 
	+ \int_\Gamma \widehat{\pi}_\Gamma ( u_{0\Gamma}) 
	\dg 
	+
	\int_0^t \! \! \int_\Gamma f_\Gamma \partial _t u_{\Gamma,\kappa} 
	\dg \ds. 
	\label{2nd3-1}
\end{align}
About the terms involving $\widehat{\pi}$ and $\widehat{\pi}_\Gamma$, we remark that from the assumption (A3) it follows that 
\begin{align*}
	\bigl| \widehat{\pi}(r) \bigr| & \le 
	\int_0^r \bigl| \pi (s) - \pi (0) \bigr| \ds + \int_0^r \bigl| \pi(0) \bigr| \ds 
	\notag \\	
	& \le \frac{L}{2} r^2 + \left( \frac{L}{2} r^2 + \frac{1}{2L} \bigl| \pi(0) \bigr|^2\right)
	= L |r|^2 + \frac{1}{2L} \bigl| \pi(0) \bigr|^2
	\quad {\rm for~all~} r \in \mathbb{R}
\end{align*}
and similar inequalities hold for $\widehat{\pi}_\Gamma$. Then there exists a constant $C_{\rm L}>0$ such that 
\begin{equation}
	\int_\Omega \bigl| \widehat{\pi}(z) \bigr| \dx \le L \|z\|_{H}^2 +C_{\rm L}, 
	\quad 
	\int_\Gamma \bigl| \widehat{\pi}_\Gamma(z_\Gamma) \bigr| \dg \le L_\Gamma\|z_\Gamma\|_{H_\Gamma}^2 +C_{\rm L},
	\label{L}
\end{equation}
for all $z \in H$ and $z_\Gamma \in H_\Gamma$, respectively. On the other hand,
for the last term of \eqref{2nd3-1} we note that}
\begin{align}
	& \int_0^t \! \! \int_\Gamma f_\Gamma \partial _t u_{\Gamma,\kappa} 
	\dg \ds 
	\notag \\
	& =  - \int_0^t \! \! \int_\Gamma \partial _t f_\Gamma \,  u_{\Gamma,\kappa} 
	\dg \ds 
	+ \int_\Gamma  f_\Gamma(t) u_{\Gamma,\kappa} (t)
	\dg
	- 
	\int_\Gamma  f_\Gamma(0) u_{0\Gamma}
	\dg \notag \\
	& \le \int_0^t \| \partial _t f_\Gamma \|_{H_\Gamma} \|u_{\Gamma,\kappa} \|_{H_\Gamma}
	\ds 
	+ \ \bigl\| f_\Gamma(t) \bigr\|_{H_\Gamma} \bigl\| u_{\Gamma,\kappa} (t) \bigr\|_{H_\Gamma}
	+ \bigl\| f_\Gamma(0) \bigr\|_{H_\Gamma}\| u_{0\Gamma} \|_{H_\Gamma}.
	\label{3rd3-1}
\end{align}
Next, \pier{multiplying \eqref{ek1lam} by $\mu_\kappa$, \eqref{ek5lam} by $\mu_{\Gamma,\kappa}$ and using \eqref{ek3lam} we infer that} 
\begin{align}
	& -\tau  \int_\Omega \partial_t u_\kappa \mu_\kappa  \dx 
	-  \int_\Gamma \partial_t u_{\Gamma, \kappa} \mu_{\Gamma, \kappa}  \dg
	 = \varepsilon  \int_\Omega |\nabla \mu_\kappa |^2 \dx 
	+ \int_\Gamma |\nabla _\Gamma \mu_{\Gamma,\kappa }|^2 \dg. 
	\label{4th3-1}
\end{align}
\pier{Then, we integrate \eqref{4th3-1} over \pier{$(0,t)$} with respect to the time variable
and take advantage of \eqref{1st3-1}--\eqref{3rd3-1}. 
Then, summing and adding $(1/2) \int_\Omega |u_\kappa(t)|^2 \dx$ to both sides, thanks to the properties of the Moreau--Yosida regularizations and Young's inequality we deduce that
\begin{align}
	& \frac{1}{2}\int_0^t \! \! \int_\Omega |\partial_t u_\kappa |^2 
	\dx \ds 
	+ \frac{1}{2}\bigl\| u_\kappa (t) \bigr\|_{V}^2
	+ \int_\Omega \widehat{\beta}_\lambda \bigl( u_\kappa (t)\bigr) 
	\dx 
	\notag \\
	& \quad {}
	+\tau \int_0^t \! \! \int_\Gamma |\partial_t u_{\Gamma,\kappa}|^2 \dg \ds 
	+\frac{\kappa}{2} \int_\Gamma \bigl| \nabla_\Gamma u_{\Gamma,\kappa} (t) \bigr|^2 
	\dg 
	+\int_\Gamma \widehat{\beta}_{\Gamma,\lambda} \bigl( u_{\Gamma,\kappa} (t)\bigr) 
	\dg 
	\notag \\
	& \quad {}
	+
	\varepsilon \int_0^t \! \! \int_\Omega |\nabla \mu_\kappa |^2 \dx \ds 	
	+ \int_0^t \!\! \int_\Gamma |\nabla _\Gamma \mu_{\Gamma,\kappa }|^2 \dg \ds
	\notag \\ 
	 & \le 
	 \frac{1}{2} \int_\Omega \bigl| u_\kappa(t) \bigr|^2 \dx
	 + 
	\frac{1}{2} \int_\Omega | \nabla u_0 |^2 
	\dx 
	+ \int_\Omega \widehat{\beta} (u_0) 
	\dx 
	+ L \|u_0\|^2_H + L  \| u_\kappa (t)\|^2_H +2 C_L 
	\notag \\
	& \quad {} 
	+
	\frac{\kappa}{2} \int_\Gamma | \nabla_\Gamma u_{0\Gamma}|^2 
	\dg 
	+\int_\Gamma \widehat{\beta}_{\Gamma} ( u_{0\Gamma}) 
	\dg  
	+
	L_\Gamma\|u_{0\Gamma}\|_{H_\Gamma}^2
	+
	L_\Gamma\|u_{\Gamma,\kappa} (t)\|_{H_\Gamma}^2 + 2 C_{\rm L}
	\notag \\
	& \quad {}
	+\frac{1}{2}
	\int_0^t \! \! \int_\Omega \|f\|^2 
	\dx \ds 
	+
	\int_0^t \| \partial _t f_\Gamma \|_{H_\Gamma} \|u_{\Gamma,\kappa} \|_{H_\Gamma}
	\ds 
	\notag \\
	& \quad {}
	+ \| f_\Gamma \|_{L^\infty(0,T; H_\Gamma)}\bigl( \| u_{\Gamma,\kappa} (t) \bigr\|_{H_\Gamma} +
   \| u_{0\Gamma} \|_{H_\Gamma} \bigr)
	\notag \\
	& \le 
	C + C  \bigl\| u_\kappa(t) \bigr\|_H^2 
	+ C \bigl\| u_{\Gamma,\kappa}(t) \bigr\|_{H_\Gamma}^2 
	+ C_{\rm tr} \int_0^t \| \partial _t f_\Gamma \|_{H_\Gamma} \|u_{\kappa} \|_{V}
	\ds
	\label{5th3-1}
\end{align}
for all $t \in [0,T]$,  where in the last inequality we have used the assumptions (A4), (A5) and the inequality~\eqref{statre1}. 
Let us discuss the treatment of the terms in the right-hand side. Note that
\begin{align*}
	&C \bigl\| u_\kappa(t) \bigr\|_H^2 
     = C \biggl(\int_0^t 2 ( 
	\partial_t u_\kappa, u_\kappa)_H \ds + \|u_0\|_H^2 \biggr)
	\notag \\
	& \le {\delta} \int_0^t 
	\| \partial_t u_\kappa\|_H^2 \ds 
	+ C_\delta
	\int_0^t \| u_\kappa \|_H^2 \ds
	+ C
\end{align*}
for all $t \in [0,T]$ and some $\delta >0$. In addition, using \eqref{comp2} we can infer that 
\begin{align*}
	 &C \bigl\| u_{\Gamma,\kappa}(t) \bigr\|_{H_\Gamma}^2 \leq 
	\delta \bigl\| u_\kappa(t) \bigr\|_V^2 + C_\delta \bigl\| u_\kappa(t) \bigr\|_H^2 
     \\ 
	& \le \delta \bigl\| u_\kappa(t) \bigr\|_V^2 + {\delta} \int_0^t 
	\| \partial_t u_\kappa\|_H^2 \ds 
	+ C_\delta
	\int_0^t \| u_\kappa \|_H^2 \ds 
	+ C_\delta .
\end{align*}
Hence, choosing $\delta$ small enough, from \eqref{5th3-1} it is straightforward to obtain in particular 
that
\begin{equation*}
	\bigl\| u_\kappa (t) \bigr\|_{V}^2
	\le M_1' \left( 1+ \int_0^t \| u_\kappa \|_V^2 
	\ds + \int_0^t \| \partial _t f_\Gamma \|_{H_\Gamma} \|u_{\kappa} \|_{V}
	\ds
	\right)
\end{equation*}
for all $t \in [0,T]$, 
where $M_1'>0$ is a constant independent of $\tau, \lambda, \varepsilon$, and $\kappa$. 
%Now put
%\begin{equation*}
%	\Phi(t) := 
%	\bigl\| u_\kappa (t) \bigr\|_V 
%\end{equation*}
%for $t \in [0,T]$, which is the function in $C([0,T])$. 
Now, as from the assumption (A5) we have that $ \|\partial _t f_\Gamma(\cdot)\|_{H_\Gamma} \in L^1(0,T)$, by applying a combination of the two Gronwall lemmas reported in \cite[Appendix, \takeshi{pp.~156--157}]{Bre73}, we find that \takeshi{$\| u_\kappa \|_V$} is uniformly bounded in $L^\infty (0,T)$. Consequently, observing that (cf.~\eqref{ek4lam}) $\|u_{\Gamma,\kappa}\|_{L^\infty(0,T;Z_\Gamma)}$ is uniformly bounded as well and  using again \eqref{5th3-1}, we easily conclude the proof of the lemma.}
\end{proof}
\smallskip

The role of the approximation by $\tau>0$ \pier{was that of guaranteeing the regularity of solutions in order to prove the above lemma in a rigorous way. Now, based on the results of~\cite{CF20}, we know that letting $\tau \to 0$ and keeping  $\lambda, \varepsilon,\kappa \in (0,1]$ fixed, we obtain the limit problem on which we can perform the next estimates (cf.~Lemmas~\ref{L3}--\ref{L6}) directly. Let us recall the limit problem with  $\lambda, \varepsilon,\kappa \in (0,1]$:}
\begin{align}
   \partial_t u_\kappa - \Delta u_\kappa + \beta_\lambda (u_\kappa) +\pi(u_\kappa) = f 
   \quad {\rm a.e.\ in~} Q, \label{ek3lamb}\\
   (u_\kappa)_{|_\Gamma} = u_{\Gamma,\kappa} 
   \quad {\rm a.e.\ on~} \Sigma, \label{ek4lamb}\\ 
    (\mu_\kappa)_{|_\Gamma} =\mu_{\Gamma,\kappa}
   \quad {\rm a.e.\ on~} \Sigma, \label{ek2lamb}\\
%\end{align}
\noalign{\smallskip}
%\begin{align}
	\bigl\langle \partial _t u_{\Gamma,\kappa}(t),z_\Gamma \bigr\rangle_{V_\Gamma',V_\Gamma}
	+ \varepsilon \int_\Omega \nabla \mu_\kappa (t) \cdot \nabla z \dx 
	+ \int_\Gamma \nabla _\Gamma \mu_{\Gamma,\kappa}(t) \cdot \nabla_\Gamma z_\Gamma \dg
	=0 
	\notag
	\\
	\text{for all } (z,z_\Gamma) \in \boldsymbol{V}, 
	\quad \text{for~a.a.\ } t \in (0,T),
	\label{ek1lamb}\\
%\end{align}
\noalign{\smallskip}
%\begin{align}
  \mu_{\Gamma,\kappa} = 
  \partial_{\boldsymbol{\nu}} u_\kappa - \kappa \Delta _\Gamma u_{\Gamma,\kappa} 
  + \beta_{\Gamma, \lambda} (u_{\Gamma,\kappa}) + \pi_\Gamma(u_{\Gamma,\kappa}) - f_\Gamma
  \quad {\rm a.e.\ on~} \Sigma, \label{ek6lamb}
  \\
   u_\kappa(0) = u_0  \quad {\rm a.e.\ in~} \Omega,\label{ek7lamb}
   \\
  u_{\Gamma,\kappa} (0) = u_{0\Gamma} 
    \quad {\rm a.e.\ on~} \Gamma.\label{ek8lamb}
\end{align}
\pier{Of course, for the solution to~\eqref{ek3lamb}--\eqref{ek8lamb} the estimates stated in Lemma~\ref{FE} still hold. Note however that the regularity of $u_{\Gamma,\kappa}$ is here replaced 
by (cf.~Proposition~\ref{ANA})}
\begin{equation*}
	u_{\Gamma,\kappa} \in H^1(0,T;V_\Gamma') \cap L^\infty(0,T;V_\Gamma) \cap L^2(0,T;W_\Gamma).
\end{equation*}

\begin{lemma}
\label{L3}
There exists a positive constant $M_2$, independent of $\lambda, \varepsilon$, and $\kappa$, such that
\begin{align*}
	\|\partial_t u_{\Gamma, \kappa} \|_{L^2(0,T;V_\Gamma')} \le M_2. 
\end{align*}
\end{lemma}

\begin{proof}
Taking an \pier{arbitrary function $\zeta _\Gamma \in L^2(0,T;V_\Gamma)$, we choose
$(z,z_\Gamma)=({\mathcal R} \zeta_\Gamma(s), \zeta_\Gamma(s))$ as test function in \eqref{ek1lamb}, 
where ${\mathcal R}:Z_\Gamma \to V$ is the recovering operator specified by \eqref{recov} and it satisfies the estimate 
\begin{equation}
	\| {\mathcal R} z_\Gamma \|_{V} \le C_{\mathcal R} \| z_\Gamma \|_{Z_\Gamma} 
	\quad {\rm for~all~} z_\Gamma \in Z_\Gamma,
	\label{recover}
\end{equation} 
for some constant $C_{\mathcal R}>0$.
Now, integrating the resultant over $(0,T)$ with respect to the time variable $s$, and using Lemma~\eqref{FE} we obtain
\begin{align*}
	\left| \int_0^T \langle \partial _t u_{\Gamma, \kappa}, 
	\zeta_\Gamma \rangle_{V'_\Gamma,V_\Gamma} \ds 
	\right|
	& \le 
	\varepsilon \int_0^T \!\! \int_\Omega 
	|\nabla \mu_\kappa || \nabla {\mathcal R} \zeta_\Gamma| \dx \ds 
	+ \int_0^T \!\! \int_\Gamma 
	| \nabla _\Gamma \mu_{\Gamma,\kappa}| |\nabla _\Gamma \zeta_\Gamma| \dg \ds 
	\notag \\
	& \le \sqrt{\varepsilon} M_1 C_{\mathcal R} \|\zeta_\Gamma\|_{L^2(0,T;Z_\Gamma)}
	+  M_1 \|\zeta_\Gamma\|_{L^2(0,T;V_\Gamma)} 
	\notag \\
	& \le M_2 \|\zeta_\Gamma\|_{L^2(0,T;V_\Gamma)},
\end{align*}
where $M_2$ is a positive constant independent of $\lambda, \varepsilon$, and $\kappa$.
The proof is complete.} 
\end{proof}
\smallskip

\begin{lemma}
\label{L4}
There exist two positive constants $M_3$ and $M_4$, independent of $\lambda, \varepsilon$, and $\kappa$, such that
\begin{gather*}
	\bigl\| \beta_\lambda (u_\kappa) \bigr\|_{L^2(0,T;L^1(\Omega))} + 
	\bigl\| \beta_{\Gamma,\lambda} (u_{\Gamma,\kappa}) \bigr\|_{L^2(0,T;L^1(\Gamma))} \le M_3, 
	\notag \\
	\pier{\sqrt{\varepsilon} \|\mu_\kappa\| _{L^2(0,T;V)} + \|\mu_{\Gamma, \kappa} \|_{L^2(0,T;V_\Gamma)}}  \le M_4. 
\end{gather*}
\end{lemma}

\begin{proof}
We test \eqref{ek6lamb} by $u_{\Gamma,\kappa}-m_\Gamma$\pier{, where $m_\Gamma$ is defined in (A4),} and recover 
\begin{align}
	& \int_\Gamma \partial _{\boldsymbol{\nu}} u_\kappa (u_{\Gamma,\kappa} -m_\Gamma) \dg
	+ \kappa \int_\Gamma |\nabla _\Gamma u_{\Gamma,\kappa} |^2 \dg + 
	\int_\Gamma \beta_{\Gamma,\kappa}(u_{\Gamma,\kappa}) (u_{\Gamma,\kappa}-m_\Gamma) \dg 
	\notag \\
	& \quad {}+\int_\Gamma \bigl( \pi_\Gamma (u_{\Gamma,\kappa} )
	-f_\Gamma \bigr) (u_{\Gamma,\kappa}-m_\Gamma) \dg  
	= \int _\Gamma \mu_{\Gamma,\kappa} (u_{\Gamma,\kappa}-m_\Gamma) \dg  
	\label{m1}
\end{align}
a.e.\ in $(0,T)$. 
Now, thanks to \eqref{ek3lamb} and \eqref{ek4lamb}, we have that
\begin{align}
	& \int_\Gamma \partial _{\boldsymbol{\nu}} u_\kappa (u_{\Gamma,\kappa} -m_\Gamma ) \dg \notag \\ 
	& = \int_\Omega \Delta u_\kappa (u_\kappa-m_\Gamma) \dx 
	+ \int _\Omega |\nabla u_\kappa |^2 \dx 
	\notag \\
	& = \int_\Omega \bigl( \partial_t u_\kappa +\beta_\lambda (u_\kappa) + \pi(u_\kappa)-f \bigr) (u_\kappa -m_\Gamma ) \dx
	+ \int _\Omega |\nabla u_\kappa |^2 \dx 
	\label{m2}
\end{align}
a.e.\ in $(0,T)$. From \eqref{ek1lamb} and the assumption (A4) it is easy check that 
\begin{equation}
	\langle u_{\Gamma, \kappa}-m_\Gamma, 1 \rangle_{V_\Gamma',V_\Gamma} =\int_\Gamma (u_{\Gamma, \kappa}-m_\Gamma)\dg = 0
	\label{nec}
\end{equation}
in $(0,T)$. 
Here, we denote by $(y_\varepsilon, y_{\Gamma,\varepsilon}) \in \pier{L^2(0,T;\boldsymbol{V})}$, the solution to \pier{the variational equality
\begin{equation}
	\varepsilon \int_\Omega \nabla y_\varepsilon \cdot \nabla z \dx 
	+ \int_\Gamma \nabla _\Gamma y_{\Gamma, \varepsilon} \cdot \nabla _\Gamma z_\Gamma \dg
	= 
	\int_\Gamma (u_{\Gamma,\kappa} - m_\Gamma ) z_\Gamma \dg
	\label{aux}
\end{equation}
for all $(z,z_\Gamma) \in \boldsymbol{V}$, complemented with
$\int_\Gamma y_{\Gamma,\varepsilon} \dg=0$, almost everywhere in $(0,T)$. We underline that
the condition \eqref{nec} is necessary to solve \eqref{aux}. 
Taking $(z,z_\Gamma)=(y_\varepsilon, y_{\Gamma,\varepsilon})$ in \eqref{aux}, and 
using Poincar\'e inequalities~\eqref{Poin1} and~\eqref{Poin2}
we find out that} 
\begin{align*}
	\varepsilon \|\nabla y_\varepsilon\|_H^2 + \|\nabla _\Gamma y_{\Gamma,\varepsilon}\|_{H_\Gamma}^2 
	& \le 
	\|y_{\Gamma,\varepsilon}\|_{H_\Gamma}\|u_{\Gamma,\kappa}-m_\Gamma \|_{H_\Gamma} 
	\notag \\
	& \le 
	\sqrt{C_{\rm P}} \|\nabla _\Gamma y_{\Gamma,\varepsilon}\|_{H_\Gamma}
	\|u_{\Gamma,\kappa}-m_\Gamma \|_{H_\Gamma} 
	\notag \\
	& \le \frac{\delta}{2} \|\nabla _\Gamma y_{\Gamma,\varepsilon}\|_{H_\Gamma}^2 
	+ \frac{C_{\rm P}}{2\delta} \|u_{\Gamma,\kappa}-m_\Gamma \|_{H_\Gamma} ^2,
\end{align*}
for all $\delta>0$, 
that is, there exists a positive constant $C_{\rm P}'$ depends only on $C_{\rm P}$ such that 
\begin{align*}
	\varepsilon \|y_\varepsilon\|_V^2 + \|y_{\Gamma,\varepsilon}\|_{V_\Gamma}^2 
	& \le C_{\rm P}' \|u_{\Gamma,\kappa}-m_\Gamma \|_{H_\Gamma} ^2.
\end{align*}
Now, we take $(z,z_\Gamma):=(\mu_\kappa, \mu_{\Gamma,\kappa})$ in \eqref{aux} and use \eqref{ek1lamb} 
\begin{align*}
	 \int_\Gamma (u_{\Gamma,\kappa} - m_\Gamma ) \mu_{\Gamma,\kappa} \dg
	&= 
	\varepsilon \int_\Omega \nabla y_\varepsilon \cdot \nabla \mu_\kappa \dx 
	+ \int_\Gamma \nabla _\Gamma y_{\Gamma, \varepsilon} \cdot \nabla _\Gamma \mu_{\Gamma,\kappa} \dg
	\\
	&= - \pier{\bigl\langle \partial _t u_{\Gamma,\kappa}(t), y_{\Gamma, \varepsilon}\bigr\rangle_{V_\Gamma',V_\Gamma}}
\end{align*}
and last term is under control by 
\begin{equation*}
	\|\partial _t u_{\Gamma,\kappa} \|_{V_\Gamma'} \| y_{\Gamma,\varepsilon}\|_{V_\Gamma} \le  
	\sqrt{ C_{\rm P}' } \|\partial _t u_{\Gamma,\kappa} \|_{V_\Gamma'} \|u_{\Gamma,\kappa}-m_\Gamma \|_{H_\Gamma}.
\end{equation*}
Merging \eqref{m1} and \eqref{m2}, \pier{and using the above inequality, 
% the assumption \eqref{ccond} of (A2), 
it turns out that there exist
some positive constants $\delta_0$ and $M_3'$, independent of $\lambda, \varepsilon$, and $\kappa$, such that
\begin{align}
	& \int_\Omega |\nabla u_\kappa |^2 \dx 
	+ \delta_0 \int_\Omega \bigl| \beta(u_\kappa) \bigr| \dx + 
	\kappa \int_\Gamma |\nabla _\Gamma u_{\Gamma,\kappa} |^2 \dg
	+ \delta_0 \int_\Gamma \bigl| \beta_\Gamma (u_{\Gamma,\kappa}) \bigr| \dg
	\notag \\
	& \le M_3' 
	+ \bigl\| \partial _t u_\kappa +\pi(u_\kappa) -f \bigr\|_H 
	\|u_\kappa -m_\Gamma\|_H
	+ \bigl\| \pi_\Gamma (u_{\Gamma,\kappa}) -f_\Gamma \bigr\|_{H_\Gamma} 
	\|u_{\Gamma,\kappa} -m_\Gamma\|_{H_\Gamma}
	\notag \\
	& \quad {}+
	\sqrt{ C_{\rm P}' } \|\partial _t u_{\Gamma,\kappa} \|_{V_\Gamma'} 
	\|u_{\Gamma,\kappa}-m_\Gamma \|_{H_\Gamma}
	\label{gms}
\end{align}
a.e.\ in $(0,T)$. 
In the above computation, we exploited a useful inequality, whose proof can be found e.g.\ 
in~\cite[p.~908]{GMS09}, asserting that there are two positive constants $\delta_0$ and $c_1$ such that 
\begin{equation}
	\beta_\lambda (r) (r-m_\Gamma ) \ge \delta_0 \bigl| \beta_\lambda (r) \bigr|-c_1, \quad 
	\beta_{\Gamma,\lambda} (r) (r-m_\Gamma ) \ge \delta_0 \bigl| \beta_{\Gamma,\lambda} (r) \bigr|-c_1
\label{pier3}
\end{equation}
for all $r \in \mathbb{R}$. For the validity of \eqref{pier3} one needs that the value 
$m_\Gamma $ belongs to the interior of both domains 
$D(\beta_\Gamma)$ and $D(\beta)$ (see (A2) and (A4)).}

\smallskip

\pier{About \eqref{gms}, we notice that the right-hand side is uniformly bounded in $L^2(0,T)$ due to Lemmas~\ref{FE} and~\ref{L3}. Then we square both sides of \eqref{gms} and, in view of the estimates already proved, we deduce that
\begin{equation*}
	\bigl\| \beta_\lambda (u_\kappa) \bigr\|_{L^2(0,T;L^1(\Omega))} + 
	\bigl\| \beta_{\Gamma,\lambda} (u_{\Gamma,\kappa}) \bigr\|_{L^2(0,T;L^1(\Gamma))} \le M_3
\end{equation*} 
for some positive constant $M_3$.  
Next, we observe that combining \eqref{ek3lamb} and \eqref{ek6lamb} tested by the constant function $1$ 
and squaring lead to  
\begin{align*}
	\left| \int_\Gamma \mu_{\Gamma,\kappa} \dg \right|^2
	& \le C \|\partial_t u_\kappa \|_{L^1(\Omega)}^2
	+ C\bigl\| \beta_\lambda (u_\kappa)\bigr\|_{L^1(\Omega)}^2
	+ C\bigl\| \pi(u_\kappa)-f \bigr\|_{L^1(\Omega)}^2
	\notag \\
	& \quad {}
	+ C\bigl\| \beta_{\Gamma,\lambda} (u_{\Gamma,\kappa})\bigr\|_{L^1(\Gamma)}^2
	+ C\bigl\| \pi_\Gamma (u_{\Gamma,\kappa})-f_\Gamma \bigr\|_{L^1(\Gamma)}^2.
\end{align*}
Thus, in view of Lemma~\ref{FE} and the Poincar\'e type inequalities~\eqref{Poin0} and~\eqref{Poin2}, we easily deduce that also the second estimate in the statement of the lemma holds.}
\end{proof}
\smallskip

\begin{lemma}
\label{L5}
There exists a positive constant $M_5$, independent of $\lambda, \varepsilon$, and $\kappa$, such that
\begin{gather}
	\bigl\| \beta_\lambda (u_\kappa) \bigr\|_{L^2(0,T;H)} 
	+ 
	\bigl\| \beta_\lambda (u_{\Gamma,\kappa}) \bigr\|_{L^2(0,T;H_\Gamma)} \le M_5, 
	\notag \\
	\|\Delta u_\kappa\|_{L^2(0,T;H)} + 
	\|\partial_{\boldsymbol{\nu}} u_\kappa \|_{L^2(0,T;Z_\Gamma')} \le M_5.
	\notag 
\end{gather}
\end{lemma}

\begin{proof}
We test \eqref{ek3lamb} by $\beta_\lambda(u_\kappa(t)) \in V$ and obtain, with the help of \eqref{ek6lamb},
\begin{align}
	& \int_\Omega \bigl| \beta_\lambda (u_\kappa) \bigr|^2 \dx
	+ \int_\Gamma \beta_{\Gamma,\lambda} 
	(u_{\Gamma, \kappa}) \beta_\lambda (u_{\Gamma, \kappa}) \dg
	\notag \\
	& \quad {}+ \int_\Omega \beta'_\lambda(u_\kappa) |\nabla u_\kappa|^2 \dx
	+ \kappa \int_\Gamma \beta_\lambda'(u_{\Gamma,\kappa})
	| \nabla _\Gamma u_{\Gamma, \kappa} |^2 \dg
	\notag \\
	& \le \bigl\| f - \partial _t u_\kappa -\pi(u_\kappa) \bigr\|_H 
	\bigl\| \beta_\lambda (u_\kappa) \bigr\|_H 
	+ \bigl\| f_\Gamma -\pi_\Gamma (u_{\Gamma,\kappa}) -\mu_{\Gamma,\kappa} \bigr\|_{H_\Gamma} 
	\bigl\| \beta_\lambda (u_{\Gamma,\kappa}) \bigr\|_{H_\Gamma} \label{pier4}
\end{align}
a.e.\ in $(0,T)$, 
where we take care of the fact $(\beta_\lambda (u_\kappa))_{|_\Gamma}=\beta_\lambda (u_{\Gamma,\kappa}) \ne \beta_{\Gamma,\lambda} (u_{\Gamma,\kappa}) $ a.e.\ on $\Gamma$. \pier{Now, let us recall assumption (A2) and, in particular, the condition \eqref{ccond}: in view of \cite[Lemma~A.1]{CF20b}, we have that the same estimate holds for the Yosida approximations:
\begin{equation}
	\bigl| \beta_\lambda (r) \bigr| \le \varrho \bigl| \beta_{\Gamma, \lambda} (r) \bigr|+ c_0
	\quad \hbox{for all $r \in \mathbb{R}$ and $\lambda \in (0,1]$}.
	\label{same}
\end{equation}
Hence, from \eqref{same} and Young's inequality it follows that}
\begin{align*}
	\int _\Gamma \beta_{\Gamma,\lambda}(u_{\Gamma,\kappa}) \beta_\lambda (u_{\Gamma, \kappa}) \dg
	& = \int_\Gamma \bigl|  \beta_{\Gamma,\lambda}(u_{\Gamma,\kappa}) \bigr| 
	\bigl| \beta_\lambda (u_{\Gamma, \kappa}) \bigr| \dg
	\notag \\
	& \ge \frac{1}{\varrho} 
	\int_\Gamma
	\bigl| \beta_\lambda(u_{\Gamma,\kappa}) \bigr|^2 \dg
	- \frac{c_0}{\varrho} 
	\int_\Gamma
	\bigl| \beta_{\lambda}(u_{\Gamma,\kappa}) \bigr| \dg
	\notag \\
	& \ge 
	\frac{1}{2\varrho} 
	\int_\Gamma
	\bigl| \beta_\lambda(u_{\Gamma,\kappa}) \bigr|^2 \dg
	- \frac{c_0^2}{2\varrho}|\Gamma|
\end{align*}
a.e.\ in $(0,T)$\pier{, where $|\Gamma|$ denotes the surface measure of $\Gamma $.
We can use this inequality in the left-hand side of~\eqref{pier4}, observe that the third and fourth terms
in~\eqref{pier4} are nonnegative by monotonicity, and estimate the terms on the right-hand side of~\eqref{pier4} by the Young inequality. Then, on account of Lemmas~\ref{FE} and~\ref{L4}, it is straightforward to conclude that
\begin{equation*}
	\bigl\| \beta_\lambda (u_\kappa) \bigr\|_{L^2(0,T;H)} 
	+ 
	\bigl\| \beta_\lambda (u_{\Gamma,\kappa}) \bigr\|_{L^2(0,T;H_\Gamma)} \le C
\end{equation*}
for some positive constant independent of $\lambda, \varepsilon$, and $\kappa$. 
Now, from a comparison of terms in \eqref{ek3lamb} it turns out that
\begin{equation}
	\| \Delta u_\kappa \|_{L^2(0,T;H)} \le C.
	\label{estlap}
\end{equation}
Combining this with the estimate of $\|u_\kappa\|_{L^\infty(0,T;V)}$ obtained in Lemma~\ref{FE}, 
and thanks to~\eqref{normal1}, we deduce that
\begin{equation}
	\|\partial _{\boldsymbol{\nu}} u_\kappa\|_{L^2(0,T;Z_\Gamma')} \le C. 
	\label{pier7} 
\end{equation}
Therefore, the lemma is completely proved.}
\end{proof}
\smallskip

\begin{lemma}
\label{L6}
There exists a positive constant $M_6$, independent of $\lambda, \varepsilon$, and $\kappa$, such that
\begin{gather*}
	\pier{\bigl\| 
	-\kappa \Delta _\Gamma u_{\Gamma,\kappa} + \beta_{\Gamma,\lambda}(u_{\Gamma,\kappa})
	\bigr\| _{L^2(0,T;Z_\Gamma')} \le M_6,}
	\\
	\pier{\sqrt{\kappa} \|u_\kappa\| _{L^2(0,T;H^{3/2}(\Omega))} +
	\sqrt{\kappa} \|\partial_{\boldsymbol{\nu}} u_\kappa\| _{L^2(0,T;H_\Gamma)} \le M_6,}
	\\
	\pier{\sqrt{\kappa} \|\beta_{\Gamma, \lambda} (u_{\Gamma,\kappa}) \|_{L^2(0,T;H_\Gamma)} 
	+ {\kappa}^{3/2} \|\Delta_\Gamma u_{\Gamma,\kappa} \|_{L^2(0,T;H_\Gamma)} \le M_6,}
	\\
	\sqrt{\kappa} \| 
	\Delta _\Gamma u_{\Gamma,\kappa}
	\|_{L^\infty(0,T;V_\Gamma')} + 
	\bigl\| \beta_{\Gamma,\lambda}(u_{\Gamma,\kappa})
	\bigr\| _{L^2(0,T;V_\Gamma')} \le M_6.
\end{gather*}%
\end{lemma}

\begin{proof}
\pier{In view of Lemmas~\ref{L4} and~\ref{L5}, a comparison of terms in~\eqref{ek6lamb} yields
\begin{equation}
\bigl\| -\kappa \Delta _\Gamma u_{\Gamma,\kappa} + \beta_{\Gamma,\lambda}(u_{\Gamma,\kappa})
	\bigr\| _{L^2(0,T;Z_\Gamma')} \le C.
\label{lastest}
\end{equation}
Next, owing to \eqref{estlap} and to the estimate of 
$\sqrt{\kappa}\|u_{\Gamma,\kappa}\|_{L^\infty(0,T;V_\Gamma)}$ (see Lemma~\ref{FE}), we can invoke the embedding inequalities~\eqref{LM1} and~\eqref{normal2} to deduce that 
\begin{equation*}
	\sqrt{\kappa} \|u_\kappa\|_{L^2(0,T;H^{3/2}(\Omega))} + 
	\sqrt{\kappa} \|\partial _{\boldsymbol{\nu}} u_\kappa\|_{L^2(0,T;H_\Gamma)} \le C. 
\end{equation*}
Moreover, we can test \eqref{ek6lamb} by $\kappa \beta_{\Gamma, \lambda} (u_{\Gamma,\kappa})$ 
and integrate by parts to find that 
\begin{align}
&\kappa^2 \int _\Gamma  \beta_{\Gamma, \lambda}' (u_{\Gamma,\kappa})
  |\nabla_\Gamma u_{\Gamma,\kappa}|^2 \dg 
    + \kappa \takeshi{\bigl \|} \beta_{\Gamma, \lambda} (u_{\Gamma,\kappa}) \takeshi{ \bigr \|}_{H_\Gamma}^2
\notag \\
&   =
\kappa \int _\Gamma \takeshi{ \bigl(} 
\mu_{\Gamma,\kappa} - \partial_{\boldsymbol{\nu}} u_\kappa - \pi_\Gamma(u_{\Gamma,\kappa}) - f_\Gamma
\takeshi{ \bigr)} \beta_{\Gamma, \lambda} (u_{\Gamma,\kappa}) \dg 
\notag \\
& \le \frac \kappa 2 \takeshi{\bigl\|} \beta_{\Gamma, \lambda} (u_{\Gamma,\kappa}) 
\takeshi{\bigr \|}_{H_\Gamma} +
C \kappa \| \partial _{\boldsymbol{\nu}} u_\kappa \|_{H_\Gamma}^2 
+ C \takeshi{\bigl \|} \mu_{\Gamma,\kappa}  - \pi_\Gamma(u_{\Gamma,\kappa}) - f_\Gamma \takeshi{\bigr \|}^2_{H_\Gamma}  
\label{pier6}
\end{align}
a.e.\ on~$(0,T)$. Then, integrating the resultant of \eqref{pier6} over (0,T), and 
accounting for Lemmas~\ref{FE} and~\ref{L4}, we easily infer that
\begin{equation*}
	\sqrt{\kappa} \takeshi{\bigl \|} 
	\beta_{\Gamma, \lambda} (u_{\Gamma,\kappa}) 
	\takeshi{\bigr \|}_{L^2(0,T;H_\Gamma)} \le C,
\end{equation*}
which also implies, by comparison of terms in \eqref{ek6lamb}, that
\begin{equation*}
	{\kappa}^{3/2} \|\Delta_\Gamma u_{\Gamma,\kappa} \|_{L^2(0,T;H_\Gamma)} \le C. 
\end{equation*}
At this point, note that (the natural extension of) the Laplace--Beltrami operator $- \Delta_\Gamma$ is linear and bounded from $V_\Gamma$ to $V_\Gamma'$. Hence, recalling Lemma~\ref{FE} as well, there exists a constant $C_{\rm D}>0$ such that 
\begin{equation*}
	\sqrt{\kappa} \|\Delta _\Gamma u_{\Gamma,\kappa} \|_{L^\infty(0,T;V_\Gamma')} \le 
	\sqrt{\kappa} C_{\rm D}
	 \| u_{\Gamma,\kappa}\|_{L^\infty(0,T;V_\Gamma)} 
	 \le C_{\rm D} M_1.
\end{equation*}
Therefore, in view of~\eqref{lastest} we deduce that 
\begin{align*}
	&\bigl\| \beta_{\Gamma,\lambda}(u_{\Gamma,\kappa})
	\bigr\| _{L^2(0,T;V_\Gamma')}
	\le \bigl\| \beta_{\Gamma,\lambda}(u_{\Gamma,\kappa})
	-\kappa \Delta_\Gamma u_{\Gamma,\kappa}
	\bigr\| _{L^2(0,T;V_\Gamma')}
	+\kappa\,\|  \Delta_\Gamma u_{\Gamma,\kappa}
	\| _{L^2(0,T;V_\Gamma')}
	\notag \\
	& \le C\bigl\| \beta_{\Gamma,\lambda}(u_{\Gamma,\kappa})
	-\kappa \Delta_\Gamma u_{\Gamma,\kappa}
	\bigr\| _{L^2(0,T;Z_\Gamma')}
	+ C \sqrt{\kappa} \| 
	\Delta _\Gamma u_{\Gamma,\kappa}\|_{L^\infty(0,T;V_\Gamma')}
	\le C.
\end{align*}
Thus, we arrive at the conclusion.}
\end{proof}

\subsection{Proof of the Theorem 3.1.}
\pier{Let us recall the previously established well-posedness result~\cite[Theorems 2.3, 2.4]{CF20}, which 
pertains to the limiting case as $\lambda \to 0$ while keeping \takeshi{$\varepsilon, \kappa>0$} fixed. 
The well-posedness of $({\rm P})_{\varepsilon\kappa}$ is already known. Accordingly, we interpret the 
family $\{(u_\kappa,\mu_\kappa, \xi_\kappa, u_{\Gamma,\kappa},$ $\mu_{\Gamma,\kappa}, \xi_{\Gamma,\kappa})
\}_{\kappa \in (0,1]}$ as solutions to $({\rm P})_{\varepsilon\kappa}$. In the light of Lemmas~\ref{FE}--\ref{L6} and accounting for the weak or weak star lower semicontinuity of norms, this family of solutions 
satisfies the estimate
\begin{align}
	&\|u_\kappa \|_{H^1(0,T;H)\cap L^\infty (0,T;V)} 
	+	\|\Delta u_\kappa\|_{L^2(0,T;H)} 
	+ 	\|\partial_{\boldsymbol{\nu}} u_\kappa \|_{L^2(0,T;Z_\Gamma')}
	\notag \\
	&{} 	
		+ \sqrt{\varepsilon} \|\mu_{\kappa}\|_{L^2(0,T;V)}
		+ \takeshi{ \|} \xi_\kappa \takeshi{\|}_{L^2(0,T;H)} 
	+ \| u_{\Gamma,\kappa}\|_{H^1(0,T;V_\Gamma')\cap L^\infty(0,T;Z_\Gamma)}
	\notag \\
	&{} 
		+ \sqrt{\kappa} \| u_{\Gamma,\kappa}\|_{L^\infty(0,T;V_\Gamma)}
		+ \| \mu_{\Gamma, \kappa}\|_{L^2(0,T;V_\Gamma)}
	+ \takeshi{\|} \xi_{\Gamma,\kappa} \takeshi{\|}_{L^2(0,T;V_\Gamma')}
	\notag \\
	&{}
	+ \sqrt{\kappa} \takeshi{\|} \xi_{\Gamma,\kappa} \takeshi{\|}_{L^2(0,T;H_\Gamma)}	
	+\takeshi{\|} 
	-\kappa \Delta _\Gamma u_{\Gamma,\kappa} + \xi_{\Gamma,\kappa}
	\takeshi{\|} _{L^2(0,T;Z_\Gamma')}
	\le C .
	\label{pier5}
\end{align}
Hereafter, we  consider the limiting procedure $\kappa \to 0$ keeping $\varepsilon>0$ fixed. 
Then we claim that there exists a \pier{sextuple} $(u, \mu, \xi, u_\Gamma, \mu_\Gamma, \xi_\Gamma)$ and 
 a subsequence $\{\kappa_k \}_{k \in \mathbb{N}}$ such that, as $k \to +\infty$, the convergences  $\kappa_k \to 0$ and 
\begin{align*}
	&u_{\kappa_k} \to u \quad \text{weakly~star~in~} H^1(0,T;H) \cap L^\infty (0,T;V),
	\\
	&\Delta u_{\kappa_k} \to \Delta u \quad \text{weakly~in~} L^2 (0,T;H),
	\\
	& \partial _{\boldsymbol{\nu}} u_{\kappa _k} \to \partial _{\boldsymbol{\nu}} u 
	\quad \text{weakly~in~} L^2 (0,T;Z_\Gamma'), 
	\\
	&\mu_{\kappa_k} \to \mu  \quad \text{weakly~in~} L^2 (0,T;V),
	\\
	&\xi_{\kappa_k} \to \xi \quad  \text{weakly~in~} L^2 (0,T;H),
	\\
	&u_{\Gamma,\kappa_k} \to u_\Gamma \quad \text{weakly~star~in~} H^1(0,T;V_\Gamma') \cap L^\infty (0,T;Z_\Gamma),
	\\
	&{\kappa_k} u_{\Gamma,\kappa_k} \to 0 \quad \text{strongly~in~} L^\infty (0,T;V_\Gamma),
	\\
	&\mu_{\Gamma,\kappa_k} \to \mu_\Gamma \quad \text{weakly~in~} L^2 (0,T;V_\Gamma),
	\\
	&\xi_{\Gamma,\kappa_k}\to \xi_\Gamma \quad \text{weakly~in~} L^2 (0,T;V_\Gamma'),
	\\
	&(-{\kappa_k} \Delta_\Gamma u_{\Gamma,\kappa_k} +\xi_{\Gamma,\kappa_k}) \to \xi_\Gamma \quad \text{weakly~in~} 
	L^2 (0,T;Z_\Gamma')
\end{align*}
hold.} 
%Here, we use the standard trace theory \eqref{statre1} to obtain the boundedness of $\{ u_{\Gamma, 
%\kappa_k} \}_{k \in \mathbb{N}}$ in $L^\infty(0,T;Z_\Gamma)$. 
\pier{Moreover, applying the compactness theorem in~\cite[Sect.~8, Cor.~4]{Sim87} and recalling the 
compact embeddings $V \hookrightarrow \hookrightarrow H$, $Z_\Gamma \hookrightarrow \hookrightarrow H_\Gamma$ and assumption (A3), we have that 
\begin{align*}
	&u_{\kappa_k} \to u, \quad \pi(u_{\kappa_k}) \to \pi(u)  \quad \text{strongly~in~} C\bigl([0,T];H \bigr),
	\\
	&u_{\Gamma,\kappa_k} \to u_\Gamma, \quad \pi_\Gamma(u_{\Gamma,\kappa_k}) \to \pi_\Gamma(u) \quad \text{strongly~in~} C\bigl([0,T];H_\Gamma \bigr)
\end{align*}
as $k \to \infty$. Note that now we have all the convergences stated in~\eqref{cv1k}--\eqref{cv11k}. 
It remains to prove that  $(u, \mu, \xi, u_\Gamma, \mu_\Gamma, \xi_\Gamma)$ solves $({\rm P})_{\varepsilon}$.
By the strong convergences above it is straightforward to pass to the limit in the initial conditions and obtain~\eqref{e5} and \eqref{e8}. In addition, the boundary conditions in \eqref{e2}, \eqref{e4} and the equation in \eqref{e3} follow from the weak and weak star convergences previously recalled. 
Thanks to the standard maximal monotone property of demi-closedness~\cite{Bar10}, from \eqref{cv2k} and 
\eqref{cv5k} we easily infer that 
\begin{equation*}
	\xi \in \beta (u) \quad \text{\rm a.e.\ in } Q
\end{equation*}
and this allows us to fully show~\eqref{e3}. 
Now, we can take the limit $k \to \infty$ in \eqref{ek1lamb} to deduce that 
\begin{equation}
	\langle \partial _t u_{\Gamma},z_\Gamma \rangle_{V_\Gamma',V_\Gamma}
	+ \varepsilon \int_\Omega \nabla \mu \cdot \nabla z \dx 
	+ \int_\Gamma \nabla _\Gamma \mu_{\Gamma} \cdot \nabla_\Gamma z_\Gamma \dg
	=0 
	\label{weak1}
\end{equation}
for all $(z,z_\Gamma) \in \boldsymbol{V}$, a.e.\ in $(0,T)$.
By taking $(z,0) \in \boldsymbol{V}$ with $z \in {\mathcal D}(\Omega)$ in \eqref{weak1}, 
we obtain $-\varepsilon\Delta \mu=0$ in ${\mathcal D}'(\Omega)$,  a.e.\ in $(0,T)$, with
the right-hand side $0$ that is clearly in $H$. Hence, $ \Delta \mu \in L^2(0,T;H)$ and 
\eqref{e1} follow.} 

\smallskip

Next, using the \pier{characterization of the normal derivative in~\eqref{normal} and \eqref{weak1}, we obtain a.e.\ in $(0,T)$ that
\begin{align*}
	 \varepsilon \langle \partial _{\boldsymbol{\nu}} \mu, z_\Gamma \rangle_{Z_\Gamma',Z_\Gamma} 
	& = \int_\Omega \varepsilon \Delta \mu \, {\mathcal R}z_\Gamma \dx 
	+ \varepsilon \int_\Omega \nabla \mu \cdot \nabla {\mathcal R}z_\Gamma \dx \\ 
	& = -\langle \partial _t u_\Gamma, z_\Gamma \rangle_{V_\Gamma',V_\Gamma} 
	- \int_\Gamma \nabla _\Gamma \mu_\Gamma \cdot \nabla_\Gamma z_\Gamma \dg
\end{align*}
for all $z _\Gamma \in V_\Gamma \subset Z_\Gamma$ because $({\mathcal R}z_\Gamma, z_\Gamma ) \in \boldsymbol{V}$. It is evident that the final equality directly implies~\eqref{e6weak}.}

\smallskip

\pier{At this point, we take an arbitrary pair $(z,z_\Gamma) \in \boldsymbol{Z}$ and test
\eqref{ek3} by $z$, then integrate by parts using the boundary equation~\eqref{ek6}. 
Then, letting $k \to +\infty$ and exploiting the convergence in~\eqref{cv11k}, we arrive at}
\begin{align}
	&\int_\Omega \partial_t u z \dx 
	+ 
	\int_\Omega \nabla u \cdot \nabla z\dx 
	+ 
	\int_\Omega \bigl( \xi + \pi(u) \bigr) z \dx 
	+
	\langle \xi_\Gamma,z_\Gamma \rangle_{Z_\Gamma',Z_\Gamma}
	+
	\int_\Gamma 
	\pi_\Gamma(u_\Gamma) z_\Gamma \dg
	\notag \\
	&{}
	= 
	\int_\Omega f  z \dx 
	+
	\int_\Gamma ( 
	f_\Gamma + \mu_\Gamma )z_\Gamma \dg
	\quad 
	\text{for all } (z,z_\Gamma) \in \boldsymbol{Z}, 
	\text{ a.e.\ in } (0,T).
	\label{main1-3}
\end{align}
\pier{Therefore, in view of the equation in \eqref{e3} and using~\eqref{normal} again, by a 
cancellation of the corresponding terms we infer that}
\begin{align*}
	\langle \partial _{\boldsymbol{\nu}} u, z_\Gamma \rangle_{Z_\Gamma',Z_\Gamma} 
	& = \int_\Omega \Delta u {\mathcal R}z_\Gamma \dx 
	+ \int_\Omega \nabla u \cdot \nabla {\mathcal R}z_\Gamma \dx \\ 
	& = -\langle \xi_\Gamma, z_\Gamma \rangle_{Z_\Gamma',Z_\Gamma} 
	- \int_\Gamma \bigl( \pi_\Gamma(u_\Gamma) - f_\Gamma - \mu_\Gamma \bigr) z_\Gamma \dg
\end{align*}
for all $z _\Gamma \in Z_\Gamma$, a.e.\ in $(0,T)$\pier{, which is nothing but the equality in~\eqref{e7weak}. In order to complete the proof of~\eqref{e7weak}, we multiply \eqref{ek3} by $u_{\kappa_k}$, integrating the resultant over $Q = \Omega \times (0,T)$ 
with respect to space and time variables. With the help of~\eqref{ek6} we have~that}
\begin{align*}
	&\iint _Q |\nabla u_{\kappa_k} |^2 \dx \dt
	+ {\kappa_k} \iint_\Sigma |\nabla_\Gamma u_{\kappa_k} |^2 \dg \dt
	+
	\iint _Q \xi_{\kappa_k} u_{\kappa_k} \dx \dt 
	+ 
	\iint_\Sigma \xi_{\Gamma, \kappa_k} u_{\Gamma,\kappa_k} \dg \dt
	\nonumber \\
	{}
	&=\iint _Q \bigl( f -\partial_t u_{\kappa_k} -\pi(u_{\kappa_k} )\bigr) u_{\kappa_k} \dx \dt
	+ \iint_\Sigma \bigl(f_\Gamma + \mu_{\Gamma, \kappa_k}  -\pi_\Gamma(u_{\Gamma, \kappa_k} )\bigr) u_{\Gamma,\kappa_k} \dg \dt,\notag
\end{align*}
\pier{where $\Sigma=\Gamma\times (0,T).$
Then, using the lower semicontinuity and the weak and strong convergence results obtained above, we deduce that}
\begin{align}
	& \limsup_{k \to +\infty} 
	\iint_\Sigma \xi_{\Gamma, \kappa_k} u_{\Gamma, \kappa_k} \dg \dt \notag \\
	& \le \limsup_{k \to +\infty} 
	 \iint _Q \bigl( f-\partial_t u_{\kappa_k}-\pi(u_{\kappa_k} ) \bigr) u_{\kappa_k} \dx \dt 
	\notag \\	 
	& {} + 
	\limsup_{k \to +\infty} 
	\iint_\Sigma \bigl(  f_\Gamma+\mu_{\Gamma,\kappa_k}  -\pi_\Gamma(u_{\Gamma,\kappa_k}) \bigr) u_{\Gamma,\kappa_k} \dg \dt 
 - \liminf_{k \to +\infty} \iint _Q |\nabla u_{\kappa_k} |^2 \dx \dt  
 	\notag \\
	& \quad {}
	- \liminf_{k \to +\infty} {\kappa_k} \iint_\Sigma |\nabla_\Gamma u_{\Gamma,\kappa_k} |^2 \dg \dt 
	- \liminf_{k \to +\infty} \iint _Q \xi_{\kappa_k} u_{\kappa_k} \dx \dt 
	\notag \\
	& \le \iint _Q \bigl(f -\partial_t u -\pi(u) \bigr) u \dx \dt + 
	\iint_\Sigma \bigl(f_\Gamma+ \mu_\Gamma  -\pi_\Gamma(u_\Gamma) \bigr) u_\Gamma \dg \dt 
	\notag \\
	& \quad {} - \iint _Q |\nabla u |^2 \dx \dt  
	- \iint _Q \xi u \dx \dt
%	\notag \\
%	& 
	= \int_0^T \langle \xi_\Gamma , u_{\Gamma} \rangle _{Z_\Gamma',Z_\Gamma} \dt,
	\label{last3}
\end{align}
where the last equality \pier{is a consequence of~\eqref{main1-3} when taking $(z,z_\Gamma)= (u,u_\Gamma)$. Now, on account of  the definition of subdifferential for $\beta_{\Gamma} $ in $L^2(0,T;H_\Gamma)\equiv L^2(\Sigma)$, we claim that 
\begin{equation}
	\iint_\Sigma \xi_{\kappa_k} (\zeta_\Gamma - u_{\Gamma,\kappa_k}) \dg \dt
+ \iint_\Sigma \widehat{\beta}_{\Gamma} (u_{\Gamma, \kappa_k}) \dg \dt 
	\le 
	\iint_\Sigma \widehat{\beta}_{\Gamma} (\zeta_\Gamma) \dg \dt 
	\label{liminf}
\end{equation} 
for all $\zeta_\Gamma \in L^2(0,T;H_\Gamma)$.
For  a while, let us take $\zeta_\Gamma \in L^2(0,T;V_\Gamma)$. 
In this case, from the weak convergence~\eqref{cv10k} we infer that} 
\begin{equation*}
	\lim _{k \to +\infty} \iint_\Sigma \xi_{\kappa_k} \zeta_\Gamma \dg \dt
	=\int_0^T \langle \xi_\Gamma, \zeta_\Gamma \rangle_{V_\Gamma',V_\Gamma} \dt
	= \int_0^T \langle \xi_\Gamma, \zeta_\Gamma \rangle_{Z_\Gamma',Z_\Gamma} \dt
\end{equation*}
\pier{since $V_\Gamma \subset Z_\Gamma$ and $\xi_\Gamma \in L^2(0,T;Z_\Gamma')$.
Moreover, from \eqref{last3} it follows that}
\begin{align*}
	\liminf_{k \to +\infty} \left( - \iint_\Sigma \xi_{\kappa_k} u_{\Gamma, \kappa_k} \dg 
\dt \right)
	&= - \limsup_{k \to +\infty} \iint_\Sigma \xi_{\kappa_k} u_{\Gamma, \kappa_k} \dg \dt \\
	&\ge 
	- \int_0^T \langle \xi_\Gamma, u_\Gamma \rangle_{Z_\Gamma',Z_\Gamma} \dt.
\end{align*}
\pier{Finally, using the lower semicontinuity of the extension of the convex function $ \widehat{\beta}_{\Gamma}$ to $L^2(\Sigma)$, we have that}
\begin{equation*}
	\iint_\Sigma \widehat{\beta}_{\Gamma}
	(u_{\Gamma})\dg \dt \le \liminf_{k \to +\infty} 
	\iint_\Sigma \widehat{\beta}_{\Gamma}
	(u_{\Gamma, \kappa_k}) \dg \dt. 
\end{equation*}
Therefore, taking the infimum limit in \eqref{liminf}, we deduce that
\begin{equation}
	\int_0^T \langle \xi_\Gamma, \zeta_\Gamma-u_\Gamma \rangle_{Z_\Gamma',Z_\Gamma } \dt 
	\pier{{}+	\iint_\Sigma \widehat{\beta}_\Gamma (u_\Gamma) \dg \dt }
	\le \iint_\Sigma \widehat{\beta}_\Gamma (\zeta_\Gamma) \dg \dt
	\label{sub}
\end{equation}
for all $\zeta_\Gamma \in L^2(0,T;V_\Gamma)$. Next, as $\xi_\Gamma \in L^2 (0,T;Z_\Gamma')$, 
by a density argument we can prove that \eqref{sub} holds also for all 
$\zeta_\Gamma \in L^2(0,T;Z_\Gamma)$. Indeed,
for a given arbitrary $\zeta_\Gamma \in L^2(0,T;Z_\Gamma)$, 
we can take the approximations $\{ \zeta_{\Gamma,n} \}_{n \in \mathbb{N}}$ in $L^2(0,T;V_\Gamma)$ defined as the solutions to  
\begin{equation*}
	\zeta_{\Gamma, n} - \frac{1}{n} \Delta_\Gamma \zeta_{\Gamma,n} =\zeta_\Gamma 
	\quad {\rm a.e.\ on~} \Sigma.  
\end{equation*}
In fact, thanks to \cite[Lemma~A.1]{CF16} we have that
\begin{gather*}
	\zeta_{\Gamma,n} \to \zeta_\Gamma \quad {\rm in~} L^2(0,T;Z_\Gamma) \quad {\rm as~} n \to +\infty, 
	\\
	\widehat{\beta}_\Gamma (\zeta_{\Gamma,n}) \le \widehat{\beta}_\Gamma (\zeta_\Gamma) \quad 
	{\rm a.e.\ on~} \Sigma, \hbox{ for~all~} n \in \mathbb{N}.
\end{gather*}
Thus, replacing $\zeta_\Gamma$ by $\zeta_{\Gamma,n}$ in \eqref{sub} and letting $n \to +\infty$, 
we obtain the validity of \eqref{sub} for all $\zeta_\Gamma \in L^2(0,T;Z_\Gamma)$, \pier{which is equivalent to the formulation in~\eqref{e7weak}.}
\hfill $\Box$ 

\begin{corollary}\label{cor1}
\pier{In the same framework of Theorem~\ref{thm1} and under the further assumption~{\rm (A6)}, 
the found \pier{sextuple} $(u,\mu, \xi, u_\Gamma, \mu_\Gamma, \xi_\Gamma)$ additionally fulfils 
\begin{equation*}
	u \in L^2\bigl( 0,T;H^{3/2}(\Omega) \bigr), \quad 
	\partial_{\boldsymbol{\nu}} u \in L^2(0,T;H_\Gamma), \quad 
	u_\Gamma \in L^2(0,T;V_\Gamma), \quad
	\xi_\Gamma \in L^2(0,T;H_\Gamma).
\end{equation*}
Moreover, the further convergence properties 
\begin{align}
	&\xi_{\Gamma,\kappa_k} \to \xi_\Gamma \quad \text{weakly~in~} L^2 (0,T;H_\Gamma), \label{pier8}
	\\
	&\partial _{\boldsymbol{\nu}} u_{\kappa _k} - {\kappa_k} \Delta_\Gamma u_{\Gamma,\kappa_k}  \to \partial _{\boldsymbol{\nu}} u \quad \text{weakly~in~} 
	L^2 (0,T;H_\Gamma)  \label{pier9}
\end{align} 
hold and the conditions in~\eqref{e7weak} can be equivalently formulated as~\eqref{e7}.}
\end{corollary}

\begin{proof}
The idea of the proof is essentially \pier{the same} as in \cite[Theorem 2.10]{CFS22} or \cite[Theorem 2.6]{CFS22b}.
\pier{Let us now briefly return to the derivation of uniform estimates for the approximating problem \eqref{ek3lamb}--\eqref{ek8lamb}. In light of the results presented in~\cite[Appendix]{CF20b}, it follows that the left-hand side inequality in assumption (A6) also holds for the Yosida approximations. Therefore,} 
\begin{equation*}
	\frac{1}{2C_{\beta}^2} \int_\Gamma \bigl| \beta_{\Gamma, \lambda} (u_{\Gamma, \kappa}) \bigr|^2 \dg
	\le \int_\Gamma \Bigl\{ \bigl| \beta _\lambda (u_{\Gamma,\kappa}) \bigr|^2+C_{\beta}^2  \Bigr\} \dg
\end{equation*}
a.e.\ in $(0,T)$. This implies with Lemma~\ref{L5} that 
\begin{equation*}
	\bigl\| \beta_{\Gamma,\lambda} (u_{\Gamma,\kappa}) \bigr\|_{L^2(0,T;H_\Gamma)}^2 \le 2C_{\beta}^2 \bigl( M_5^2+TC_{\beta}^2 |\Gamma| \bigr),
\end{equation*}
\pier{whence, taking $\lambda \to 0$, we deduce that} 
\begin{equation*}
	\| \xi_{\Gamma,\kappa} \|_{L^2(0,T;H_\Gamma)} \le \liminf _{\lambda \to 0} 
	\bigl\| \beta_{\Gamma,\lambda} (u_{\Gamma,\kappa}) \bigr\|_{L^2(0,T;H_\Gamma)}
	 \le \pier{C.} % C_{\beta} \bigl(2 M_5^2+2TC_{\beta}^2 |\Gamma| \bigr)^{1/2}.
\end{equation*}
From \pier{a comparison of terms in the equation~\eqref{ek6lamb}, and thanks to~\eqref{pier5},} 
we see that
\begin{equation}
	\| \partial _{\boldsymbol{\nu}} u_{\kappa} - \kappa \Delta_\Gamma u_{\Gamma,\kappa}
	\|_{L^2(0,T;H_\Gamma)} +  
	\| 
	-\kappa \Delta _\Gamma u_{\Gamma,\kappa} 
	\| _{L^2(0,T;Z_\Gamma')} \le C.
	\label{uniest01}
\end{equation}
Then, \pier{with respect to~\eqref{cv1k}--\eqref{cv11k}, we can infer the additional weak convergences \eqref{pier8}--\eqref{pier9}
as $k \to +\infty$. At this point, it suffices to use the regularity estimate~\eqref{LM2} to deduce that $u \in L^2 \bigl( 0,T;H^{3/2}(\Omega) \bigr)$ and the trace inequality~\eqref{statre1} to conclude that $u_\Gamma \in L^2(0,T;V_\Gamma)$.
Thanks to them, we obtain the equation in \eqref{e7} directly from \eqref{e7weak} and, due to the density of $Z_\Gamma$  in
$H_\Gamma$, we recover the inclusion in~\eqref{e7} as well.}
\end{proof}

\subsection{Second asymptotic result}

\pier{The argument concerning the limiting procedure as $\varepsilon \to 0$ with $\kappa >0$ is similar to that presented in the previous section. We remark that the target problem  $({\rm P})_{\kappa}$ corresponds to the Allen--Cahn equation with a dynamic boundary 
condition of Cahn--Hilliard type. This represents a rather novel model for dynamic boundary conditions. Indeed, while the Allen--Cahn equation with dynamic boundary conditions of heat or Allen--Cahn type has been extensively studied (see, e.g., \cite{CC13, GG08, 
Isr12}), and similarly, the Cahn--Hilliard equation with dynamic boundary conditions of heat, Allen--Cahn, or Cahn--Hilliard type has been addressed in the literature (see, e.g., \cite{CF15, CFW20, Gal06, GK20, GMS09, GMS11, KLLM21, LW19}), to the best of our knowledge, 
the Allen–Cahn equation in the bulk combined with a Cahn–Hilliard type dynamic boundary condition has not yet been investigated. A key point of interest is the difference in the order of the partial differential equations: the bulk equation is second order, whereas the boundary 
equation is fourth order with respect to the spatial variables.
In what follows, we address also the asymptotic analysis linking $({\rm P})_{\varepsilon\kappa}$ and $({\rm P})_{\kappa}$.}

\begin{theorem}\label{thm2}
\pier{Assume {\rm (A1)}--{\rm (A5)}. Then there exists a quintuple} $(u, \xi, u_\Gamma, \mu_\Gamma, \xi_\Gamma)$ fulfilling the regularity properties
\begin{gather*}
	u \in H^1(0,T;H) \cap L^\infty(0,T;V) \cap L^2(0,T;W),
	\quad 
	\xi \in L^2(0,T;H), \\
	u_\Gamma \in H^1(0,T;V_\Gamma') \cap C\bigl([0,T];H_\Gamma \bigr) \cap L^\infty (0,T;V_\Gamma) \cap L^2(0,T;W_\Gamma), \\
	\mu_\Gamma \in L^2(0,T;V_\Gamma), \quad \xi_\Gamma \in L^2(0,T;H_\Gamma)
\end{gather*}
and satisfying 
\pier{\eqref{e3}--\eqref{e5}, \eqref{e8}, \eqref{k7} and the equation \eqref{k6} in the following weak sense:}
\begin{equation}  
	\langle \partial_t u_{\Gamma},z_\Gamma \rangle_{V_\Gamma',V_\Gamma}
	+ \int_\Gamma \nabla _\Gamma \mu_{\Gamma} \cdot \nabla_\Gamma z_\Gamma  \dg
	=0 \quad 
	\text{ for all } z_\Gamma \in V_\Gamma, 
	\text{ a.e.\ in } (0,T).
	\label{main2-1}
\end{equation}
\pier{Moreover,} the \pier{quintuple} $(u, \xi, u_\Gamma, \mu_\Gamma, \xi_\Gamma)$ is obtained as limit of the family 
$\{(u_\varepsilon,\mu_\varepsilon, \xi_\varepsilon$, $u_{\Gamma,\varepsilon}$, $\mu_{\Gamma,\varepsilon}, \xi_{\Gamma,\varepsilon})\}_{\varepsilon \in (0,1]}$ 
\pier{of solutions to $({\rm P})_{\varepsilon\kappa}$ as $\varepsilon \to 0 $ in the following sense: 
there is a vanishing} subsequence $\{\varepsilon_k \}_{k \in \mathbb{N}}$ such that, as $k \to +\infty$, 
\begin{align}
	&u_{\varepsilon_k} \to u \quad \text{weakly~star~in~} H^1(0,T;H) \cap L^\infty (0,T;V) \cap L^2(0,T;W),
	\label{cv1e} \\
	&u_{\varepsilon_k} \to u \quad \text{strongly~in~} C\bigl([0,T];H \bigr) \cap L^2(0,T;V),
	\label{cv2e} \\
	&\pier{ \partial _{\boldsymbol{\nu}} u_{\kappa _k} \to \partial _{\boldsymbol{\nu}} u 
	\quad \text{weakly~in~} L^2 (0,T;H_\Gamma),} \label{cv2e-bis}
	\\	
	&\varepsilon_k \mu_{\varepsilon_k} \to 0  \quad \text{strongly~in~} L^2 (0,T;V),
	\label{cv3e} \\
	&\xi_{\varepsilon_k} \to \xi \quad  \text{weakly~in~} L^2 (0,T;H),
	\label{cv4e} \\
	&u_{\Gamma,\varepsilon_k} \to u_\Gamma \quad \text{weakly~star~in~} H^1(0,T;V_\Gamma') \cap L^\infty (0,T;V_\Gamma) \cap L^2(0,T;W_\Gamma),
	\label{cv5e} \\
	&u_{\Gamma,\varepsilon_k} \to u_\Gamma \quad \text{strongly~in~} C\bigl([0,T];H_\Gamma \bigr) \cap L^2(0,T;V_\Gamma),
	\label{cv6e} \\
	&\mu_{\Gamma,\varepsilon_k} \to \mu_\Gamma \quad \text{weakly~in~} L^2 (0,T;V_\Gamma),
	\label{cv7e} \\
	&\xi_{\Gamma,\varepsilon_k} \to \xi_\Gamma \quad \text{weakly~in~} \pier{L^2 (0,T;H_\Gamma)}.
	\label{cv8e}
\end{align}
\end{theorem}

\begin{proof} 
\pier{Let now the family 
$\{(u_\varepsilon,\mu_\varepsilon, \xi_\varepsilon, u_{\Gamma,\varepsilon},$ $\mu_{\Gamma,\varepsilon}, \xi_{\Gamma,\varepsilon})\}_{\varepsilon \in (0,1]}$ 
denote the solutions to~$({\rm P})_{\varepsilon\kappa}$, obtained by passing to the limit as $\lambda\to 0$ in the approximating problem (cf.~Subsection~3.1). Then, the uniform estimate~\eqref{pier5} can be 
confirmed for $(u_\varepsilon,\mu_\varepsilon, \xi_\varepsilon, u_{\Gamma,\varepsilon},$ $\mu_{\Gamma,
\varepsilon}, \xi_{\Gamma,\varepsilon})$. Moreover, in view of Lemma~\ref{L6}, we also point out the 
estimate 
\begin{align}
   &{\kappa}^{1/2} \|u_\varepsilon\| _{L^2(0,T;H^{3/2}(\Omega))} +
	{\kappa}^{1/2} \|\partial_{\boldsymbol{\nu}} u_\varepsilon\| _{L^2(0,T;H_\Gamma)} 
	+ {\kappa}^{3/2} \|\Delta_\Gamma u_{\Gamma,\kappa} \|_{L^2(0,T;H_\Gamma)} 
	 \le C. \label{pier10}
\end{align}%
Then, by elliptic regularity on the boundary we have that 
\begin{equation}
	\kappa^{3/2} \| u_{\Gamma,\varepsilon}\|_{L^2(0,T;W_\Gamma)} \le C. \label{unie4}
\end{equation}%
Then, owing to the boundary condition~\eqref{e4} and elliptic regularity we deduce that (see, e.g., \cite[Theorem~3.2, p.~1.79]{BG87}) 
\begin{equation}
	\kappa^{3/2} \| u_{\varepsilon}\|_{L^2(0,T;W)} \le C. \label{unie5}
\end{equation}
%From now, we discuss the limiting procedure $\lambda \to 0$ and $\kappa \to 0$, here from the same reason we 
%omit the limiting $\lambda \to 0$ remaining with $\kappa>0$, namely 
%we can interpret 
%that the family 
%$\{(u_\varepsilon,\mu_\varepsilon, \xi_\varepsilon, u_{\Gamma,\varepsilon},$ $\mu_{\Gamma,\varepsilon}, \xi_{\Gamma,\varepsilon})\}_{\varepsilon \in (0,1]}$ 
%is the solutions of $({\rm P})_{\varepsilon\kappa}$ and hereafter 
%let us consider the limiting procedure $\lambda=\varepsilon \to 0$ remaining $\kappa>0$. 
Hence, based on the uniform estimates, there exists a quintuple
 $(u, \xi, u_\Gamma, \mu_\Gamma, \xi_\Gamma)$ and 
 a subsequence $\{\varepsilon_k \}_{k \in \mathbb{N}}$ such that 
the \pier{weak and weak star convergences} stated in \eqref{cv1e}--\eqref{cv8e} hold as $k \to +\infty$.
Furthermore, from \eqref{pier5} we obtain the strong convergence~\eqref{cv3e}. 
By applying the the Aubin--Lions compactness theorems~(see~\cite[Sect.~8, Cor.~4]{Sim87}), 
we also derive the strong convergences \eqref{cv2e} and \eqref{cv6e}. These, in particular, ensure that $u$ and $u_\Gamma$  satisfy the initial conditions~\eqref{e5} and~\eqref{e8}, respectively.
Moreover, due to the {L}ipschitz continuity of $\pi$ and $\pi_\Gamma$, we obtain the strong convergences 
of $\pi(u_{\varepsilon_k})$ and $\pi_\Gamma(u_{\varepsilon_k})$, as in the proof of Theorem~\ref{thm1}.
Then, it is straightforward to pass to the limit in the equations in \eqref{ek3} and \eqref{ek7}.
Additionally, by the standard demi-closedness property of maximal monotone operators~\cite{Bar10, Bre73}, the inclusions in \eqref{e3} and \eqref{e7} follow directly. The trace condition~\eqref{e4} is also a direct consequence of the strong convergences \eqref{cv2e} and \eqref{cv6e}. 
Finally, the variational equation~\eqref{main2-1} is obtained immediately from \eqref{ek1and6}, inlight of~\eqref{cv3e}. This completes the proof of the theorem.}
\end{proof}
\smallskip

\subsection{Both parameters tending to zero}

\pier{In order to deal with the limiting procedure $({\rm P})_{\varepsilon\kappa} \to ({\rm P})$, here we let $\{(u_{\takeshi{\varepsilon, \kappa}},\mu_{\takeshi{\varepsilon, \kappa}}, 
\xi_{\takeshi{\varepsilon, \kappa}}, u_{\Gamma,\takeshi{\varepsilon, \kappa}},$ 
$\mu_{\Gamma,\takeshi{\varepsilon, \kappa}}, \xi_{\Gamma,\takeshi{\varepsilon, \kappa}})
\}_{\takeshi{\varepsilon, \kappa} \in (0,1]}$ 
denote the solution of $({\rm P})_{\varepsilon\kappa}$. 
We consider the case $\varepsilon \to 0$ and $\kappa\to 0$. 
Recalling \eqref{ek1}--\eqref{ek8} and Proposition~\ref{ANA}, it is clear that 
$u_{\takeshi{\varepsilon, \kappa}}$, $\mu_{\takeshi{\varepsilon, \kappa}}$, 
$\xi_{\takeshi{\varepsilon, \kappa}}$, $u_{\Gamma,\takeshi{\varepsilon, \kappa}}$,  
$\mu_{\Gamma,\takeshi{\varepsilon, \kappa}}$, $\xi_{\Gamma,\takeshi{\varepsilon, \kappa}}$
 satisfy 
\begin{align} 
	\int_\Omega \partial_t u_\ke z \dx 
	+ 
	\int_\Omega \nabla u_\ke \cdot \nabla z\dx 
	+
	\kappa
	\int_\Gamma \nabla _\Gamma u_\Gke \cdot \nabla _\Gamma z_\Gamma \dg 
	\notag \\
	+
	\int_\Omega \bigl( \xi_\ke + \pi(u_\ke) \bigr) z \dx 
	+
	\int_\Gamma \bigl( \xi_\Gke+
	\pi_\Gamma(u_\Gke) \bigr) z_\Gamma \dg
	= 
	\int_\Omega f z \dx 
	\notag \\
	+
	\int_\Gamma \bigl(
	f_\Gamma +\mu_\Gke \bigr)z_\Gamma \dg\quad {\rm for~all~} (z,z_\Gamma ) \in \boldsymbol{V}, \ {\rm a.e.~in~} (0,T),
	\label{ee1}
	\\[2mm]
   (u_\ke)_{|_\Gamma} = u_\Gke
   \quad {\rm a.e.\ on~} \Sigma, \quad 
  (\mu_\ke)_{|_\Gamma} =\mu_\Gke
   \quad {\rm a.e.\ on~} \Sigma, \label{pier11} \\[2mm]
   \xi_{\ke} \in \beta (u_{\ke}) 
   \quad {\rm a.e.\ in~} Q, \quad 
  \xi_{\Gke} \in \beta_{\Gamma} (u_\Gke) 
   \quad {\rm a.e.\ on~} \Sigma, \label{pier12}
	\\[2mm]
	\langle \partial_t u_\Gke,z_\Gamma \rangle_{V_\Gamma',V_\Gamma}
	+ \varepsilon \int_\Omega \nabla \mu_\ke \cdot \nabla z \dx 
	+ \int_\Gamma \nabla _\Gamma \mu_\Gke \cdot \nabla_\Gamma z_\Gamma \dg
	=0 \notag \\
	\quad {\rm for~all~} (z,z_\Gamma ) \in \boldsymbol{V}, \ {\rm a.e.~in~} (0,T), \label{ee2}\\[2mm]
     u_\ke(0) = u_0 \quad {\rm a.e.\ in~} \Omega,
 \quad 
  u_\Gke (0) = u_{0\Gamma} 
    \quad {\rm a.e.\ on~} \Gamma. 
   \label{pier13}
\end{align}
We also recall the uniform estimate \eqref{pier5}, which is still useful for the proof of the following result.}
\pier{%
\begin{theorem}\label{thm3}
Under the assumptions {\rm (A1)}--{\rm (A5)}, there exists at least one \pier{quintuple} $(u, \xi, u_\Gamma, \mu_\Gamma, \xi_\Gamma)$ fulfilling
\begin{gather*}
	u \in H^1(0,T;H) \cap L^\infty(0,T;V),\quad 
	\Delta u \in L^2(0,T;H), \quad 
	\xi \in L^2(0,T;H), 	
	\\
	u_\Gamma \in H^1(0,T;V_\Gamma') \cap C\bigl([0,T];H_\Gamma \bigr) \cap L^\infty (0,T;Z_\Gamma), \\
	\mu_\Gamma \in L^2(0,T;V_\Gamma), \quad \xi_\Gamma \in L^2(0,T;Z_\Gamma')
\end{gather*}
and satisfying \eqref{e3}--\eqref{e5}, \eqref{e8}, and the equation \eqref{k6} and the conditions \eqref{e7} in the following weak sense: 
\begin{align}
	\langle \partial_t u_{\Gamma},z_\Gamma \rangle_{V_\Gamma',V_\Gamma}
	+ \int_\Gamma \nabla _\Gamma \mu_{\Gamma} \cdot \nabla_\Gamma z_\Gamma \dg
	=0 \quad 
	\text{ for all } z_\Gamma \in V_\Gamma, 
	\text{ a.e.\ in } (0,T),
	\label{main3-2} \\
	\noalign{\smallskip}
	\int_\Gamma \mu_\Gamma z_\Gamma \dg =
	\langle \partial_{\boldsymbol{\nu}} u + \xi _{\Gamma},z_\Gamma \rangle_{Z_\Gamma',Z_\Gamma}
	+ \int_\Gamma \bigl( \pi_\Gamma(u_\Gamma) -f_\Gamma \bigr) z_\Gamma \dg \quad \hbox{and} 
	\notag %\label{main1-2}
	\\
	\langle \xi_\Gamma ,z_\Gamma-u_\Gamma \rangle_{Z_\Gamma',Z_\Gamma } 
	+ \int_\Gamma \widehat{\beta}_\Gamma (u_\Gamma) \dg 
	\le \int_\Gamma \widehat{\beta}_\Gamma(z_\Gamma)\, \dg \notag \\
	\mbox{for all } z_\Gamma \in Z_\Gamma, 
	\mbox{ a.e.\ in } (0,T).
	\label{main3-5}
\end{align}
Moreover, the \pier{quintuple} $(u, \xi, u_\Gamma, \mu_\Gamma, \xi_\Gamma)$ is obtained as limit of the family $\{(u_\ke,\mu_\ke, \xi_\ke,$ $u_\Gke,\mu_\Gke, \xi_\Gke)\}_{\takeshi{\varepsilon,\kappa} \in (0,1]}
$ of solutions to $({\rm P})_{\varepsilon\kappa}$ as $(\ke) \to (0,0)$ in the following sense:
there is a subsequence $\{(\kek) \}_{k \in \mathbb{N}}$ such that, as $k \to +\infty$, 
\begin{align}
	&u_{\kek} \to u \quad \text{weakly~star~in~} H^1(0,T;H) \cap L^\infty (0,T;V),
	\label{cv1}	\\
	&u_{\kek} \to u \quad \text{strongly~in~} C\bigl([0,T];H \bigr),
	\label{cv2} \\
	&\varepsilon_k \mu_{\kek} \to 0 \quad \text{strongly~in~} L^2 (0,T;V),
	\label{cv3} \\
	&\xi_{\kek} \to \xi \quad  \text{weakly~in~} L^2 (0,T;H),
	\label{cv4} \\
	&u_\Gkek \to u_\Gamma \quad \text{weakly~star~in~} H^1(0,T;V_\Gamma') \cap L^\infty (0,T;Z_\Gamma),
	\label{cv5} \\
	&u_\Gkek \to u_\Gamma \quad \text{strongly~in~} C\bigl([0,T];H_\Gamma \bigr),
	\label{cv6} \\
	&\kappa_k u_\Gkek  \to 0 \quad \text{strongly~in~} L^\infty (0,T;V_\Gamma),
	\label{cv6b} \\
	&\mu_\Gkek  \to \mu_\Gamma \quad \text{weakly~in~} L^2 (0,T;V_\Gamma),
	\label{cv7} \\
	&\xi_\Gkek  \to \xi_\Gamma \quad \text{weakly~in~} L^2 (0,T;V_\Gamma'),
	\label{cv8} \\
	& (-\kappa_k \Delta_\Gamma u_\Gkek  + \xi_\Gkek)  \to \xi_\Gamma \quad 
	\text{weakly~in~} L^2 (0,T;Z_\Gamma').
	\label{cv8b}
\end{align}
\end{theorem}
\begin{proof} 
The proof proceeds along the lines of the arguments used in Theorems~\ref{thm1} and~\ref{thm2}, beginning with the uniform estimates and then passing to the limit via weak and weak star compactness along a suitable subsequence $\{(\kek) \}.$ In addition to the convergences stated in~\eqref{cv1}--\eqref{cv8b} we also observe the following additional convergence properties 
\begin{align*}
	&\Delta u_{\kek} \to \Delta u \quad \text{weakly~in~} L^2 (0,T;H),
	\\
	& \partial _{\boldsymbol{\nu}} u_\kek \to \partial _{\boldsymbol{\nu}} u 
	\quad \text{weakly~in~} L^2 (0,T;Z_\Gamma'), 
\end{align*} 
which are also useful in the passage to the limit. At this point, it suffices to closely follow the arguments employed in the proofs of the two preceding asymptotic results in order to arrive at the desired conclusion. 
\end{proof}}%

\begin{corollary}\label{cor3}
\pier{In the same framework of Theorem~\ref{thm3} and under the further assumption~{\rm (A6)}, 
the found \pier{quintuple} $(u,\mu, \xi, u_\Gamma, \mu_\Gamma, \xi_\Gamma)$ additionally fulfils 
\begin{equation*}
	u \in L^2\bigl( 0,T;H^{3/2}(\Omega) \bigr), \quad 
	\partial_{\boldsymbol{\nu}} u \in L^2(0,T;H_\Gamma), \quad 
	u_\Gamma \in L^2(0,T;V_\Gamma), \quad
	\xi_\Gamma \in L^2(0,T;H_\Gamma).
\end{equation*}
Moreover, the further convergence properties 
\begin{align}
	&\xi_{\Gamma,\takeshi{\varepsilon_k,\kappa_k}} \to \xi_\Gamma \quad \text{weakly~in~} L^2 (0,T;H_\Gamma), \label{pier14}
	\\
	&\partial _{\boldsymbol{\nu}} u_{\takeshi{\varepsilon_k,\kappa _k}} - {\kappa_k} \Delta_\Gamma u_{\Gamma,\takeshi{\varepsilon_k,\kappa_k}}  \to \partial _{\boldsymbol{\nu}} u \quad \text{weakly~in~} 
	L^2 (0,T;H_\Gamma).  \label{pier15}
\end{align} 
hold and the conditions in~\eqref{main3-5} can be equivalently formulated as~\eqref{e7}.}
\end{corollary}

\smallskip
\pier{The proof of this result follows identically from that of Corollary~\ref{cor1} and is therefore omitted for brevity.}

\section{Continuous dependence results}
\label{cont-dep}

\pier{In this section, we address the continuous dependence results for the problems $({\rm P})_{\varepsilon}$, $({\rm P})_{\kappa}$, and 
$({\rm P})$, respectively. Throughout the discussion, we assume that conditions  
{\rm (A1)}--{\rm (A5)} are satisfied. With this assumption in place, all theorems presented in this section imply the uniqueness of the functions  $u$ and $u_\Gamma$ corresponding to each of the problems~$({\rm P})_{\varepsilon}$, $({\rm P})_{\kappa}$, and $({\rm P})$. 
Moreover, if the graphs $\beta$ and $\beta_\Gamma$ are single-valued functions, then the remaining unknowns are also uniquely determined (cf.~Theorems~\ref{thm1}, \ref{thm2}, and \ref{thm3}).
It is important to note that although the same notation $\bar{u}$ is used throughout this section, it refers to different functions in each subsection, depending on the specific problem under consideration.}

\subsection{Continuous dependence for $({\rm P})_{\varepsilon}$}
\pier{Throughout this subsection, let
$(u^{(i)},\mu^{(i)}, \xi^{(i)},$ $ u_\Gamma^{(i)},\mu_\Gamma^{(i)}, \xi_\Gamma^{(i)})$, $i=1,2$, denote 
two solutions 
of Problem~$({\rm P})_{\varepsilon}$ corresponding to the data $\{u_{0}^{(i)}, u_{0\Gamma}^{(i)}, f^{(i)}, f_\Gamma^{(i)}\}$,
$i=1,2$, that satisfy the assumptions {\rm (A4)} and {\rm (A5)}. We further assume that
\begin{equation}
	(u_{0\Gamma}^{(1)}-u_{0\Gamma}^{(2)}, 1)_{H_\Gamma} = 0,
\label{pier16}
\end{equation}
so that the mean value $m_\Gamma$ is the same for both initial data on the boundary. 
By ``solutions'' to Problem~$({\rm P})_{\varepsilon}$, we mean that the sextuplets~$(u^{(i)},\mu^{(i)}, \xi^{(i)}, u_\Gamma^{(i)}, \mu_\Gamma^{(i)}, \xi_\Gamma^{(i)})$, $i=1,2$, possess the regularity properties stated in~Theorem~\ref{thm1} and satisfy the conditions~\eqref{e1}--\eqref{e5}, \eqref{e8}, \eqref{e6weak}, \eqref{e7weak} in terms of  their respective data.  
We now put $\bar{u}=u^{(1)}-u^{(2)}$ and analogously use the bar notation for the differences of the other functions. With this notation in place, we can derive the estimate stated in the following result.}

\begin{theorem}\label{thm1b}
There exists a positive constant \pier{$C$, independent of~$\varepsilon \in (0,1]$,} such that
\begin{align*}
	& \|\bar{u}\|_{C([0,T];H)} + \|\bar{u}\|_{L^2(0,T;V)} +
	\|\bar{u}_\Gamma\|_{C([0,T];V_\Gamma')}
	\notag \\ 
	& \le\pier{C} \left( \|\bar{u}_0\|_H + \|\bar{u}_{0\Gamma}\|_{V_\Gamma'} 
	+ \|\bar{f}\|_{L^2(0,T;V')} + \|\bar{f}_\Gamma\|_{L^2(0,T;H_\Gamma)}\right).
\end{align*}
\end{theorem}

\begin{proof}
Taking the difference of the \pier{equations~\eqref{main1-3} and choosing}
$(z,z_\Gamma):=(\bar{u}, \bar{u}_\Gamma)$ we have  
\begin{align} 
	& \frac{1}{2} \frac{d}{dt}
	\int_\Omega |\bar{u}|^2 \dx 
	+ 
	\int_\Omega |\nabla \bar{u}|^2  \dx 
	+ 
	\int_\Omega \bar{\xi} \bar{u} \dx 
	+
	\langle \bar{\xi}_\Gamma, \bar{u}_\Gamma \rangle_{Z_\Gamma',Z_\Gamma}
	- \int_\Gamma \bar{\mu}_\Gamma  \bar{u}_\Gamma \dg
	\notag \\
	\quad {}
	& = 
	- 
	\int_\Omega
	\bigl( 
	\pi(u^{(1)})-\pi(u^{(2)}) \bigr) \bar{u} \dx 
	-
	\int_\Gamma 
	\bigl( \pi_\Gamma(u^{(1)}_\Gamma) - \pi_\Gamma \pier{(u_\Gamma^{(2)})} \bigr) \bar{u}_\Gamma \dg
	\notag \\
	& \quad {}
	+ 
	\int_\Omega \bar{f} \bar{u} \dx 
	+
	\int_\Gamma 
	\bar{f}_\Gamma \bar{u}_\Gamma \dg
	\label{eqforconti}
\end{align}
a.e.\ in $(0,T)$. 
Here, in the same way of the proof of Lemma~\ref{L4}, 
we define $(\bar{y}, \bar{y}_{\Gamma}) \in H^1(0,T;\boldsymbol{V})$ as the solution to 
\begin{align}
	\varepsilon \int_\Omega \nabla \bar{y}(t) \cdot \nabla z \dx 
	+ \int_\Gamma \nabla _\Gamma \bar{y}_{\Gamma}(t) \cdot \nabla _\Gamma z_\Gamma \dg
	= 
	\bigl\langle \bar{u}_{\Gamma}(t), z_\Gamma \bigr\rangle _{V_\Gamma',V_\Gamma}\nonumber\\
	\hbox{for all $t \in [0,T]$ and $(z,z_\Gamma) \in \boldsymbol{V}$}	\label{subeq2}
\end{align}
that satisfies $\int_\Gamma \bar{y}_{\Gamma} d\Gamma=0$.
Of course, it is important that 
\begin{equation} 
	\langle \bar{u}_\Gamma, 1 \rangle_{V_\Gamma',V_\Gamma}
	= 0 \quad {\rm in~} (0,T),
	\label{pier16-1}
\end{equation}
and this condition is ensured from \pier{(cf. \eqref{e6weak} and \eqref{e1})}
\begin{equation*} 
	\langle \bar{u}_{0\Gamma}, 1 \rangle_{V_\Gamma',V_\Gamma}
	= 0.
\end{equation*}
Taking $(z,z_\Gamma)=(\bar{\mu}, \bar{\mu}_\Gamma)$ in \eqref{subeq2}, and 
 $(z,z_\Gamma)=(\bar{y}, \bar{y}_\Gamma)$ in 
the difference between \pier{the equalities~\eqref{e6weak} written for $u^{(1)}$ and $u^{(2)}$, we easily compare and, with the help of~\eqref{e1}, deduce that}
\begin{equation}
	\langle \bar{u}_\Gamma, \bar{\mu}_\Gamma \rangle_{V_\Gamma', V_\Gamma} = 
	- \langle \partial_t \bar{u}_{\Gamma},\bar{y}_\Gamma \rangle_{V_\Gamma',V_\Gamma}.
	\label{eqforconti2}
\end{equation}
Moreover, differentiating \eqref{subeq2} with respect to time, then taking 
$(z,z_\Gamma)=(\bar{y}, \bar{y}_\Gamma)$, we have from integration of the resultant that 
\begin{gather}
	\frac{\varepsilon}{2} \bigl\| \nabla \bar{y}(t) \bigr\|_H^2 
	 + \frac{1}{2} \bigl\| \nabla_\Gamma \bar{y}_\Gamma (t) \bigr\|_{H_\Gamma}^2  
	 - \frac{\varepsilon}{2} \| \nabla \bar{y}_0 \|_H^2 
	 - \frac{1}{2} \| \nabla_\Gamma \bar{y}_{0\Gamma}\|_{H_\Gamma}^2 
	 = \int_0^t \langle \partial_t \bar{u}_\Gamma, \bar{y}_\Gamma \rangle_{V_\Gamma',V_\Gamma} \ds \label{base}
\end{gather}
for all $t \in [0,T]$. \pier{On the other hand, from~\eqref{subeq2} it follows that} 
\begin{align*}
	\bigl\| \bar{u}_\Gamma (t) \bigr\|_{V_\Gamma'} 
	& = \sup_{%\tiny 
	\substack{z_\Gamma \in V_\Gamma \\ \|z_\Gamma\|_{V_\Gamma} \le 1}}
	\bigl| \langle \bar{u}_\Gamma(t), z_\Gamma \rangle_{V_\Gamma',V_\Gamma} \bigr|
	\notag \\
	& = \sup_{%\tiny 
	\substack{z_\Gamma \in V_\Gamma \\ \|z_\Gamma \|_{V_\Gamma} \le 1}}
	\left| \,
	\varepsilon \!\int_\Omega \nabla \bar{y}(t) \cdot \nabla 
	{\mathcal R}\pier{z_\Gamma }\dx 
	+ \int_\Gamma \nabla _\Gamma \bar{y}_{\Gamma}(t) \cdot \nabla _\Gamma z_\Gamma \dg
	\right|
	\notag \\
	& \le \sup_{%\tiny 
	\substack{z_\Gamma \in V_\Gamma \\ \|z_\Gamma\|_{V_\Gamma} \le 1}}
	\left\{ 
	\pier{C}\,\varepsilon  \bigl\| \nabla \bar{y}(t) \bigl\|_H 
	\|z_\Gamma\|_{V_\Gamma} 
	+ \bigl\| \nabla _\Gamma \bar{y}_{\Gamma}(t) \bigr\|_{H_\Gamma} 
	\|z_\Gamma\|_{V_\Gamma}
	\right\},
\end{align*}
that is, there exists a \pier{constant $c>0$, independent of $\varepsilon \in (0,1]$,} such that 
\begin{equation}
	\frac{\varepsilon}{2} \bigl\| \nabla \bar{y}(t) \bigr\|_H^2 
	 + \frac{1}{2} \bigl\| \nabla_\Gamma \bar{y}_\Gamma (t) \bigr\|_{H_\Gamma}^2 
	 \ge \pier{c} \bigl\| \bar{u}_\Gamma (t) \bigr\|_{V_\Gamma'}^2
	 \label{eqc}
\end{equation}
for all $t \in [0,T]$. Next, \pier{by considering \eqref{subeq2} at $t=0$, we have that}
\begin{align*}
	\varepsilon \| \nabla \bar{y}_0 \|_H^2 
	+  \| \nabla_\Gamma \bar{y}_{0\Gamma}\|_{H_\Gamma}^2 
	& = 
	\langle \bar{u}_{0\Gamma}, \bar{y}_{0\Gamma} \rangle _{V_\Gamma',V_\Gamma} \\
	& \le \frac{\delta}{2} \left( 
	2\| \bar{y}_{0\Gamma}  \|_{H_\Gamma}^2 + 2\| \nabla _\Gamma \bar{y}_{0\Gamma}  \|_{H_\Gamma}^2 \right)
	+ \frac{1}{2\delta }\|\bar{u}_{0\Gamma}\|_{V_\Gamma'}^2 
	\notag \\
	& \le \delta(\pier{C_{\rm P}+1}) 
	\| \nabla \bar{y}_{0\Gamma}  \|_{H_\Gamma}^2
	+ \frac{1}{2\delta }\|\bar{u}_{0\Gamma}\|_{V_\Gamma'}^2 \pier{,}
\end{align*}
where we \pier{have used~\eqref{Poin2} along with the condition} $\int_\Gamma  \bar{y}_{0\Gamma} \dg =0$. Thus, 
choosing $\delta=1/2(\pier{C_{\rm P}+1})$, we see that there exists a positive constant $C>0$ such that
\begin{equation}
	\frac{\varepsilon}{2} \| \nabla \bar{y}_0\|_H^2 
	 + \frac{1}{2} \| \nabla_\Gamma \bar{y}_{0\Gamma} \|_{H_\Gamma}^2 
	 \le C \| \bar{u}_{0\Gamma} \bigr\|_{V_\Gamma'}^2. 
	 \label{eqC}
\end{equation}
Now, we go back to \eqref{eqforconti} \pier{and add $\|\bar{u}\|_H^2$ to both sides. Then,
in the light of \eqref{eqforconti2}, \eqref{eqc}, \eqref{eqC} and by integrating with respect to time,} the above equation~\eqref{base} allows us to infer that 
\begin{align*} 
	& \frac{1}{2} 
	\bigl\| \bar{u}(t) \bigr\|_H^2
	+ \frac{1}{2}\int_0^t \|\bar{u}\|_V^2 \ds
	+ 
	\pier{c} \bigl\| \bar{u}_\Gamma (t) \bigr\|_{V_\Gamma'}^2
	\notag \\
	\quad {}
	& \le \frac{1}{2} 
	\| \bar{u}_0 \|_{H}^2
	+
	C \| \bar{u}_{0\Gamma} \|_{V_\Gamma'}^2
	+
	(1+L)
	\int_0^t 
	\| 
	 \bar{u} \|_H^2 \ds 
	+\left(L_\Gamma+\frac{1}{2} \right)
	\int_0^t 
	\|
	\bar{u}_\Gamma
	\|_{H_\Gamma}^2 \ds
	\notag \\
	& \quad {}
	+ \frac{1}{2}
	\int_0^t \|\bar{f}\|_{V'}^2 \ds 
	+\frac{1}{2}
	\int_0^t 
	\| \bar{f}_\Gamma\|_{H_\Gamma}^2 \ds
\end{align*}
for all $t \in [0,T]$, where \pier{the monotonicity of $\beta$ and $\beta_\Gamma$ has been taken into account.} 
Additionally, we can use the compactness inequality \eqref{comp2} to 
control the term \pier{involving} $\|\bar{u}_\Gamma\|_{H_\Gamma}^2$. Thus, \pier{with the help of the Gronwall lemma} we easily arrive at the conclusion.
\end{proof}

\subsection{Continuous dependence for $({\rm P})_{\kappa}$}

Throughout this subsection, we denote by 
\pier{$(u^{(i)}, \xi^{(i)}, u_\Gamma^{(i)}, \mu_\Gamma^{(i)}, \xi_\Gamma^{(i)})$}, $i=1,2$, two \pier{solutions of Problem~$({\rm P})_{\kappa}$ corresponding to the respective data $\{u_{0}^{(i)}, u_{0\Gamma}^{(i)}, f^{(i)}, f_\Gamma^{(i)}\}$,
$i=1,2$. The data are supposed to satify the assumptions~{\rm (A4)} and {\rm (A5)}, together with \eqref{pier16}.
Put $\bar{u}=u^{(1)}-u^{(2)}$, 
and adopt the same notation with the bar for the differences of other functions. 
Then, we can show the following result.}
\begin{theorem}\label{thm2b}
There exists a positive constant \pier{$C_\kappa$, which depends on $\kappa \in (0,1]$,} such~that
\begin{align*}
	& \|\bar{u}\|_{C([0,T];H)} + \|\bar{u}\|_{L^2(0,T;V)}+
	\|\bar{u}_\Gamma\|_{C([0,T];V_\Gamma')} +\|\bar{u}_\Gamma\|_{L^2(0,T;V_\Gamma)}\\
	& \le C_\kappa\left( \|\bar{u}_0\|_H + \|\bar{u}_{0\Gamma}\|_{V_\Gamma'} + \|\bar{f}\|_{L^2(0,T;V')} + \|\bar{f}_\Gamma\|_{L^2(0,T;V_\Gamma')}\right).
\end{align*}
\end{theorem}

\begin{proof}
\pier{We take the differences of the equations in \eqref{e3} and \eqref{k7}. Then, using a pair 
$(z,z_\Gamma)\in \boldsymbol{V}$ as test function, it is not difficult to derive the variational equality 
\begin{align}
	&\int_\Omega \partial_t \bar{u} z \dx 
	+ 
	\int_\Omega \nabla \bar{u} \cdot \nabla z\dx 
	+ 
	\kappa\int_\Gamma \nabla_\Gamma \bar{u}_\Gamma \cdot \nabla_\Gamma z_\Gamma\d\Gamma
	\notag \\
	&{}+ 
	\int_\Omega \bigl( \bar\xi + \pi(u^{(1)})-\pi(u^{(2)})\bigr) z \dx 
	+
	\int_\Gamma \bigl(\bar\xi_\Gamma + 
	\pi_\Gamma(u^{(1)}_\Gamma) - \pi_\Gamma (u^{(2)_\Gamma})  \bigr) z_\Gamma \dg
	\notag \\
	&{}
	= 
	\int_\Omega \bar f  z \dx 
	+
	\int_\Gamma ( 
	\bar f_\Gamma + \bar\mu_\Gamma )z_\Gamma \dg
	\quad 
	\text{for all } (z,z_\Gamma) \in \boldsymbol{V}, 
	\text{ a.e.\ in } (0,T).
	\label{pier17}
\end{align}
Then, choosing $(z,z_\Gamma):=(\bar{u}, \bar{u}_\Gamma)$ and using the Lipschitz continuity of $\pi$ and $\pi_\gamma$ lead to
\begin{align}
	& \frac{1}{2}\frac{d}{dt} \int_\Omega |\bar{u}|^2 \dx
	+ \int_\Omega |\nabla \bar{u}|^2 \dx
	+\kappa \int_\Gamma |\nabla_\Gamma \bar{u}_\Gamma|^2 \dg 
	\notag \\
	& +{}
	\int_\Omega \bar{\xi} \bar{u} \dx 
	+ 
	\int_\Gamma \bar{\xi}_\Gamma \bar{u}_\Gamma \dg	
	-\int _\Gamma \bar{\mu}_\Gamma \bar{u}_\Gamma \dg 
	\notag \\
	& \leq 
	L
	\| \bar{u} \|^2_H
	+ 
	L_\Gamma 
	\| \bar{u}_\Gamma \|_{H_\Gamma} ^2 +
	 \|\bar{f}\|_{V'} \| \bar{u} \|_V
	+ 
	\| \bar{f}_\Gamma\|_{V_\Gamma'} \| \bar{u}_\Gamma \|_{V_\Gamma},
	\label{base2}
\end{align}
a.e.\ in $(0,T)$. Next, we define $V_{\Gamma,0}:=\{ z _\Gamma \in V_\Gamma : \int _\Gamma z_\Gamma \dg =0 \}$ and consider the linear operator $F_\Gamma :V_{\Gamma,0} \to V_{\Gamma,0}'$ specified by 
\begin{equation*}
	\langle F_\Gamma z_\Gamma, \tilde{z}_\Gamma \rangle_{V_{\Gamma,0}',V_{\Gamma,0}}
	:=\int_\Gamma \nabla _\Gamma z_\Gamma \cdot \nabla _\Gamma \tilde{z}_\Gamma \dg
	\quad {\rm for~all~} z_\Gamma, \tilde{z}_\Gamma \in V_{\Gamma,0}.
\end{equation*}
Hence, from the Poincar\'e inequality we see that there exists a 
positive constant $c_{\rm P}>0$ such that 
\begin{equation}
	c_{\rm  P} \|z_\Gamma \|_{V_\Gamma}^ 2 \le \langle F_\Gamma z_\Gamma, z_\Gamma \rangle_{V_{\Gamma,0}',V_{\Gamma,0}} =: \|z_\Gamma \|_{V_{\Gamma,0}}^2
	\quad {\rm for~all~} z_\Gamma \in V_{\Gamma,0}.
	\label{pier18} 
\end{equation}
Thanks to the fact that $\|z_\Gamma\|_{V_{\Gamma,0}} \le \|z_\Gamma\|_{V_\Gamma}$
for all $z_\Gamma \in V_\Gamma$, we see that $\|\cdot \|_{V_{\Gamma}}$ and 
$\|\cdot \|_{V_{\Gamma,0}}$ are equivalent norms on $V_{\Gamma,0}$ and 
then $F_\Gamma : V_{\Gamma,0} \to V_{\Gamma,0}'$ is a duality mapping. 
Moreover, since the kernel $\ker (F_\Gamma) $ contains only the null function, it turns out that  $F_\Gamma^{-1}:R(F_\Gamma)=V_{\Gamma,0}' \to V_{\Gamma,0}$ is linear continuous.} 
Additionally, we can define the inner product in $V_{\Gamma,0}'$ by
\begin{equation*}
	(z_\Gamma^*, \tilde{z}_\Gamma^*)_{V_{\Gamma,0}'} := \langle z_\Gamma^*, F^{-1}_\Gamma \tilde{z}_{\Gamma}^* \rangle_{V_{\Gamma,0}', V_{\Gamma,0}}
	\quad {\rm for~all~} z_\Gamma^*, \tilde{z}_\Gamma^* \in V_{\Gamma,0}'.
\end{equation*}
\pier{Now, taking the difference of the two equations~\eqref{main2-1} we obtain
\begin{equation*}
	\langle \partial_t \bar{u}_{\Gamma},z_\Gamma \rangle_{V_\Gamma',V_\Gamma}
	+ \int_\Gamma \nabla _\Gamma \bar{\mu}_{\Gamma} \cdot \nabla_\Gamma z_\Gamma \dg
	=0 \quad 
	\text{ for all } z_\Gamma \in V_\Gamma, 
	\text{ a.e.\ in } (0,T).
\end{equation*}
and here we are allowed to choose $z_\Gamma =F_\Gamma^{-1} \bar{u}_\Gamma$ (cf.~\eqref{pier16} and \eqref{pier16-1}). Then, using this in \eqref{base2}, and   
adding $\|\bar{u}\|_H^2$ to both sides of \eqref{base2}, the integration of the resultant over $[0,t]$ yields}
\begin{align}
	& \frac{1}{2} \bigl\| \bar{u}(t) \bigr\|_{H}^2 
	+ \frac{1}{2} \bigl\| \bar{u}_\Gamma(t) \bigr\|_{V_{\Gamma,0}'}^2
	+ \frac{1}{2}\int_0^t \|\bar{u}\|_V^2 \ds
 	+\kappa \int_0^t \|\nabla_\Gamma \bar{u}_\Gamma\|_{H_\Gamma}^2 \ds
	\notag \\
	& \le \frac{1}{2} \|\bar{u}_0\|_{H}^2
	+ \frac{1}{2} \|\bar{u}_{0\Gamma}\|_{V_{\Gamma,0}'}^2 +
	(1+L) 
	\int_0^t
	\| \bar{u} \|^2_H \ds
	+ 
	L_\Gamma 
	\int_0^t
	\| \bar{u}_\Gamma \|_{H_\Gamma} ^2 \ds
	\notag \\
	& \quad {}
	+ \frac{1}{2} \int_0^t \|\bar{f}\|_{V'}^2 \ds 
	+ 
	\frac{1}{4\delta}
	\int_0^t 
	\| \bar{f}_\Gamma\|_{V_\Gamma'}^2 \ds + 
	\delta \int_0^t \| \bar{u}_\Gamma \|_{V_\Gamma}^2 \ds 
	\label{sim} 
\end{align}
for all $t \in [0,T]$. \pier{Therefore, by virtue of \eqref{pier18} we may take 
$\delta :=(c_{\rm P}\kappa)/2$ and thus gain the contribution of the last term in the right-hand side of \eqref{sim}. At this point, we can conclude as in the proof of Theorem~\ref{thm1b},
by the compactness inequality~\eqref{comp2} and the Gronwall lemma.}
\end{proof}

\subsection{Continuous dependence for $({\rm P})$}
\pier{As before,  
we let 
$(u^{(i)}, \xi^{(i)}, u_\Gamma^{(i)}, \mu_\Gamma^{(i)}, \xi_\Gamma^{(i)})$, $i=1,2$, be two  solutions of the problem~$({\rm P})$ corresponding to the data 
$\{u_{0}^{(i)}, u_{0\Gamma}^{(i)}, f^{(i)}, f_\Gamma^{(i)}\}$,
$i=1,2$, respectively. The data are assumed to fulfill the assumptions~{\rm (A4)}, {\rm (A5)} and condition \eqref{pier16}. We use the notation $\bar{u}=u^{(1)}-u^{(2)}$ and similarly  
for the differences of other functions. 
Here, we have the following result.}

\begin{theorem}\label{thm3b}
There exists a positive constant $C>0$ such that
\begin{align*}
	& \|\bar{u}\|_{C([0,T];H)} + \|\bar{u}\|_{L^2(0,T;V)}
	+\|\bar{u}_\Gamma\|_{C([0,T];V_\Gamma')} 
	\\
	& \le C\left( \|\bar{u}_0\|_H + \|\bar{u}_{0\Gamma}\|_{V_\Gamma'} + \|\bar{f}\|_{L^2(0,T;V')} 
	+ \|\bar{f}_\Gamma\|_{L^2(0,T;H_\Gamma)}\right).
\end{align*}
\end{theorem}

\begin{proof}
We proceed exactly as in the proof of Theorem~\ref{thm2b}, arriving at an inequality 
very similar to~\eqref{sim} but without the term $\kappa \|\nabla _\Gamma \bar{u}_\Gamma\|_{H_\Gamma}^2$, that is,
\begin{align}
	& \frac{1}{2} \bigl\| \bar{u}(t) \bigr\|_{H}^2 
	+ \frac{1}{2} \bigl\| \bar{u}_\Gamma(t) \bigr\|_{V_{\Gamma,0}'}^2
	+ \frac{1}{2}\int_0^t \|\bar{u}\|_V^2 \ds
	\notag \\
	& \le \frac{1}{2} \|\bar{u}_0\|_{H}^2
	+ \frac{1}{2} \|\bar{u}_{0\Gamma}\|_{V_{\Gamma,0}'}^2 +
	(1+L) 
	\int_0^t
	\| \bar{u} \|^2_H \ds
	+ 
	L_\Gamma 
	\int_0^t
	\| \bar{u}_\Gamma \|_{H_\Gamma} ^2 \ds
	\notag \\
	& \quad {}
	+ \frac{1}{2} \int_0^t \|\bar{f}\|_{V'}^2 \ds 
	+ 
	\frac{1}{2}
	\int_0^t 
	\| \bar{f}_\Gamma\|_{H_\Gamma}^2 \ds + 
	\frac{1}{2} \int_0^t \| \bar{u}_\Gamma \|_{H_\Gamma}^2 \ds  
\label{pier19}	
\end{align}
for all $t \in [0,T]$. Note that in the proof of Theorem~\ref{thm2b}
the term $(\bar{f}_\Gamma, \bar{u}_\Gamma)_{H_\Gamma}$ was 
controlled by	$\| \bar{f}_\Gamma\|_{V_\Gamma'} \| \bar{u}_\Gamma \|_{V_\Gamma}$, while now we have to bound it by $\| \bar{f}_\Gamma\|_{H_\Gamma} \| \bar{u}_\Gamma \|_{H_\Gamma}$ in order to arrive at ~\eqref{pier19}. Then, we can use the compactness inequality~\eqref{comp2} and the Gronwall lemma as in the previous proofs. 
\end{proof}

\section{Error estimates}
In this section, we present the error estimates. For simplicity, we fix the same data 
$ u_0, u_{0\Gamma}, f, f_\Gamma $ for all problems under consideration. The same symbol 
$ \bar{u} $ is used throughout the section; however, its meaning may vary between subsections, 
as was the case in the previous section. Therefore, care should be taken to interpret 
$ \bar{u} $ appropriately in each context.
\smallskip

From this point on, we denote the convolution product of two time functions $ a $ and $ b $~by
\[
(a * b)(t) := \int_0^t a(t - s)\, b(s)\, \mathrm{d}s.
\]

\subsection{Error estimate for $({\rm P})_{\varepsilon}$}
\pier{In this subsection, we define $ \bar{u} = u_{\varepsilon, \kappa} - u_\varepsilon $, representing the 
difference between the solution component $ u_{\varepsilon, \kappa} $ of the original problem 
$(\mathrm{P})_{\takeshi{\varepsilon\kappa}}$ and the corresponding component $ u_\varepsilon $ of $(\mathrm{P})_\varepsilon$, 
as established in Theorem~\ref{thm1}. Analogously, we use the notation 
$ \bar{\mu} := \mu_{\varepsilon, \kappa} - \mu_\varepsilon $, and similarly for other functions.}

\pier{To derive the error estimate, the additional assumption \textnormal{(A6)} becomes essential. 
In particular, this assumption ensures further regularity for the unknown function on the boundary. 
See Corollary~\ref{cor1} for further details.}

\begin{theorem}\label{thm1c} 
Assume {\rm (A1)}, {\rm (A3)}--{\rm (A6)}.~\pier{Then
there exists a positive constants $C$, independent of $\kappa\in (0,1]$,} such that
\begin{gather}
%	\|\bar{u}\|_{L^\infty(0,T;H)}
%	\le C \sqrt{\kappa},
%	\label{error0}\\[2mm]
	\pier{\|\bar{u}\|_{L^\infty(0,T;H)\cap L^2(0,T;V)}}
	+ \sqrt{\varepsilon}  
	\bigl\| \nabla (1*\bar{\mu}) \bigr\|_{L^\infty(0,T;H)} 
	+ 
	\bigl\| \nabla _\Gamma (1 * \bar{\mu}_{\Gamma}) \bigr\|_{L^\infty(0,T;H_\Gamma)}
	\notag \\
	{} + \sqrt{\kappa}
	\|\nabla _\Gamma u_{\Gamma,\takeshi{\varepsilon, \kappa}}\|_{L^2(0,T;H_\Gamma)} 
	\le C \sqrt{\kappa}, 
	\label{errora}\\[2mm]
	\|\bar{u}_{\Gamma}\|_{L^\infty(0,T;V_\Gamma')} 
	\le C  \sqrt{\kappa}.
	\label{errorb}
\end{gather}
\end{theorem}

\begin{proof}
\pier{In view of~\eqref{e1}, by} subtracting \eqref{e6weak} for $({\rm P})_\varepsilon$ from \pier{the the variational equality  \eqref{ek1and6} for $({\rm P})_{\varepsilon\kappa}$, we have that} 
\begin{equation*}
	\langle \partial_t \bar{u}_{\Gamma},z_\Gamma \rangle_{V_\Gamma',V_\Gamma}
	+ \varepsilon \int_\Omega 
	\nabla \bar{\mu} \cdot \nabla z \dx
	+ \int_\Gamma \nabla _\Gamma \bar{\mu}_{\Gamma} \cdot \nabla_\Gamma z_\Gamma \dg
	= 0
\end{equation*}
for all $(z,z_\Gamma) \in \boldsymbol{V}$, a.e.\ in $(0,T)$. 
Now, integrating this equality \pier{with respect to time, we deduce that}  
\begin{equation}
	\bigl 
	\langle \bar{u}_{\Gamma}(t),z_\Gamma 
	\bigr\rangle_{V_\Gamma',V_\Gamma}
	 = 
	 - \varepsilon \int_\Omega 
	\nabla (1* \bar{\mu})(t) \cdot \nabla z \dx
	- \int_\Gamma \nabla _\Gamma (1 * \bar{\mu}_{\Gamma})(t) \cdot \nabla_\Gamma z_\Gamma \dg
	\label{error2}
\end{equation}
for all $(z,z_\Gamma) \in \boldsymbol{V}$ and all $t \in [0,T]$, where 
we used the same initial value for $u_{\Gamma, \varepsilon\kappa}$ and $u_\Gamma$. 
Then, \pier{in~\eqref{error2} we choose $(z,z_\Gamma)=(\bar{\mu}, \bar{\mu}_\Gamma)$ and obtain}
\begin{align}
	\int_\Gamma \bar{u}_{\Gamma}(t) \bar{\mu}_\Gamma(t) \dg
	 & = \bigl 
	\langle \bar{u}_{\Gamma}(t),\bar{\mu}_\Gamma(t)
	\bigr\rangle_{V_\Gamma',V_\Gamma} 
	\notag \\
	& =
	- \frac{\varepsilon}{2} \frac{d}{dt} \int_\Omega 
	\bigl| \nabla (1* \bar{\mu})(t) \bigr| ^2 \dx
	- \frac{1}{2} \frac{d}{dt} \int_\Gamma 
	\bigl| \nabla _\Gamma (1 * \bar{\mu}_{\Gamma})(t) \bigr|^2
	 \dg.
	 \label{replace1}
\end{align}
On the other hand, take the difference between the weak form of the first equation in 
\eqref{ek3} \pier{coupled with the first equation in \eqref{ek7} for $({\rm P})_{\varepsilon\kappa}$
 and the one in \eqref{e3} with \eqref{e7} for $({\rm P})_{\varepsilon}$. Then, we arrive at}
\begin{align*}
	& \int_\Omega \partial_t \bar{u} z \dx 
	+ 
	\int_\Omega \nabla \bar{u} \cdot \nabla z\dx 
	+
	\int_\Omega \bar{\xi} z \dx 
	+
	\langle \bar{\xi}_\Gamma, z_\Gamma \rangle_{Z_\Gamma',Z_\Gamma}
	- \int_\Gamma \bar{\mu}_\Gamma
	z_\Gamma \dg
	\notag \\
	\quad 
	& = 
	- \kappa \int_\Gamma \nabla _\Gamma u_{\Gamma,\takeshi{\varepsilon, \kappa}} \cdot \nabla _\Gamma
	z_\Gamma \dg
	- 
	\int_\Omega
	\bigl( 
	\pi(u_{\takeshi{\varepsilon, \kappa}})-\pi(u_\varepsilon) \bigr) z \dx 
%	\notag \\
%	& \quad {}
	-
	\int_\Gamma 
	\bigl( \pi_\Gamma(u_{\Gamma,\takeshi{\varepsilon, \kappa}}) 
	- \pi_\Gamma (u_{\Gamma,\varepsilon}) \bigr) z_\Gamma \dg
\end{align*}
for all \pier{$(z,z_\Gamma) \in \boldsymbol{V}$}, a.e.\ in $(0,T)$. 
Now, \pier{we take $(z,z_\Gamma) = (\bar{u}, \bar{u}_\Gamma)$, which is possible on account of Corollary ~\ref{cor1},
and rewrite the term $\int_\Gamma \bar{\mu}_\Gamma \bar{u}_\Gamma \dg$ on account of~\eqref{replace1}.}
Using the monotonicity of $\beta$ and $\beta_\Gamma$, 
integrating the resultant \pier{with respect to time, and 
adding $\int_0^t \| \bar{u} \|_{H}^2 \ds$ to both sides,}
we infer that 
\begin{align*}
	& \frac{1}{2} \bigl\| \bar{u}(t) \bigr\|_H^2 
	+ 
	\int_0^t \| \bar{u} \|_V^2 \ds 
	 + \frac{\varepsilon}{2} 
	\bigl\| \nabla (1*\bar{\mu})(t) \bigr\|_H ^2 
	\notag \\
	\quad 
	& \quad {}
	+ \frac{1}{2} 
	\bigl\| \nabla _\Gamma (1 * \bar{\mu}_{\Gamma})(t) \bigr\|_{H_\Gamma}^2
	+ \kappa \int_0^t\| \nabla _\Gamma u_{\Gamma,\takeshi{\varepsilon, \kappa}}\|_{H_\Gamma}^2  \ds	
	\notag \\
	\quad 
	& \le
	- \kappa \int_0^t 
	( 
	\nabla _\Gamma u_{\Gamma,\takeshi{\varepsilon, \kappa}}, \nabla _\Gamma
	u_{\Gamma,\varepsilon} )_{H_\Gamma} \ds + (1+L) \int_0^t \|\bar{u}\|_H^2 \ds
	+ L_\Gamma \int_0^t \|\bar{u}_\Gamma\|_{H_\Gamma}^2 \ds
	\notag \\
	& \le \frac{\kappa}{2} \int_0^t 
	\|
	\nabla _\Gamma u_{\Gamma,\takeshi{\varepsilon, \kappa}}
	\|_{H_\Gamma}^2 \ds
	+ 
	\frac{\kappa}{2} \int_0^t 
	\|
	\nabla _\Gamma
	u_{\Gamma,\varepsilon} \|_{H_\Gamma}^2 \ds
		+ \frac{1}{2} \int_0^t \| \bar{u}\|_{V}^2 \ds + C \int_0^t \|\bar{u}\|_H^2 \ds
\end{align*}
for all $t \in [0,T]$, 
where \pier{\eqref{comp2} has been used to estimate the term with factor $L_\Gamma$.}
The point of emphasis is \pier{now  the uniform estimate (cf.~\eqref{statre1})}% and \eqref{LM2})
\begin{align*}
	\int_0^t \| \nabla_\Gamma u_{\Gamma,\varepsilon} \|_{H_\Gamma}^2 \ds 
	& \le 
	C_{\rm tr} \int_0^t \|  u_{\varepsilon} \|_{H^{3/2}(\Omega)}^2 \ds 
%	\\
%	& \le C_{\rm tr}2 C_{\rm e}^2 \left\{ \int_0^t \| \Delta u_{\varepsilon} \|_{H}^2 \ds
%	+ \int_0^t \| \partial _{\boldsymbol{\nu}} u_\varepsilon \|_{H_\Gamma}^2 \ds \right\}
\end{align*}
%obtained from \eqref{LM2} with 
%\begin{align*}
%	\| \Delta u_\varepsilon \|_{L^2(0,T;H)} & \le \liminf _{k \to +\infty} 
%	\| \Delta u_{\varepsilon \kappa_k} \|_{L^2(0,T;H)} \\
%	& \le M_5,\\
%	\| \partial _{\boldsymbol{\nu}} u_\varepsilon \|_{L^2(0,T;H_\Gamma)} & \le \liminf _{k \to +\infty} 
%	\| \partial _{\boldsymbol{\nu}} u_{\varepsilon \kappa_k} -\kappa_k \Delta_\Gamma u_{\Gamma,\varepsilon \kappa_k} \|_{L^2(0,T;H_\Gamma)} \\
%	& \le M_7,
%\end{align*}
%see Lemma~\ref{L5} and \eqref{uniest01} 
\pier{which allows us to conclude for the uniform boundedness as in Corollary \ref{cor1} under the additional assumption {\rm (A6)}.} 
Therefore, applying the Gronwall inequality, we derive~\pier{\eqref{errora}.
Finally, from \eqref{errora} and a comparison in \eqref{error2} we easily infer that~\eqref{errorb} holds for all $\varepsilon \in (0,1]$.}
\end{proof}

\subsection{Error estimate for $({\rm P})_{\kappa}$}
\pier{In this subsection, we set $ \bar{u} = u_{\varepsilon, \kappa} - u_\kappa $, representing the difference 
between the first component $ u_{\varepsilon, \kappa} $ of the solution to the original problem 
$(\mathrm{P})_{\takeshi{\varepsilon\kappa}}$ and the corresponding component $ u_\kappa $ of the solution to 
$(\mathrm{P})_\kappa$, as established in Theorem~\ref{thm2}. The same bar notation is used analogously 
for the other functions.}

\pier{For the validity of the following theorem, the assumption \textnormal{(A6)} is not required.}

\begin{theorem}\label{thm2c} 
Assume {\rm (A1)}--{\rm (A5)}. \pier{Then there exists a positive constant $C$, independent of $\varepsilon\in (0,1]$, such that}
\begin{gather}
	\pier{ \|\bar{u}\|_{L^\infty(0,T;H) \cap L^2(0,T;V)} 
	+ \sqrt{\kappa}
	\|\nabla _\Gamma \bar{u}_{\Gamma}\|_{L^2(0,T;H_\Gamma)} 	
	+
	\sqrt{\varepsilon}  
	\bigl\| \nabla (1*\mu_{\takeshi{\varepsilon, \kappa}}) \bigr\|_{L^\infty(0,T;H)} }
	\notag
	\\
	\pier{{}+ \bigl\| \nabla _\Gamma (1 * \bar{\mu}_{\Gamma}) \bigr\|_{L^\infty(0,T;H_\Gamma)}
	\le C \sqrt{\varepsilon}, }
	\label{error2a}
	\\[2mm]
	\pier{\|\bar{u}_{\Gamma}\|_{L^\infty(0,T;V_\Gamma')} 
	\le C \sqrt{\varepsilon}.}
	\label{error2b}
\end{gather}%
\end{theorem}
\begin{proof}
The proof is similar to one of Theorem \ref{thm1c}. 
Subtracting \eqref{main2-1} for $({\rm P})_\kappa$ from \pier{the the variational equality~\eqref{ek1and6} for $({\rm P})_{\varepsilon\kappa}$, 
we obtain}
\begin{equation*}
	\langle \partial_t \bar{u}_{\Gamma},z_\Gamma \rangle_{V_\Gamma',V_\Gamma}
	+ \int_\Gamma \nabla _\Gamma \bar{\mu}_{\Gamma} \cdot \nabla_\Gamma z_\Gamma \dg
	= - \varepsilon \int_\Omega 
	\nabla \mu_{\takeshi{\varepsilon, \kappa}} \cdot \nabla z \dx
\end{equation*}
for all \pier{$(z,z_\Gamma) \in \boldsymbol{V}$,} a.e.\ in $(0,T)$, where \pier{particular attention must be paid to the last term.}
Now, integrating this equality over $[0,t]$, we have that 
\begin{equation}
	\bigl 
	\langle \bar{u}_{\Gamma}(t),z_\Gamma 
	\bigr\rangle_{V_\Gamma',V_\Gamma}
	 = 
	 - \varepsilon \int_\Omega 
	\nabla (1*\mu_{\takeshi{\varepsilon, \kappa}})(t) \cdot \nabla z \dx
	- \int_\Gamma \nabla _\Gamma (1 * \bar{\mu}_{\Gamma})(t) \cdot \nabla_\Gamma z_\Gamma \dg
	\label{error-ek}
\end{equation}
for all \pier{$(z,z_\Gamma) \in \boldsymbol{V}$} and all $t \in [0,T]$. 
Then, in the above we want to choose $\pier{(z,z_\Gamma)}=(\mu_{\takeshi{\varepsilon, \kappa}}-{\mathcal H} \mu_{\Gamma,\kappa}, \bar{\mu}_\Gamma)$, where ${\mathcal H}:Z_\Gamma \to V$ is the harmonic extension defined by 
\begin{equation}
	\begin{cases}
	\displaystyle
	\varepsilon \int _\Omega \nabla {\mathcal H} v_\Gamma \cdot \nabla z \dx = 0
	& \quad {\rm for~all~} z \in H_0^1(\Omega), \\
	\bigl( {\mathcal H} v_\Gamma \bigr)|_\Gamma \pier{{}= v_\Gamma}
	& \quad {\rm a.e.\ on~}\Gamma
	\end{cases}
	\label{harmonic}
\end{equation}
for all $v_\Gamma \in Z_\Gamma$. 
Here, from \pier{\eqref{e1}} we see that 
\begin{equation*}
	\varepsilon \int_\Omega \nabla \mu_{\takeshi{\varepsilon, \kappa}} (t) \cdot \nabla z \dx 
	= 0
	\quad {\rm for~all~} z \in H^1_0(\Omega).
\end{equation*}
Moreover, using \pier{\eqref{e2} we deduce  that} ${\mathcal H} \mu_{\Gamma,\takeshi{\varepsilon, \kappa}}=\mu_{\takeshi{\varepsilon, \kappa}}$, that is, $\mu_{\takeshi{\varepsilon, \kappa}}-{\mathcal H} \mu_{\Gamma,\kappa}={\mathcal H} \bar{\mu}_\Gamma$. 
Hence, \pier{letting $\pier{(z,z_\Gamma)}=(\mu_{\takeshi{\varepsilon, \kappa}}-{\mathcal H} \mu_{\Gamma,\kappa}, \bar{\mu}_\Gamma)=({\mathcal H} \bar{\mu}_\Gamma, \bar{\mu}_\Gamma) $ in \eqref{error-ek}, we infer that}
\begin{align}
	&\int_\Gamma \bar{u}_{\Gamma}(t) \bar{\mu}_\Gamma(t) \dg
	 = \bigl 
	\langle \bar{u}_{\Gamma}(t),\bar{\mu}_\Gamma(t)
	\bigr\rangle_{V_\Gamma',V_\Gamma} 
	\notag \\
	&\quad {} =
	 - \frac{\varepsilon}{2} \frac{d}{dt} \int_\Omega 
	\bigl| \nabla (1*\mu_{\takeshi{\varepsilon, \kappa}})(t) \bigr| ^2 \dx
	+ \varepsilon \int_\Omega 
	\nabla (1*\mu_{\takeshi{\varepsilon, \kappa}})(t) \cdot \nabla ({\mathcal H} \pier{\mu_{\Gamma,\kappa}})(t) \dx
	\notag \\
	&\quad\quad\   {} 
	- \frac{1}{2} \frac{d}{dt} \int_\Gamma 
	\bigl| \nabla _\Gamma (1 * \bar{\mu}_{\Gamma})(t) \bigr|^2
	 \dg
	.
	\label{replace2}
\end{align}
On the other hand, take the difference between the weak form of the first equation in 
\eqref{ek3} \pier{complemented by} the first equation in \eqref{ek7} for $({\rm P})_{\varepsilon\kappa}$
 and \pier{the one in \eqref{e3} with \eqref{k7}} for $({\rm P})_{\kappa}$. \pier{As the coefficient $\kappa> 0$ is present there, we  
 test by $(z,z_\Gamma) = (\bar{u}, \bar{u}_\Gamma)$ and use the monotonicity of $\beta$ and $\beta_\Gamma$ 
 to infer that}
\begin{align*}
	& \frac{1}{2} \frac{d}{dt} \int_\Omega |\bar{u}|^2 \dx 
	+ \int_\Omega | \nabla \bar{u} | ^2 \dx 
	+
	\kappa \int_\Gamma | \nabla _\Gamma \bar{u}_{\Gamma}|^2 \dg
	- \int_\Gamma \bar{\mu}_\Gamma
	\bar{u}_\Gamma \dg
	\notag \\
	\quad 
	& \le 
	- 
	\int_\Omega
	\bigl( 
	\pi(u_{\takeshi{\varepsilon, \kappa}})-\pi(u_\kappa) \bigr) \bar{u} \dx
	-
	\int_\Gamma 
	\bigl( \pi_\Gamma(u_{\Gamma,\takeshi{\varepsilon, \kappa}}) 
	- \pi_\Gamma (u_{\Gamma,\kappa}) \bigr) \bar{u}_\Gamma \dg
\end{align*}
\pier{a.e.\ in $(0,T)$.}
Then, in the above, replacing the term $\int_\Gamma \bar{\mu}_\Gamma \bar{u}_\Gamma \dg$ by \pier{the expression in~\eqref{replace2}, 
integrating the resultant over $(0,t)$ and adding $\int_0^t \| \bar{u} \|_H^2 \ds$ to both sides,
we deduce that}
\begin{align}
	& \frac{1}{2} \bigl\|\bar{u} (t) \bigr\|_H^2 
	+ 
	\int_0^t \| \bar{u} \|_V ^2 \ds 
	+ \kappa \int_0^t\| \nabla _\Gamma \bar{u}_{\Gamma}\|_{H_\Gamma}^2 \ds
	 \notag \\
	& \quad {}
	+ \frac{\varepsilon}{2} 
	\bigl\| \nabla (1*\mu_{\takeshi{\varepsilon, \kappa}})(t) \bigr\|_H ^2 
	+ \frac{1}{2} 
	\bigl\| \nabla _\Gamma (1 * \bar{\mu}_{\Gamma})(t) \bigr\|_{H_\Gamma}^2
	\notag \\
	\quad 
	& \le -  \varepsilon \int_0^t \bigl(
	\nabla (1*\mu_{\takeshi{\varepsilon, \kappa}}), \nabla ({\mathcal H} \mu_{\Gamma,\kappa}) \bigr)_H \ds
	 + (1+L) \int_0^t \|\bar{u}\|_H^2 \ds
	+ L_\Gamma \int_0^t \|\bar{u}_\Gamma\|_{H_\Gamma}^2 \ds
	\notag \\
	& \le \frac{\varepsilon}{2} \int_0^t \bigl\|
	\nabla (1*\mu_{\takeshi{\varepsilon, \kappa}})\bigr\|_H^2 \ds
	+
	\frac{\varepsilon}{2} \int_0^t \bigl\|
	\nabla ({\mathcal H} \mu_{\Gamma,\kappa}) \bigr\|_H^2 \ds
	\notag \\
	& \quad{}	+ \pier{C}\int_0^t \|\bar{u}\|_H^2 \ds
	+ \frac{1}{2} \int_0^t \| \bar{u}\|_{V}^2 \ds
	\label{pier20}
\end{align}
for all $t \in [0,T]$, where we used \eqref{comp2} again. Now, 
\pier{employing the recovering operator ${\mathcal R}:Z_\Gamma \to V$ specified by~\eqref{recov},}
we see that $z := {\mathcal H} \mu_{\Gamma,\kappa}-{\mathcal R} \mu_{\Gamma,\kappa} \in H_0^1(\Omega)$.
Therefore, taking $z := {\mathcal H} \mu_{\Gamma,\kappa}-{\mathcal R} \mu_{\Gamma,\kappa}$ in \eqref{harmonic}
\pier{and recalling~\eqref{recover}}, we have that
\begin{align}
	\varepsilon \int_\Omega |\nabla {\mathcal H} \mu_{\Gamma,\kappa}|^2\ dx 
	& = \varepsilon \int_\Omega \nabla {\mathcal H} \mu_{\Gamma,\kappa}\cdot \nabla {\mathcal R} \mu_{\Gamma, \kappa} \dx \notag \\
	& \le \frac{\varepsilon}{2} \int_\Omega |\nabla {\mathcal H} \mu_{\Gamma,\kappa}|^2 \dx
	+ \frac{\varepsilon}{2} \int_\Omega |\nabla {\mathcal R} \mu_{\Gamma,\kappa}|^2 \dx \notag \\
		& \le 
		\frac{\varepsilon}{2} \int_\Omega |\nabla {\mathcal H} \mu_{\Gamma,\kappa}|^2 \dx
		+ \frac{\varepsilon }{2}C_{\mathcal R}^2 \|\mu_{\Gamma,\kappa} \|_{Z_\Gamma}^2
		\label{idea}
\end{align}
a.e in $(0,T)$. \pier{In view of Lemma~\ref{L4}, from \eqref{pier20} and \eqref{idea} it follows that
\begin{align*}
	& \frac{1}{2} \bigl\|\bar{u} (t) \bigr\|_H^2 + 
	\frac{1}{2}\int_0^t \| \bar{u} \|_V ^2 \ds 
	+ \kappa \int_0^t\| \nabla _\Gamma \bar{u}_{\Gamma}\|_{H_\Gamma}^2 \ds
          \notag \\
	& \quad {}
	+ \frac{\varepsilon}{2} 
	\bigl\| \nabla (1*\mu_{\takeshi{\varepsilon, \kappa}})(t) \bigr\|_H ^2 
	+ \frac{1}{2} 
	\bigl\| \nabla _\Gamma (1 * \bar{\mu}_{\Gamma})(t) \bigr\|_{H_\Gamma}^2
	\notag \\
	& \le 
	C \int_0^t \|\bar{u}\|_H^2 \ds
	+
	\frac{\varepsilon}{2} \int_0^t \bigl\|
	\nabla (1*\mu_{\takeshi{\varepsilon, \kappa}})\bigr\|_H^2 \ds
	+
	{\varepsilon}\,\pier{C} M_4^2 
\end{align*}
for all $t \in [0,T]$.} Therefore, applying the Gronwall inequality, we \pier{derive \eqref{error2a}.
Finally, from \eqref{error2a} and a comparison of terms in \eqref{error-ek} we arrive at}~\eqref{error2b}. 
\end{proof}

\subsection{Error estimate for $({\rm P})$}
In this subsection, we set $\bar{u}:=u_{\takeshi{\varepsilon, \kappa}}-u$ as the difference between the solution $u_{\takeshi{\varepsilon, \kappa}}$ of the starting problem $({\rm P})_{\varepsilon\kappa}$ and the solution $u$ of $({\rm P})$ obtained in Theorem \ref{thm3}. As \pier{in Theorem~\ref{thm1c}}, we need the additional regularity for the unknown function on the boundary, obtained in 
Corollary~\ref{cor3} under the assumption~{\rm (A6)}.

\begin{theorem}\label{thm3c} 
Assume {\rm (A1)}, {\rm (A3)}--{\rm (A6)}. ~\pier{Then
there exists a positive constants $C$, independent of $\varepsilon, \kappa \in (0,1]$,} such that
\begin{gather}
\pier{\|\bar{u}\|_{L^\infty(0,T;H)\cap \takeshi{L^2}(0,T;V)} + 
	\sqrt{\varepsilon}  
	\bigl\| \nabla (1*\mu_{\takeshi{\varepsilon, \kappa}}) \bigr\|_{L^\infty(0,T;H)} 
	+
	\bigl\| \nabla _\Gamma (1 * \bar{\mu}_{\Gamma}) \bigr\|_{L^\infty(0,T;H_\Gamma)}}
	\notag\\
	{}+ \sqrt{\kappa}
	\|\nabla _\Gamma u_{\Gamma,\takeshi{\varepsilon, \kappa}}\|_{L^2(0,T;H_\Gamma)} 
	\le C (\sqrt{\varepsilon}+\sqrt{\kappa}), \label{error3a}\\[2mm]
	\|\bar{u}_{\Gamma}\|_{L^\infty(0,T;V_\Gamma')} 
	\le C (\sqrt{\varepsilon}+\sqrt{\kappa}). \label{error3b}
\end{gather}
\end{theorem}

\begin{proof}
Take the difference between the weak form of the first equation in 
\eqref{ek3} \pier{supplied with the first equation in \eqref{ek7} for $({\rm P})_{\varepsilon\kappa}$
 and the corresponding ones for $(\rm P)$ (see~\eqref{e3} and~\eqref{e7}). Then we test by $\bar{u}$
 (which is possible thanks to Corollary~\ref{cor3}) and obtain
\begin{align*}
	&  \frac{1}{2} \frac{d}{dt} \int_\Omega |\bar{u}|^2 \dx 
	+ \int_\Omega | \nabla \bar{u} | ^2 \dx 
	+
	\int_\Omega \bar{\xi} \bar{u} \dx 
	+
	\langle \bar{\xi}_\Gamma, \bar{u}_\Gamma \rangle_{Z_\Gamma',Z_\Gamma}
	- \int_\Gamma \bar{\mu}_\Gamma
	\bar{u}_\Gamma \dg
	\notag \\
	\quad 
	& {}= {}
	- \kappa \int_\Gamma \nabla _\Gamma u_{\Gamma,\takeshi{\varepsilon, \kappa}} \cdot \nabla _\Gamma
	\bar{u}_\Gamma \dg
	- 
	\int_\Omega
	\bigl( 
	\pi(u_{\takeshi{\varepsilon, \kappa}})-\pi(u) \bigr) \bar{u} \dx
	\notag \\
	& \quad\ {}
	-
	\int_\Gamma 
	\bigl( \pi_\Gamma(u_{\Gamma,\takeshi{\varepsilon, \kappa}}) 
	- \pi_\Gamma (u_\Gamma) \bigr) \bar{u}_\Gamma \dg
\end{align*}
a.e.\ in $(0,T)$. 
This is exactly the same type of result as in Theorem~\ref{thm1c}. On the other hand, 
subtracting \eqref{main3-2} for $({\rm P})$ from \eqref{weak1} for $({\rm P})_{\varepsilon\kappa}$
and arguing as in the proof of Theorem~\ref{thm2c}, we deduce~\eqref{replace2}. Hence,
replacing the term $\int_\Gamma \bar{\mu}_\Gamma \bar{u}_\Gamma \dg$ in the above 
by~\eqref{replace2}, integrating the resultant over $(0,t)$ with respect to time, and adding $\int_0^t \| \bar{u}\|_H ^2 \ds$, we infer that} 
\begin{align*}
	& \frac{1}{2} \bigl\| \bar{u}(t) \bigr\|_H^2 
	+ 
	\int_0^t \| \bar{u} \|_V ^2 \ds 
	 + \frac{\varepsilon}{2} 
	\bigl\| \nabla (1*\mu_{\takeshi{\varepsilon, \kappa}})(t) \bigr\|_H ^2 
	\notag \\
	& \quad {}
	+ \frac{1}{2} 
	\bigl\| \nabla _\Gamma (1 * \bar{\mu}_{\Gamma})(t) \bigr\|_{H_\Gamma}^2
	+ \kappa \int_0^t\| \nabla _\Gamma u_{\Gamma,\takeshi{\varepsilon, \kappa}}\|_{H_\Gamma}^2 \ds 
	\notag \\
	& \le -  \varepsilon \int_0^t \bigl(
	\nabla (1*\mu_{\takeshi{\varepsilon, \kappa}}), \nabla ({\mathcal H} \mu_\Gamma) \bigr)_H \ds
	- \kappa \int_0^t 
	( 
	\nabla _\Gamma u_{\Gamma,\takeshi{\varepsilon, \kappa}}, \nabla _\Gamma
	u_\Gamma )_{H_\Gamma} \ds
	\notag \\
	& {} \quad + (1+L) \int_0^t \|\bar{u}\|_H^2 \ds
	+ L_\Gamma \int_0^t \|\bar{u}_\Gamma\|_{H_\Gamma}^2 \ds
	\notag \\
	& \le \frac{\varepsilon}{2} \int_0^t \bigl\|
	\nabla (1*\mu_{\takeshi{\varepsilon, \kappa}})\bigr\|_H^2 \ds
	+
	\frac{\varepsilon}{2} C_{\mathcal R}^2
	\int_0^t  \|\mu_{\Gamma} \|_{Z_\Gamma}^2 ds 
	+ \frac{\kappa}{2} \int_0^t 
	\|
	\nabla _\Gamma u_{\Gamma,\takeshi{\varepsilon, \kappa}}
	\|_{H_\Gamma}^2 \ds
		\notag \\
	& {} \quad
	+ 
	\frac{\kappa}{2} \int_0^t 
	\|
	\nabla _\Gamma
	u_\Gamma \|_{H_\Gamma}^2 \ds
	+ \pier{C} \int_0^t \|\bar{u}\|_H^2 \ds
	+ \frac{1}{2} \int_0^t \|\bar{u}\|_{V}^2 \ds
\end{align*}
for all $t \in [0,T]$, where we used the same \pier{bound}~\eqref{idea} as in the proof of Theorem \ref{thm3b}.  
The point of emphasis is the regularity of $u_\Gamma \in L^2(0,T;V_\Gamma)$ which is obtained 
in Corollary~\ref{cor3} under the additional assumption~{\rm (A6)}\pier{. In fact,} we deduce  
\begin{align*}
	&\frac{1}{2} \bigl\| \bar{u}(t) \bigr\|_H^2 
	+ \frac{1}{2}
	\int_0^t \| \bar{u} \|_V ^2 \ds 
	 + \frac{\varepsilon}{2} 
	\bigl\| \nabla (1*\mu_{\takeshi{\varepsilon, \kappa}})(t) \bigr\|_H ^2 
	\notag \\
	& \quad \pier{{}
	+ \frac{1}{2} 
	\bigl\| \nabla _\Gamma (1 * \bar{\mu}_{\Gamma})(t) \bigr\|_{H_\Gamma}^2
	+ \frac{\kappa}{2} \int_0^t\| \nabla _\Gamma u_{\Gamma,\takeshi{\varepsilon, \kappa}}\|_{H_\Gamma}^2 \ds }
	\notag \\
	& \le 
	{}\pier{C} \int_0^t \|\bar{u}\|_H^2 \ds
	+
	\frac{\varepsilon}{2} \int_0^t \bigl\|
	\nabla (1*\mu_{\takeshi{\varepsilon, \kappa}})\bigr\|_H^2 \ds
%	\notag \\
%	& \quad 
+
	\frac{\varepsilon}{2} C_{\mathcal R}^2 \| \mu_\Gamma \|_{L^2(0,T;V_\Gamma)}^2
	+ \frac{\kappa}{2} \| u_\Gamma \|_{L^2(0,T;V_\Gamma)}^2 
\end{align*}
for all $t \in [0,T]$. 
\pier{Hence the Gronwall lemma allows us to conclude the proof of~\eqref{error3a}. Then \eqref{error3b} can be derived as before with the help of \eqref{error-ek}.} 
\end{proof}

\section*{Acknowledgments}
\pier{This work was supported by the Research Institute for Mathematical Sciences, an International Joint Usage/Research Center located at Kyoto University. 
In addition, P.C.\ acknowledges the support of the Next Generation EU Project No.\ P2022Z7ZAJ (\emph{A unitary mathematical framework for modelling muscular dystrophies}) and of the GNAMPA (Gruppo Nazionale per l'Analisi Matematica, la Probabilit\`a e le loro Applicazioni) of INdAM (Istituto Nazionale di Alta Matematica). 
T.F.\ also acknowledges support from the JSPS KAKENHI Grant-in-Aid for Scientific Research (C), Japan, Grant Number 21K03309.}

\end{document}